\documentclass[11pt]{article}
\usepackage{amsmath,amsthm,amsfonts,amssymb,latexsym}
\usepackage{hyperref}
\usepackage{comment}
\usepackage{enumerate}
\usepackage{color}
\usepackage[dvipsnames]{xcolor}

\usepackage[shortlabels]{enumitem}

\usepackage{xparse,etoolbox}

\usepackage{sectsty}
\allsectionsfont{\centering}

\usepackage{parskip}
\makeatletter
\def\thm@space@setup{%
  \thm@preskip=\parskip \thm@postskip=0pt
}
\makeatother

\usepackage[margin=1.2in]{geometry}

\begin{document}

\theoremstyle{plain}

\newtheorem{thm}{Theorem}[section]

\newtheorem{lem}[thm]{Lemma}
\newtheorem{Problem B}[thm]{Problem B}

\newtheorem{pro}[thm]{Proposition}
\newtheorem{conj}[thm]{Conjecture}
\newtheorem{cor}[thm]{Corollary}
\newtheorem{que}[thm]{Question}

\theoremstyle{definition}
\newtheorem{rem}[thm]{Remark}
\newtheorem{defi}[thm]{Definition}
\newtheorem{hyp}[thm]{Hypothesis}

\theoremstyle{plain}
\newtheorem*{thmA}{Theorem A}
\newtheorem*{thmB}{Theorem B}
\newtheorem*{corB}{Corollary B}
\newtheorem*{thmC}{Theorem C}
\newtheorem*{thmD}{Theorem D}
\newtheorem*{thmE}{Theorem E}
 
\newtheorem*{thmAcl}{Main Theorem$^{*}$}
\newtheorem*{thmBcl}{Theorem B$^{*}$}
\newcommand{\dd}{\mathrm{d}}

\theoremstyle{plain}
\newtheorem{theoA}{Theorem}

\theoremstyle{plain}
\newtheorem{conjA}[theoA]{Conjecture}

\theoremstyle{plain}
\newtheorem{condA}[theoA]{Condition}

\theoremstyle{plain}
\newtheorem{paraA}[theoA]{Parametrisation}

\theoremstyle{plain}
\newtheorem{corA}[theoA]{Corollary}

\renewcommand{\thetheoA}{\Alph{theoA}}

\renewcommand{\thecorA}{\Alph{corA}}

\newcommand{\wh}[1]{\widehat{#1}} 
\newcommand{\miquelcomment}{\textcolor{blue}}
\newcommand{\ncomment}{\textcolor{magenta}}

\newcommand{\Maxn}{\operatorname{Max_{\textbf{N}}}}
\newcommand{\Syl}{\operatorname{Syl}}
\newcommand{\Lin}{\operatorname{Lin}}
\newcommand{\U}{\mathbf{U}}
\newcommand{\nav}{\mathrm{Nav}}
\newcommand{\R}{\mathbf{R}}
\newcommand{\dl}{\operatorname{dl}}
\newcommand{\Con}{\operatorname{Con}}
\newcommand{\rdz}{\operatorname{rdz}}
\newcommand{\rdzo}{\operatorname{rdz}^{\circ}}
\newcommand{\cl}{\operatorname{cl}}
\newcommand{\Stab}{\operatorname{Stab}}
\newcommand{\Aut}{\operatorname{Aut}}
\newcommand{\Ker}{\operatorname{Ker}}
\newcommand{\InnDiag}{\operatorname{InnDiag}}
\newcommand{\fl}{\operatorname{fl}}
\newcommand{\Irr}{\operatorname{Irr}}
\newcommand{\FF}{\mathbb{F}}
\newcommand{\CL}{\mathfrak{Cl}}
\newcommand{\EE}{\mathbb{E}}
\newcommand{\Alp}{\mathrm{Alp}}
\newcommand{\Alpr}{\mathrm{Alp_r}}
\newcommand{\normal}{\trianglelefteq}
\newcommand{\sn}{\normal\normal}
\newcommand{\Bl}{\mathrm{Bl}}
\newcommand{\NN}{\mathbb{N}}
\newcommand{\N}{\mathbf{N}}
\newcommand{\bfC}{\mathbf{C}}
\newcommand{\bfO}{\mathbf{O}}
\newcommand{\bfF}{\mathbf{F}}
\def\GGG{{\mathcal G}}
\def\HHH{{\mathcal H}}
\def\HH{{\mathcal H}}
\def\irra#1#2{{\rm Irr}_{#1}(#2)}

\renewcommand{\labelenumi}{\upshape (\roman{enumi})}

\newcommand{\PSL}{\operatorname{PSL}}
\newcommand{\PSU}{\operatorname{PSU}}
\newcommand{\alt}{\operatorname{Alt}}

\providecommand{\V}{\mathrm{V}}
\providecommand{\E}{\mathrm{E}}
\providecommand{\ir}{\mathrm{Irm_{rv}}}
\providecommand{\Irrr}{\mathrm{Irm_{rv}}}
\providecommand{\re}{\mathrm{Re}}

\numberwithin{equation}{section}
\def\irrp#1{{\rm Irr}_{p'}(#1)}

\def\ibrrp#1{{\rm IBr}_{\Bbb R, p'}(#1)}
\def\C{{\mathbb C}}

\def\isoc{{\succeq_c}}

\def\isob{{\succeq_b}}

\newcommand{\wt}[1]{\widetilde{#1}} 

\def\Rad{{\rm Rad}}
\def\Rado{{\rm Rad}^{\circ}}

\def\o{{\bf O}}
\def\c{{\bf C}}
\def\n{{\bf N}}
\def\z{{\bf Z}}
\def\F{{\bf F}}
\def\P{{\mathcal{P}}}
\def\Q{{\mathcal{Q}}}
\def\R{{\mathcal{R}}}
\def\W{{\mathcal{W}}}
\def\D{{\mathcal{D}}}
\def\Wr{{\mathcal{W}_{\rm r}}}

\def\irr#1{{\rm Irr}(#1)}
\def\irrp#1{{\rm Irr}_{p^\prime}(#1)}
\def\irrq#1{{\rm Irr}_{q^\prime}(#1)}
\def \aut#1{{\rm Aut}(#1)}
\def\cent#1#2{{\bf C}_{#1}(#2)}
\def\norm#1#2{{\bf N}_{#1}(#2)}
\def\zent#1{{\bf Z}(#1)}
\def\syl#1#2{{\rm Syl}_#1(#2)}
\def\normal{\triangleleft\,}
\def\oh#1#2{{\bf O}_{#1}(#2)}
\def\Oh#1#2{{\bf O}^{#1}(#2)}
\def\det#1{{\rm det}(#1)}
\def\gal#1{{\rm Gal}(#1)}
\def\ker#1{{\rm ker}(#1)}
\def\normalm#1#2{{\bf N}_{#1}(#2)}
\def\alt#1{{\rm Alt}(#1)}
\def\iitem#1{\goodbreak\par\noindent{\bf #1}}
   \def \mod#1{\, {\rm mod} \, #1 \, }
\def\sbs{\subseteq}

\def\gc{{\bf GC}}
\def\ngc{{non-{\bf GC}}}
\def\ngcs{{non-{\bf GC}$^*$}}
\newcommand{\notd}{{\!\not{|}}}

\newcommand{\Z}{\mathbf{Z}}

\newcommand{\Out}{{\mathrm {Out}}}
\newcommand{\Mult}{{\mathrm {Mult}}}
\newcommand{\Inn}{{\mathrm {Inn}}}
\newcommand{\Fong}{{\mathrm{Fong}}}
\newcommand{\IBR}{{\mathrm {IBr}}}
\newcommand{\IBRL}{{\mathrm {IBr}}_{\ell}}
\newcommand{\IBRP}{{\mathrm {IBr}}_{p}}
\newcommand{\bl}{{\mathrm{bl}}}
\newcommand{\cd}{\mathrm{cd}}
\newcommand{\ord}{{\mathrm {ord}}}
\def\id{\mathop{\mathrm{ id}}\nolimits}
\renewcommand{\Im}{{\mathrm {Im}}}
\newcommand{\Ind}{{\mathrm {Ind}}}
\newcommand{\diag}{{\mathrm {diag}}}
\newcommand{\soc}{{\mathrm {soc}}}
\newcommand{\End}{{\mathrm {End}}}
\newcommand{\sol}{{\mathrm {sol}}}
\newcommand{\Hom}{{\mathrm {Hom}}}
\newcommand{\Mor}{{\mathrm {Mor}}}
\newcommand{\Mat}{{\mathrm {Mat}}}
\def\rank{\mathop{\mathrm{ rank}}\nolimits}
\newcommand{\Tr}{{\mathrm {Tr}}}
\newcommand{\tr}{{\mathrm {tr}}}
\newcommand{\Gal}{{\rm Gal}}
\newcommand{\Spec}{{\mathrm {Spec}}}
\newcommand{\ad}{{\mathrm {ad}}}
\newcommand{\Sym}{{\mathrm {Sym}}}
\newcommand{\Char}{{\mathrm {Char}}}
\newcommand{\pr}{{\mathrm {pr}}}
\newcommand{\rad}{{\mathrm {rad}}}
\newcommand{\abel}{{\mathrm {abel}}}
\newcommand{\PGL}{{\mathrm {PGL}}}
\newcommand{\PCSp}{{\mathrm {PCSp}}}
\newcommand{\PGU}{{\mathrm {PGU}}}
\newcommand{\codim}{{\mathrm {codim}}}
\newcommand{\ind}{{\mathrm {ind}}}
\newcommand{\Res}{{\mathrm {Res}}}
\newcommand{\Lie}{{\mathrm {Lie}}}
\newcommand{\Ext}{{\mathrm {Ext}}}
\newcommand{\EBr}{{\mathrm {EBr}}}
\newcommand{\Alt}{{\mathrm {Alt}}}
\newcommand{\AAA}{{\sf A}}
\newcommand{\SSS}{{\sf S}}
\newcommand{\DDD}{{\sf D}}
\newcommand{\QQQ}{{\sf Q}}
\newcommand{\CCC}{{\sf C}}
\newcommand{\SL}{{\mathrm {SL}}}
\newcommand{\Sp}{{\mathrm {Sp}}}
\newcommand{\PSp}{{\mathrm {PSp}}}
\newcommand{\SU}{{\mathrm {SU}}}
\newcommand{\GL}{{\mathrm {GL}}}
\newcommand{\GU}{{\mathrm {GU}}}
\newcommand{\Br}{{\mathrm{Br}}}
\newcommand{\Spin}{{\mathrm {Spin}}}
\newcommand{\CC}{{\mathbb C}}
\newcommand{\CB}{{\mathbf C}}
\newcommand{\RR}{{\mathbb R}}
\newcommand{\QQ}{{\mathbb Q}}
\newcommand{\ZZ}{{\mathbb Z}}
\newcommand{\bfN}{{\mathbf N}}
\newcommand{\bfZ}{{\mathbf Z}}
\newcommand{\PP}{{\mathbb P}}
\newcommand{\cG}{{\mathcal G}}
\newcommand{\cH}{{\mathcal H}}
\newcommand{\cQ}{{\mathcal Q}}
\newcommand{\GA}{{\mathfrak G}}
\newcommand{\cT}{{\mathcal T}}
\newcommand{\cL}{{\mathcal L}}
\newcommand{\IBr}{\mathrm{IBr}}
\newcommand{\cS}{{\mathcal S}}
\newcommand{\cR}{{\mathcal R}}
\newcommand{\GCD}{\GC^{*}}
\newcommand{\TCD}{\TC^{*}}
\newcommand{\FD}{F^{*}}
\newcommand{\GD}{G^{*}}
\newcommand{\HD}{H^{*}}
\newcommand{\GCF}{\GC^{F}}
\newcommand{\TCF}{\TC^{F}}
\newcommand{\PCF}{\PC^{F}}
\newcommand{\GCDF}{(\GC^{*})^{F^{*}}}
\newcommand{\RGTT}{R^{\GC}_{\TC}(\theta)}
\newcommand{\RGTA}{R^{\GC}_{\TC}(1)}
\newcommand{\Om}{\Omega}
\newcommand{\eps}{\epsilon}
\newcommand{\varep}{\varepsilon}
\newcommand{\dz}{\mathrm{dz}}
\newcommand{\dzo}{\mathrm{dz}^\circ}
\newcommand{\Co}{\mathcal{C}^\circ}
\newcommand{\al}{\alpha}

\newcommand{\chis}{\chi_{s}}
\newcommand{\sigmad}{\sigma^{*}}
\newcommand{\PA}{\boldsymbol{\alpha}}
\newcommand{\gam}{\gamma}
\newcommand{\lam}{\lambda}
\newcommand{\la}{\langle}
\newcommand{\genf}{F^*}
\newcommand{\ra}{\rangle}
\newcommand{\hs}{\hat{s}}
\newcommand{\htt}{\hat{t}}
\newcommand{\tG}{\hat G}
\newcommand{\St}{\mathsf {St}}
\newcommand{\bfs}{\boldsymbol{s}}
\newcommand{\bfl}{\boldsymbol{\lambda}}
\newcommand{\tn}{\hspace{0.5mm}^{t}\hspace*{-0.2mm}}
\newcommand{\ta}{\hspace{0.5mm}^{2}\hspace*{-0.2mm}}
\newcommand{\tb}{\hspace{0.5mm}^{3}\hspace*{-0.2mm}}
\def\skipa{\vspace{-1.5mm} & \vspace{-1.5mm} & \vspace{-1.5mm}\\}
\newcommand{\tw}[1]{{}^#1\!}
\renewcommand{\mod}{\bmod \,}

\marginparsep-0.5cm

\newcommand{\blocktheorem}[1]{%
  \csletcs{old#1}{#1}
  \csletcs{endold#1}{end#1}
  \RenewDocumentEnvironment{#1}{o}
    {\par\addvspace{1.5ex}
     \noindent\begin{minipage}{\textwidth}
     \IfNoValueTF{##1}
       {\csuse{old#1}}
       {\csuse{old#1}[##1]}}
    {\csuse{endold#1}
     \end{minipage}
     \par\addvspace{1.5ex}}
}

\blocktheorem{theoA}
\blocktheorem{conjA}

\makeatletter
\def\blfootnote{\gdef\@thefnmark{}\@footnotetext}
\makeatother

\title{{\bf{\huge The Alperin Weight Conjecture and the Glauberman correspondence via character triples}}
\author{J. Miquel Mart\'inez, Noelia Rizo, and Damiano Rossi}
\blfootnote{\emph{$2010$ Mathematical Subject Classification:} $20$C$20$ ($20$C$15$)
\\
\emph{Key words and phrases:} Alperin Weight Conjecture, Glauberman correspondence, block theory, isomorphisms of character triples
\\
\date{}
The first and second author are partially supported by the Spanish Ministerio de Ciencia e Innovaci\'on Grant PID2022-137612NB-I00 (funded by MCIN/AEI/ 10.13039/501100011033 and “ERDF A way of making Europe”) and by Generalitat Valenciana CIAICO/2021/163. The first author is also supported by a fellowship UV-INV-PREDOC20-1356056 from Universitat de Val\`encia and by the national project PRIN 2022- 2022PSTWLB - Group Theory and Applications - CUP B53D23009410006. The second author is supported by a CDEIGENT grant CIDEIG/2022/29 funded by Generalitat Valenciana. The third author is funded by the EPSRC grant EP/W$028794$/$1$ and by a Leibniz Fellowship of the Mathematisches Forschungsinstitut Oberwolfach. The authors also acknowledge support of the Institut Henri Poincar\'e (UAR 839 CNRS-Sorbonne Universit\'e), and LabEx CARMIN (ANR-10-LABX-59-01).
}}

\maketitle

\begin{abstract}
Recently, G. Navarro introduced a new conjecture that unifies the Alperin Weight Conjecture and the Glauberman correspondence into a single statement. In this paper, we reduce this problem to simple groups and prove it for several classes of groups and blocks. Our reduction can be divided into two steps. First, we show that by assuming the so-called \textit{Inductive (Blockwise) Alperin Weight Condition} for finite simple groups, we obtain an analogous statement for arbitrary finite groups, that is, an automorphism-equivariant version of the Alperin Weight Conjecture inducing isomorphisms of modular character triples. Then, we show that the latter implies Navarro's conjecture for each finite group. 
\end{abstract}

\section{Introduction}

The Alperin Weight Conjecture, introduced in \cite{Alp87}, provides a way to determine the number of isomorphism classes of irreducible modular representations of a finite group $G$ in terms of local data called weights. This is part of a series of statements known as the local-global counting conjectures that constitute some of the main research questions in representation theory of finite groups. More precisely, for a prime number $p$, we define a $p$-weight to be a pair $(Q,\psi)$ where $Q$ is a radical $p$-subgroup of $G$ and $\psi$ is a $p$-Brauer character of $\n_G(Q)$ whose deflation to $\n_G(Q)/Q$ belongs to a $p$-block of defect zero. Furthermore, given a $p$-block $B$ of $G$, we say that $(Q,\psi)$ is a $B$-weight if $\bl(\psi)^G=B$ and where $\bl(\psi)$ denotes the unique $p$-block of $\n_G(Q)$ to which $\psi$ belongs. The set $\Alp(B)$ of $B$-weights is equipped with an action of $G$ by conjugation whose corresponding set of orbits is denoted by $\Alp(B)/G$. Then, the blockwise version of the Alperin Weight Conjecture posits that
\[\left|\IBr(B)\right|=\left|\Alp(B)/G\right|\]
where $\IBr(B)$ is the set of irreducible $p$-Brauer characters belonging to the $p$-block $B$.

The Alperin Weight Conjecture, as well as the other local-global counting conjectures, is intimately connected with the existence of natural correspondences of characters and blocks. One of the most useful such statements is the Glauberman correspondence and its blockwise version, known as the Dade--Glauberman--Nagao correspondence. In its most basic form, this asserts that whenever a $p$-group $A$ acts via automorphisms on a group $G$ of order prime to $p$ there exists a canonical bijection
\[\Irr_A(G)\to \Irr(\c_G(A))\]
where we denote by $\Irr_A(G)$ the set of $A$-invariant irreducible characters of $G$. Furthermore, observe that in this case the above bijection is equivalent to the Brauer--Glauberman correspondence introduced in \cite[Conjecture A]{Nav-Spa-Tie}. We refer the reader to Section \ref{sec:DGN} and \cite[Section 4]{Nav-Tie11} for further information on the Dade--Glauberman--Nagao correspondence.

Recently, Navarro suggested in \cite{Nav17} a new surprising statement that unifies the Alperin Weight Conjecture and the Dade--Glauberman--Nagao correspondence into a single statement. Let $G\unlhd \Gamma$ be finite groups and consider a $p$-block $B$ of $G$. For every radical $p$-subgroup $Q$ of $\Gamma$, we denote by $\dz(\n_\Gamma(Q)/Q\mid B)$ the set of irreducible characters $\overline{\vartheta}$ of $\n_\Gamma(Q)/Q$ of $p$-defect zero such that $\bl(\vartheta)^\Gamma$ covers $B$ and where $\vartheta\in\irr{\n_G(Q)}$ corresponds to $\overline{\vartheta}$ via inflation of characters. Navarro's conjecture can then be stated as follows.

\begin{conjA}[Navarro]
\label{conj:Main, Gabriel Conjecture}
Let $G\unlhd \Gamma$ be finite groups and consider a $p$-block $B$ of $G$. If $\Gamma/G$ is a $p$-group and $B$ is $\Gamma$-invariant, then
\[\left|\IBr_\Gamma(B)\right|=\sum\limits_{Q}\left|\dz(\n_\Gamma(Q)/Q\mid B)\right|\]
where $Q$ runs over a set of representatives for the action of $\Gamma$ on the set of radical $p$-subgroups of $\Gamma$ such that $\Gamma=GQ$ and $Q\cap G$ is contained in some defect group of the $p$-block $B$.
\end{conjA}

The blockwise version of the Alperin Weight Conjecture can be recovered from Conjecture \ref{conj:Main, Gabriel Conjecture} by choosing $\Gamma=G$. Furthermore, from Conjecture \ref{conj:Main, Gabriel Conjecture} we also recover the Glauberman correspondence, by considering the case where $G$ has order prime to $p$, and more generally the Dade--Glauberman--Nagao correspondence (see Lemma \ref{lem:Navarro implies DGN}). We point out that the above statement admits a more general version in which the quotient $\Gamma/G$ need not be a $p$-group (see Conjecture \ref{conj:Blockwise Conjecture E, extended}).

At the end of \cite{Nav17} it was asked whether Conjecture \ref{conj:Main, Gabriel Conjecture} could be obtained as a consequence of the so-called Inductive (Blockwise) Alperin Weight Condition introduced in \cite{Spa13I} to reduce the (blockwise) Alperin Weight Conjecture to simple groups (see also \cite{Nav-Tie11} for the original reduction of the block-free version of Alperin's conjecture). In this paper, we show that this is indeed the case and therefore obtain a reduction of Conjecture \ref{conj:Main, Gabriel Conjecture} to simple groups.

Before stating our reduction theorem for Conjecture \ref{conj:Main, Gabriel Conjecture}, we remind the reader of a (perhaps not so well-known) phenomenon that has been observed in relation to the reduction theorems for the local-global counting conjectures. Originally, going back to the work done by E.C. Dade in \cite{Dad92}, \cite{Dad94}, and \cite{Dad97}, it was expected that for each of the local-global conjectures there would be a refinement of such a statement that would be strong enough to hold for every finite group if proved for all non-abelian finite simple groups. Dade's project remained open long after its formulation and no such reduction was found for several years. The first breakthrough in this direction was achieved by Isaacs, Malle, and Navarro in \cite{Isa-Mal-Nav07} where a reduction for the McKay Conjecture was proved. This seminal work was then followed by several other reduction theorems \cite{Nav-Tie11}, \cite{Spa13II}, \cite{Spa13I}, \cite{Spa17}. Contrary to what Dade expected, all these theorems reduce a given local-global conjecture to a much stronger statement usually referred to as its \textit{inductive condition}. However, the full strength of such inductive conditions was at first not recovered for all finite groups. This was first accomplished in \cite{Nav-Spa14I} where it was shown that assuming the inductive Alperin--McKay condition for all finite simple groups then, not only would the Alperin--McKay conjecture hold for every finite group, but even a refinement analogous to its inductive condition, as it was (ideologically) expected by Dade. Following \cite{Nav-Spa14I}, an analogous result was obtained for the McKay Conjecture in \cite{Ros-iMcK}.

In this paper, we prove a similar reduction theorem in the context of the (blockwise) Alperin Weight Conjecture. First, we state a version of the inductive condition for arbitrary finite groups by using the notion of block isomorphism of modular character triples, denoted by $\isob$, as defined in Section \ref{sec:Isomorphisms of modular character triples}.

\begin{conjA}[Inductive Blockwise Alperin Weight Condition]
\label{conj:Main, iBAWC}
Let $G\unlhd A$ be finite groups and consider a $p$-block $B$ of $G$. If $A_B$ denotes the stabiliser of $B$ in $A$, then there exists an $A_B$-equivariant bijection
\[\Omega:\IBr(B)\to\Alp(B)/G\]
such that
\[\left(A_\vartheta,G,\vartheta\right)\isob\left(\n_A(Q)_\psi,\n_G(Q),\psi\right)\]
for every $\vartheta\in\IBr(B)$ and $(Q,\psi)\in\Omega(\vartheta)$.\end{conjA}

In what follows, we say that Conjecture \ref{conj:Main, iBAWC} holds for $G$ at the prime $p$ if it holds for every $p$-block $B$ of $G$ and every choice of $G\unlhd A$. Recall, furthermore that a simple group $S$ is involved in $G$ if there exists $K\unlhd H\leq G$ such that $S\simeq H/K$. We can now state our first main result.

\begin{theoA}
\label{thm:Main, Reduction}
Let $G$ be a finite group and $p$ a prime number. If Conjecture \ref{conj:Main, iBAWC} holds at the prime $p$ for every covering group of any non-abelian finite simple group involved in $G$, then Conjecture \ref{conj:Main, iBAWC} holds for $G$.
\end{theoA}

These enhanced reduction theorems have been shown to have important implications. For instance, the reduction of the (inductive) Alperin--McKay Conjecture obtained in \cite{Nav-Spa14I} was used to deduce a reduction theorem for Brauer's Height Zero Conjecture, which led ultimately to a final solution of Brauer's conjecture for the prime $p=2$ thanks to work of Ruhstorfer \cite{Ruh22AM}. The latter was in turn used in the final proof recently obtained by Malle, Navarro, Schaeffer-Fry, and Tiep in \cite{MNSFT} while relying on a different argument for odd primes. On the other hand, the reduction of the (inductive) McKay Conjecture from \cite{Ros-iMcK} is used in the verification of the inductive McKay condition for finite simple groups of Lie type (in type D) in the work of Cabanes and Sp\"ath \cite{Cab-Spa23} and ultimately contributes to the final proof of the McKay Conjecture itself. Similarly, the inductive condition for Dade's Conjecture, also known as the Character Triple Conjecture (see \cite[Conjecture 6.3]{Spa17}), has been shown to impact the construction of certain character bijections needed to even verify the original version of Dade's Conjecture (see \cite[Section 6]{Ros-Generalized_HC_theory_for_Dade} and \cite[Section 4.2]{Ros-Unip}).

Following the path described in the above paragraph, we prove yet another application of these stronger reduction theorems. In fact, we use Theorem \ref{thm:Main, Reduction} to obtain a reduction theorem for Conjecture \ref{conj:Main, Gabriel Conjecture}. This will follow as a consequence of the following result.

\begin{theoA}
\label{thm:Main, iBAWC implies Gabriel Conjecture}
Let $G\unlhd \Gamma$ be finite groups with $\Gamma/G$ a $p$-group and consider a $p$-block $B$ of $G$. If Conjecture \ref{conj:Main, iBAWC} holds for the $p$-block $B$ with respect to $G\unlhd \Gamma$, then Conjecture \ref{conj:Main, Gabriel Conjecture} holds for the $p$-block $B$ with respect to $G\unlhd \Gamma$.
\end{theoA}

Combining Theorem \ref{thm:Main, iBAWC implies Gabriel Conjecture} and Theorem \ref{thm:Main, Reduction} we finally obtain the following reduction to finite simple groups for Conjecture \ref{conj:Main, Gabriel Conjecture}.

\begin{corA}
\label{cor:Main, Reduction for Gabriel Conjecture}
Let $G$ be a finite group and $p$ a prime number. If Conjecture \ref{conj:Main, iBAWC} holds at the prime $p$ for every covering group of any non-abelian finite simple group involved in $G$, then Conjecture \ref{conj:Main, Gabriel Conjecture} holds for every $p$-block of $G$.
\end{corA}

As mentioned above, Conjecture \ref{conj:Main, Gabriel Conjecture} can be extended to arbitrary quotients $\Gamma/G$ (see Conjecture \ref{conj:Blockwise Conjecture E, extended}). In Section \ref{sec:Navarro's conjecture}, we show that this more general statement is also a consequence of Conjecture \ref{conj:Main, iBAWC} and can therefore be reduced to finite simple groups (see Proposition \ref{prop:iBAWC implies Gabriel extended}). Furthermore, we obtain block-free versions of these results (see Section \ref{sec:Block-free Navarro Conjecture}).

As an application of our reduction theorems, and using the fact that Conjecture \ref{conj:Main, iBAWC} has been verified for several classes of finite simple groups (see, for instance, \cite{Fen-Zha22} and the references therein), we show that Conjecture \ref{conj:Main, Gabriel Conjecture} and Conjecture \ref{conj:Main, iBAWC} hold for several classes of groups and blocks, including groups with abelian Sylow $2$-subgroups or abelian Sylow $3$-subgroups, groups with odd Sylow automiser, $p$-blocks with cyclic defect groups, nilpotent $p$-blocks, and $2$-blocks with abelian defect groups.

The paper is organised as follows: after introducing the relevant notation and preliminary results on block isomorphisms of modular character triples in Section \ref{sec:Notations} and Section \ref{sec:Isomorphisms of modular character triples}, we prove certain consequences of the Inductive Blockwise Alperin Weight Condition in Section \ref{sec:iBAW}. Next, in Section \ref{sec:DGN} we prove a modular version of the Dade--Glauberman--Nagao correspondence compatible with block isomorphisms of modular character triples (see Theorem \ref{thm:Above modular DGN}). Using this result, we then prove Theorem \ref{thm:Main, Reduction} in Section \ref{sec:Reduction} (see also Section \ref{sec:Reduction for block-free version} for a block-free version of the reduction). Section \ref{sec:Navarro's conjecture} is devoted to Navarro's conjecture and to the proofs of Theorem \ref{thm:Main, iBAWC implies Gabriel Conjecture} and Corollary \ref{cor:Main, Reduction for Gabriel Conjecture}, as well as analogous block-free results (see Section \ref{sec:Block-free Navarro Conjecture}). In Section \ref{sec:Navarro with isomorphisms}, we use one further result from Section \ref{sec:DGN} (see Corollary \ref{cor:DGN modular iso}) to obtain a version of Conjecture \ref{conj:Main, Gabriel Conjecture} compatible with isomorphisms of modular character triples that extends the main result of \cite{Tur08I}. Finally, in Section \ref{sec:Results} we prove Conjecture \ref{conj:Main, Gabriel Conjecture} and Conjecture \ref{conj:Main, iBAWC} for the above-mentioned classes of groups and blocks.

\subsection*{Acknowledgments}

We thank Gabriel Navarro for an insightful conversation on the content of Section \ref{sec:DGN} during the workshop 2316 Representations of Finite groups at Mathematisches Forschungsinstitut Oberwolfach, where part of this research was conducted. We are thankful to Marc Cabanes and Gunter Malle for providing comments and corrections to an earlier version of this paper. The third author would also like to thank Andrei Marcus for a discussion about the results of \cite{Mar-Min21} (in particular for pointing out at the validity of a modular version of \cite[Proposition 5.6]{Mar-Min21}), Yuanyang Zhou for providing us with a copy of the preprint \cite{Hu-Zho}, and Julian Brough for helpful conversations about the Inductive Alperin Weight Condition during this author's doctoral studies at the Bergische Universit\"at Wuppertal. We extend our gratitude to the Institut Henri Poincar\'e for the hospitality during our visit as part of the Research in Paris program.

\section{Notation and preliminary results}
\label{sec:Notations}

Throughout the paper we use standard notation from ordinary and modular character theory. We refer the reader to \cite{Nav18} and \cite{Nav98} for a detailed introduction to the subject.

Let $p$ be a prime, $\mathbf{R}$ the ring of algebraic integers in $\overline{\mathbb{Q}}$, and fix a maximal ideal $\mathbf{M}$ of $\mathbf{R}$ containing $p$. Then the quotient $\FF=\mathbf{R}/\mathbf{M}$ is an algebraically closed field of characteristic $p$. Furthermore, if $\mathbf{S}$ is the localization of $\mathbf{R}$ at $\mathbf{M}$ then we denote by 
\[^*:\mathbf{S}\to \mathbb{F}\]
the epimorphism from \cite[Chapter 2]{Nav98}.

We denote by $\Bl(G)$ the set of $p$-blocks (or simply blocks) of a finite group $G$ and by $\lambda_{B}:\zent{\mathbb{F}G}\to \mathbb{F}$ the central function associated to each $B\in\Bl(G)$. Whenever $\chi\in\Irr(G)\cup\IBr(G)$, the central function $\lambda_\chi:\z(\FF G)\to \FF$ coincides with $\lambda_B$ if and only if the block $\bl(\chi)$ of $G$ containing $\chi$ coincides with $B$. In this case, we write $\chi\in\Irr(B)\cup\IBr(B)$. If $H\leq G$ and $b\in\Bl(H)$, then the induced block $b^G$ is defined if the linear map $\lambda_b^G$ defined in \cite[p.87]{Nav98} is an algebra homomorphism. In this case, there is a unique $B\in\Bl(G)$ such that $\lambda_b^G=\lambda_B$ and we write $b^G=B$. The central functions considered here are determined by their values on a basis of $\z(\FF G)$. One such basis is provided by the conjugacy class sums
\[\CL_G(x)^+=\sum_{y\in\CL_G(x)}y\]
considered as an element of the group algebra of $G$ and where $\CL_G(x)$ denotes the conjugacy class of $x$ in $G$.

Consider now an ordinary character $\chi\in\Irr(G)$ and let $G^0$ be the set of $p$-regular elements of $G$, that is, the set of elements of $G$ whose order is prime to $p$. Then, the restriction $\chi^0$ of $\chi$ to $G^0$ is a Brauer character that decomposes as
\[\chi^0=\sum_{\varphi\in\IBr(G)}d_{\chi,\varphi}\varphi\]
for some integers $d_{\chi,\varphi}$ called decomposition numbers. Using decomposition numbers, we can then define projective indecomposable characters. More precisely, for any Brauer character $\varphi\in\IBr(G)$, the projective indecomposable character associated to $\varphi$ is the ordinary character of $G$ defined by
\[\Phi_{\varphi}:=\sum_{\chi\in\Irr(G)}d_{\chi,\varphi}\chi.\]
 
We finally introduce the notion of character triple that will be fundamental in the rest of this work. If $N\normal G$ and $\vartheta\in\Irr(N)\cup\IBr(N)$ is $G$-invariant then we say $(G, N,\vartheta)$ is a \textit{character triple}. We say that $(G,N,\vartheta)$ is an \textit{ordinary} or \textit{modular} character triple if $\vartheta\in\Irr(N)$ or $\vartheta\in\IBr(N)$ respectively. Moreover, in the particular situation where $\vartheta\in\Irr(N)$ and $\vartheta^0\in\IBr(N)$, we say that $(G, N,\vartheta)$ is an \textit{ordinary-modular} character triple. This is the case, for instance, when $\vartheta\in\Irr(N)$ has \textit{$p$-defect zero}, that is, when $\vartheta$ satisfies $\vartheta(1)_p=|N|_p$ (see \cite[Theorem 3.18]{Nav98}). The set of irreducible ordinary characters of defect zero of a finite group $G$ is denoted by $\dz(G)$.

\begin{lem}
\label{lem:ordinary-modular extension}
Let $(G, N,\vartheta)$ be an ordinary-modular character triple with $\vartheta\in\dz(N)$. Then $\vartheta$ extends to $G$ if and only if $\vartheta^0$ extends to $G$. Furthermore, restriction to $p$-regular elements is a surjection from the set of extensions of $\vartheta$ in $\Irr(G)$ onto the set of extensions of $\vartheta^0$ in $\IBr(G)$.
\end{lem}

\begin{proof}
To start, recall that by \cite[Problem 8.13]{Nav98} there exists an ordinary-modular character triple $(G^*, N^*, \vartheta^*)$ that is isomorphic, as ordinary-modular triple, to $(G, N,\vartheta)$ and where $N^*$ is a $p'$-groups. Therefore, it is no loss of generality to assume that $N$ has order prime to $p$. Suppose now that $\vartheta^0$ extends to $G$ and let $Q/N$ be a Sylow $q$-subgroup of $G/N$ for a prime $q$. Suppose first that $q\neq p$. My assumption, we know that $\vartheta^0$ has an extension $\psi\in\IBr(G)$ and therefore $\psi_Q\in\IBr(Q)$ is an extension of $\vartheta^0$. However, since $q\neq p$ and $p$ does not divide the order of $N$, we conclude that $\psi_Q\in\IBr(Q)=\Irr(Q)$ is an ordinary extension of $\vartheta$. On the other hand, if $q=p$, then $|Q:N|$ and $|N|$ are coprime and so $\vartheta$ extends to an irreducible character in $\Irr(Q)$ according to \cite[Theorem 6.2]{Nav18}. Then \cite[Theorem 5.10]{Nav18} implies that $\vartheta$ extends to an irreducible ordinary character in $\Irr(G)$. Assume then that $\vartheta$ extends to $G$. Then, we deduce that $\vartheta^0$ extends to an irreducible Brauer character of $G$ by arguing as before but using \cite[Theorems 8.11 and 8.29]{Nav98}.
\end{proof}

Recall that the set of characters $\Irr(G/N)$ can be identified with the set of irreducible characters of $G$ containing $N$ in their kernel. A similar remark holds for Brauer characters. This identification is often referred to as \textit{inflation} of characters.

\begin{lem}
\label{lem:gallagher-lift}
Let $N\normal G$ and let $\vartheta\in\Irr(G)$ with $\vartheta^0\in{\rm IBr}(N)$. Suppose that $\vartheta$ extends to $\wt{\vartheta}\in\Irr(G)$ and that $\wt{\vartheta}^0\in\IBr(G)$. Consider $\chi\in\Irr(G\mid\vartheta)$ and set $\chi:=\eta\wt{\vartheta}$ for some $\eta\in\Irr(G/N)$. Then $\chi^0\in\IBr(G)$ if and only if $\eta^0\in\IBr(G/N)$.
\end{lem}

\begin{proof}
By Lemma \ref{lem:ordinary-modular extension} we know that $\wt{\vartheta}^\circ\in\IBr(G)$ is an extension of $\vartheta^0$. Now, if $\eta^0\in\IBr(G/N)$, we deduce from \cite[Corollary 8.20]{Nav98} that $\chi^0=\eta^0\wt{\vartheta}^0\in\IBr(G)$. Assume conversely that $\chi^0\in\IBr(G)$ and write
\[\eta^0=\sum_{\varphi\in\IBr(G/N)}d_{\eta,\varphi}\varphi.\]
Multiplying the above equality by $\wt{\vartheta}^0$, and recalling that $\eta^0\wt{\vartheta}^0=\chi^0\in\IBr(G)$ and that $\varphi\wt{\vartheta}^0\in\IBr(G)$ for every $\varphi\in\IBr(G/N)$ (again by using \cite[Corollary 8.20]{Nav98}), we deduce that there exists a unique $\varphi\in\IBr(G/N)$ such that $d_{\eta,\varphi}=1$ and that $d_{\eta,\varphi'}=0$ for all $\varphi'\neq \varphi$. Hence $\eta^0=\varphi\in\IBr(G/N)$ and we are done.
\end{proof}

\begin{lem}
\label{lem:relative-blocks}
Let $(G, N,\vartheta)$ be an ordinary-modular character triple and assume there is an extension $\wt{\vartheta}\in\Irr(G)$ of $\vartheta$ such that $\wt{\vartheta}^0\in\IBr(G)$. Let $\eta\in\Irr(G/N)$ and $\varphi\in\IBr(G/N)$. If $\bl(\eta)=\bl(\varphi)$, then $\bl(\eta\wt{\vartheta})=\bl(\varphi\wt{\vartheta}^0)$.
\end{lem}

\begin{proof}
We need to prove that $\lambda_{\eta\wt{\vartheta}}(\CL_G(x)^+)=\lambda_{\varphi\wt{\vartheta}^0}(\CL_G(x)^+)$ for all $x\in G$. By applying \cite[Lemma 2.5]{Spa13I}, we get
\begin{align*}
\lambda_{\eta\wt{\vartheta}}(\CL_G(x)^+)&=\lambda_{\eta}(\CL_{G/N}(xN)^+)\lambda_{\wt{\vartheta}_L}({\CL_L(x)^+})
\\
&=\lambda_{\varphi}(\CL_{G/N}(xN)^+)\lambda_{\wt{\vartheta}^0_L}({\CL_L(x)^+})
\\
&=\lambda_{\varphi\wt{\vartheta}^0}(\CL_G(x)^+),
\end{align*}
where $L/N=\cent{G/N}{xN}$, as desired.
\end{proof}

We conclude this section by showing that multiplication by a linear Brauer character preserves blocks of defect zero.

\begin{lem}
\label{lem:Multiplication by linear Brauer characters}
Let $B$ be a block of defect zero of a finite group $G$ and consider its unique Brauer character $\varphi\in\IBr(B)$. If $\lambda\in\IBr(G)$ is linear, then $\lambda\varphi$ belongs to a block of defect zero.
\end{lem}

\begin{proof}
It suffices to show that there exists $\psi\in\irr{G}$ such that $\psi^0=\lambda\varphi$. In fact, this would imply that $\psi(1)_p=\lambda(1)_p\varphi(1)_p=|G|_p$ and the result would follow from \cite[Theorem 3.18]{Nav98}. To prove our claim, let $\chi\in\irr{B}$ so that $\varphi=\chi^0$. By \cite[Problem 2.13]{Nav98} we know that $\lambda\Phi_\varphi^0=\Phi_{\lambda\varphi}^0$ while by the definition of $\Phi_\varphi$, and recalling that $\chi^0=\varphi$, we have $\Phi_\varphi^0=\varphi$. From this we deduce that
\[\lambda\varphi=\Phi^0_{\lambda\varphi}=\sum\limits_{\psi\in\irr{G}}d_{\psi,\lambda\varphi}\psi^0=\sum\limits_{\psi\in\irr{G}}d_{\psi,\lambda\varphi}\left(\sum\limits_{\xi\in\IBr(G)}d_{\psi,\xi}\xi\right).\]
Since $\lambda\varphi$ is an irreducible Brauer character, and because decomposition numbers are non-negative integers, the above equality forces $d_{\psi,\lambda\varphi}d_{\psi,\xi}\neq 0$ for a unique choice of $\psi\in\irr{G}$ and $\xi\in\IBr(G)$. Then, we must have $\xi=\lambda\varphi$, $d_{\psi,\lambda\varphi}=1$, and $d_{\psi,\nu}=0$ for every $\nu\in\IBr(G)$ with $\nu\neq \lambda\varphi$. This shows that $\psi^0=\lambda\varphi$ and the result follows.
\end{proof}

\section{Central and block isomorphisms of modular character triples}
\label{sec:Isomorphisms of modular character triples}

In this section, we collect the relevant results on isomorphisms of modular character triples that will be used in the rest of this paper. We refer the reader to \cite[Section 8]{Nav98} and \cite[Section 3]{Spa17I} for an overview of this theory. In particular, we will make use of the notion of \textit{central isomorphism} and \textit{block isomorphism} of modular character triples that can be found in \cite[Definition 3.1 and Definition 3.2]{Spa17I} (see also \cite[Section 3]{Nav-Spa14I}).

Recall that given a modular character triple $(G,N,\vartheta)$ there is a projective $\mathbb{F}$-representation of $G$ such that the restriction $\P_N$ affords the Brauer character $\vartheta$. Furthermore, we can always choose $\P$ such that its factor set $\alpha:G\times G\to\mathbb{F}^\times$ satisfies $\alpha(g,n)=1=\alpha(n,g)$ for every $g\in G$ and $n\in N$ (see \cite[Section 3]{Spa-Val} and \cite[Section 8]{Nav98}). In this case, we say that $\P$ is a projective $\mathbb{F}$-representation associated with $(G,N,\vartheta)$. We will often refer to $\P$ simply as a projective representation, instead of a projective $\mathbb{F}$-representation, when it is clear from the context that it is associated to a modular character triple. The next result allows us to construct well behaved strong isomorphisms of modular character triples (see \cite[Section 3]{Spa-Val}).

\begin{thm}
\label{thm:strong-iso}
Let $(G,N, \vartheta)$ and $(H, M, \varphi)$ be modular character triples and assume that $G=NH, M=H\cap N$ and that there exist projective representations $\mathcal{P}$ and $\mathcal{P}'$ associated with $(G,N,\vartheta)$ and $(H,M,\varphi)$ respectively and whose factor sets $\alpha$ and $\alpha'$ coincide via the natural isomorphism $\tau:G/N\to H/M$. Then, for any $N\leq J\leq G$ there exists a bijection
\begin{align*}
\sigma_J:\IBr(J\mid\vartheta)&\to\IBr(J\cap H\mid\varphi)
\\
\tr\left(\mathcal{Q}\otimes\mathcal{P}_J\right)&\mapsto\tr\left(\mathcal{Q}_{J\cap H}\otimes\mathcal{P}'_{J\cap H}\right)
\end{align*}
for any irreducible projective representation $\mathcal{Q}$ of $J/N$ with factor set $\alpha_{J\times J}^{-1}$ and 
\[(\sigma, \tau):(G, N,\vartheta)\to(H, M, \varphi)\]
is a strong isomorphism of modular character triples.
\end{thm}

\begin{proof}
This is \cite[Theorem 3.1]{Spa-Val}.
\end{proof}

By imposing additional requirements on the isomorphism of modular character triple $(\sigma,\tau)$ constructed above, we can define the notion of central and block isomorphisms of modular character triples. For this we follow \cite[Section 3]{Spa-Val} (see also \cite{Nav-Spa14I} and \cite{Spa17}). In what follows, if $\varphi$ is a (possibly reducible) Brauer character, then we denote by $\IBr(\varphi)$ the set of its irreducible constituents.

\begin{defi}[Central isomorphism of modular character triples]\label{def:central-iso}
Let $(\sigma,\tau):(G, N,\vartheta)\to(H, M, \varphi)$ be as in Theorem \ref{thm:strong-iso}. Suppose furthermore that $\cent G N\leq H$ and that
\[\IBr\left(\psi_{\cent J N}\right)=\IBr\left(\sigma_J(\psi)_{\cent J N}\right)\]
for any $N\leq J\leq G$ and $\psi\in\IBr(J\mid\vartheta)$. Then, we say that $(\sigma,\tau)$ is a central isomorphism of modular character triples and write $(G,N,\vartheta)\isoc(H,M,\varphi)$.
\end{defi}

As is the case for ordinary modular character triples (see \cite[Lemma 3.3]{Nav-Spa14I}), the condition required in Definition \ref{def:central-iso} can be reformulated in terms of scalar functions. More precisely, if $\P$ is a projective representation associated with the modular character triple $(G, N,\vartheta)$ then by Schur's lemma $\P(x)$ is a scalar matrix $\zeta(x) I_{\vartheta(1)}$ for every $x\in\cent G N$. This yields a scalar function $\zeta:\c_G(N)\to \mathbb{F}^\times$.

\begin{lem}\label{lem:central-iso-scalar}
Let $(\sigma,\tau):(G, N,\vartheta)\to(H, M, \psi)$ be an isomorphism given by projective representations $\P$ and $\P'$ and assume $\cent G N\leq H$. Then the following are equivalent:
\begin{enumerate}
\item For every $x\in\cent G N$ the matrices $\P(x)$ and $\P'(x)$ are associated with the same scalar.
\item $\IBr(\psi_{\cent J N})=\IBr(\sigma_J(\psi)_{\cent J N})$ for every $N\leq J\leq G$ and $\psi\in\IBr(J\mid\vartheta)$.
\end{enumerate}
\end{lem}

\begin{proof}
See the comment following \cite[Definition 3.3]{Spa-Val}.
\end{proof}

We can additionally require some compatibility conditions for block induction. The following is a modular version of \cite[Definition 3.6]{Nav-Spa14I}. Notice however that we are considering a slightly different setting that can be found, for instance, in \cite[Definition 4.2]{Spa18}.

\begin{defi}[Block isomorphism of modular character triples]\label{def:block-iso}
Let $(\sigma,\tau):(G, N,\vartheta)\to(H, M, \varphi)$ be as in Definition \ref{def:central-iso}. Assume that there exists a defect group $D$ of $\bl(\varphi)$ such that $\cent G D\leq H$ and that
\[\bl(\psi)=\bl(\sigma_J(\psi))^J\]
for any $N\leq J\leq G$ and $\psi\in\IBr(J\mid\vartheta)$. Then we say that $(\sigma,\tau)$ is a block isomorphism of modular character triples and write $(G,N,\vartheta)\isob(H,M,\varphi)$.
\end{defi}

As done in Lemma \ref{lem:central-iso-scalar}, we can reformulate the condition on block induction required in the above definition in terms of certain scalars induced by the projective representations $\P$ and $\P'$. More precisely, notice that by \cite[Theorem 8.16]{Nav98} the representations $\P_{\langle N, x\rangle}$ and $\P'_{\langle N, x\rangle\cap H}$ afford extensions $\wt{\vartheta}$ and $\wt{\psi}$ of $\vartheta$ and $\psi$ respectively. It then follows that $\P(\CL_{\langle N, x\rangle}(x)^+)=\xi I_{\vartheta(1)}$ and $\P'((\CL_{\langle N, x\rangle}(x)\cap H)^+)=\xi' I_{\varphi(1)}$ for scalars $\xi$ and $\xi'$ in $\mathbb{F}$.

\begin{lem}\label{lem:block-iso-scalar}
Let $(\sigma,\tau):(G, N,\vartheta)\to(H, M,\psi)$ be the isomorphism of modular character triples given by Theorem \ref{thm:strong-iso} for a choice of projective representations $\P$ and $\P'$. If $(G,N,\vartheta)\isoc(H, M, \psi)$, then the following are equivalent:
\begin{enumerate}
\item $(G, N,\vartheta)\isob(H, M,\psi)$.
\item For every $x\in G$ the matrices
$\P(\CL_{\langle N, x\rangle}(x)^+)$ and $\P'((\CL_{\langle N, x\rangle}(x)\cap H)^+)$ are associated with the same scalar.
\end{enumerate}
\end{lem}
\begin{proof}
This is \cite[Proposition 3.7 (b)]{Spa17I}.
\end{proof}

Next, we collect some basic properties of block isomorphisms of modular character triples.

\begin{lem}
\label{lem:Basic properties}
Let $(G,N,\vartheta)$ and $(H,M,\psi)$ be modular character triples and assume that $(G,N,\vartheta)\isob(H,M,\psi)$. Then:
\begin{enumerate}
\item $(J,N,\vartheta)\isob(H\cap J,M,\psi)$, for every $N\leq J\leq G$;
\item $(G^a,N^a,\vartheta^a)\isob(H^a,M^a,\psi^a)$, for every $a\in\aut{G}$; and
\item if $(H, M,\psi)\isob(K, L, \rho)$ for some $(K,L,\rho)$, then $(G, N, \vartheta)\isob(K, L, \rho)$.
\end{enumerate}
\end{lem}

\begin{proof}
These properties follow directly from Definition \ref{def:block-iso}. See, for instance,  the argument used to prove \cite[Lemma 3.8]{Nav-Spa14I}.
\end{proof}

Another important feature of (central and) block isomorphisms of (ordinary and) modular character triples is the fact that the bijections $\sigma_J$ considered in Theorem \ref{thm:strong-iso} in turn give analogous isomorphisms.

\begin{lem}
\label{lem:Bijections induced by isomorphisms are compatible with isomorphisms}
Suppose that $(\sigma,\tau):(G, N, \vartheta)\to(H, M, \psi)$ is a block isomorphism of modular character triples. If $N\leq J\leq G$, then the map 
\[\sigma_J:\IBr(J\mid\vartheta)\to\IBr(J\cap H\mid\psi)\]
is $\norm H J$-equivariant and
\[\left(\norm G J_\mu,J,\mu\right)\isob\left(\norm H J_\mu, J\cap H, \sigma_J(\mu)\right)\]
for every $\mu\in\IBr(J\mid\vartheta)$.
\end{lem}

\begin{proof}
This statement can be found in \cite[Lemma 3.3]{Fen-Li-Zha22} and is based on \cite[Proposition 3.9]{Nav-Spa14I}.
\end{proof}

\subsection{Construction of block isomorphisms of modular character triples}

Several standard constructions that are used in representation theory are well behaved with respect to block isomorphisms of modular character triples. Here we collect the results needed in the subsequent sections. To start, we consider irreducible induction, which appears for instance when using the Clifford correspondence as well as the Fong--Reynolds correspondence.

\begin{lem}
\label{lem:Irreducible induction}
Let $N\normal G$, $H\leq G$, $G=NH$ and $M=N\cap H$. Consider $G_1\leq G$, $H_1=G_1\cap H$, $N_1=G_1\cap N$, $M_1=G_1\cap M$ such that $G_1N=G$ and $H_1M=H$. Assume furthermore that induction gives bijections
\[\IBr\left(J\cap G_1\enspace\middle|\enspace\vartheta\right)\to \IBr\left(J\enspace\middle|\enspace\vartheta^N\right)\]
and
\[\IBr\left(J\cap H_1\enspace\middle|\enspace\psi\right)\to \IBr\left(J\cap H\enspace\middle|\enspace\psi^M\right)\]
for every $N\leq J\leq G$. If $(G_1, N_1,\vartheta)\isob (H_1,M_1,\psi)$ and $\cent G D\leq H$ for some defect group $D$ of $\bl(\psi^M)$, then $(G, N, \vartheta^N)\isob (H, M, \psi^M)$.
\end{lem}

\begin{proof}
This is \cite[Proposition 3.4]{Fen-Li-Zha22}.
\end{proof}

Next, we show that $\isob$ is compatible with direct products.

\begin{lem}
\label{lem:Direct products}
Suppose that $(G_i, N_i, \vartheta_i)\isob (H_i, M_i, \psi_i)$ for $i=1,2$. Then 
\[(G_1\times G_2, N_1\times N_2, \vartheta_1\times\vartheta_2)\isob (H_1\times H_2, N_1\times N_2,\psi_1\times\psi_2).\]
\end{lem}

\begin{proof}
For $i=1,2$, let $(\sigma_i,\tau_i)$ be a block isomorphism associated with a choice of projective representations $\P_i$ and $\P'_i$ with factor sets $\alpha_i$ and $\alpha'_i$ respectively. Set
$\P(x,y)=\P_1(x)\otimes\P_2(y)$ and $\P'(x',y')=\P'_1(x')\otimes\P'_2(y')$ for every $x\in G_1$, $y\in G_2$, $x'\in H_1$ and $y'\in H_2$. Now $\P$ and $\P'$ are projective representations of $G_1\times G_2$ and $H_1\times H_2$ associated with $\vartheta_1\times\vartheta_2$ and $\psi_1\times\psi_2$ respectively.  By the properties of the Kronecker product, if $x,g\in G_1$ and $y,t\in G_2$, we have
\begin{align*}
P((xg,yt))&=(\alpha_1(x,g)\P_1(x)\P_1(g))\otimes(\alpha_2(y,t)(\P_2(y)\P_2(t)))=\\&=\alpha_1(x,g)\alpha_2(y,t)\P((x,y))\P((g,t))
\end{align*}
and therefore the factor set $\alpha$ of $\P$ satisfies
\[\alpha((x,y),(g,t))=\alpha_1(x,g)\alpha_2(y,t).\]
The same argument applies for the factor set $\alpha'$ of $\P'$ and hence we deduce that the factor sets $\alpha$ and $\alpha'$ coincide via the isomorphism 
\[\tau:(G_1\times G_2)/(N_1\times N_2)\to (H_1\times H_2)/(M_1\times M_2).\]
We then obtain a strong isomorphism of modular character triples  
\[(\sigma, \tau):(G_1\times G_2, N_1\times N_2, \vartheta_1\times\vartheta_2)\to (H_1\times H_2, N_1\times N_2,\psi_1\times\psi_2)\] according to Theorem \ref{thm:strong-iso}. We show that $(\sigma,\tau)$ is a block isomorphism by using Lemma \ref{lem:central-iso-scalar} and Lemma \ref{lem:block-iso-scalar}. First, notice that $\cent{G_1\times G_2}{N_1\times N_2}=\cent {G_1}{N_1}\times \cent {G_2}{N_2}\leq H_1\times H_2$ and let $x\in\cent {G_1}{N_1}$ and $y\in\cent {G_2}{N_2}$. Then $\P((x,y))=\P_1(x)\otimes\P_2(y)$ is associated with the same scalar as $\P'((x,y))=\P'_1(x)\otimes\P'_2(y)$. Consider now $x\in G_1$ and $y\in G_2$ and observe that $\langle N_1\times N_2, (x,y)\rangle=\langle N_1,x\rangle \times \langle N_2, y\rangle$ and that $\CL_{\langle N_1\times N_2, (x,y)\rangle}((x,y))=\CL_{\langle N_1, x\rangle}(x)\times\CL_{\langle N_2, y\rangle}(y)$. As a consequence
\[\CL_{\langle N_1\times N_2, (x,y)\rangle}((x,y))^+=\sum_{u\in\CL_{\langle N_1, x\rangle}(x) \atop v\in \CL_{\langle N_2, y\rangle}(y)}(u,v)\]
and hence
\[\P\left(\CL_{\langle N_1\times N_2, (x,y)\rangle}(x,y)^+\right)=\P_1\left(\CL_{\langle N_1, x\rangle}(x)^+\right)\otimes\P_2\left(\CL_{\langle N_2, y\rangle}(y)^+\right).\]
Similarly, we have
\begin{align*}
\P\left(\left(\CL_{\langle N_1\times N_2, (x,y)\rangle}(x,y)\cap (H_1\times H_2)\right)^+\right)=\P_1\left(\left(\CL_{\langle N_1, x\rangle}(x)\cap H_1)^+\right)\otimes\P_2\left((\CL_{\langle N_2, y\rangle}(y)\cap H_2\right)^+\right)
\end{align*}
and we conclude that $\P(\CL_{\langle N_1\times N_2, (x,y)\rangle}(x,y)^+)$ and $\P'(\CL_{\langle N_1\times N_2, (x,y)\rangle}(x,y)\cap (H_1\times H_2))^+)$ are associated with the same scalar as desired.
\end{proof}

Using Lemma \ref{lem:Direct products} we can show that block isomorphisms of modular character triples are compatible with wreath products. First, recall that, if $N\normal G$ and $b$ is a block of $N$, then Dade's ramification group of $b$, denoted by $G[b]$, coincides with the subgroup of $G$ generated by $N$ and the elements $x\in G$ such that $\lambda_{\wt{b}}(\CL_{\langle N, x\rangle}(x)^+)\neq 0$ for some block $\wt{b}$ of $\langle N, x\rangle$ covering $b$ (see \cite[Section 2.4]{Spa17}).

\begin{lem}
\label{lem:Wreath products}
Suppose that $(G,N,\vartheta)\isob(H,M,\psi)$ and let $n$ be a positive integer. Write $\wt{G}=G\wr S_n$, $\wt{N}=N^n$, $\wt{\vartheta}=\vartheta^n$ and similarly $\wt{H}=H\wr S_n$, $\wt{M}=M^n$, $\wt{\psi}=\psi^n$. Then
\[\left(\wt{G},\wt{N},\wt{\vartheta}\right)\isob\left(\wt{H},\wt{M},\wt{\psi}\right).\]
\end{lem}

\begin{proof}
To start, we show that $(\wt{G},\wt{N},\wt{\vartheta})\isoc(\wt{H},\wt{M},\wt{\psi})$. Let $\R:S_n\to \GL_n(\mathbb{F})$ be the representation from the proof of \cite[Theorem 4.24]{Val16}, and consider projective representations $\P$ and $\P'$ with factor sets $\alpha$ and $\alpha'$ respectively giving the isomorphism $(G,N,\vartheta)\isob(H,M,\psi)$. Let $\wt{\P}$ be the projective representation of $\wt{G}$ given by
\[\wt{\P}((g_1,\dots,g_n)\sigma)=(\P(g_1)\otimes\dots\otimes \P(g_n))\R(\sigma)\]
for $g_i\in G$ and define analogously the projective a representation $\wt{\P}'$ of $\wt{H}$. The factor sets $\wt{\alpha}$ of $\wt{\P}$ and $\wt{\alpha}'$ of $\wt{\P}'$ satisfy
\begin{align*}
\wt{\alpha}((g_1,\dots,g_n)\sigma,(x_1,\dots,x_n)\tau)=&\prod_{i=1}^n\alpha\left(g_i,x_{\sigma(i)}\right)\\
\wt{\alpha}'((h_1,\dots,h_n)\sigma,(y_1,\dots,y_n)\tau)=&\prod_{i=1}^n\alpha'\left(h_i,y_{\sigma(i)}\right)
\end{align*}
for every $g_i, x_i\in G $ and $h_i, y_i\in H$. Then, noticing that $\cent{\wt G}{\wt N}=\cent G N^n\leq H^n$ and arguing as in the final part of the proof of \cite[Theorem 10.21]{Nav18}, by using Theorem \ref{thm:strong-iso} and Lemma \ref{lem:central-iso-scalar} we conclude that $(\wt{G},\wt{N},\wt{\vartheta})\isoc(\wt{H},\wt{M},\wt{\psi})$. By Lemma \ref{lem:block-iso-scalar} it remains to check that $\wt{\P}(\CL_{\langle \wt{N},\wt{x}\rangle}(\wt{x})^+)$ and $\wt{P}'(\CL_{\langle \wt{N},\wt{x}\rangle}(\wt{x})\cap \wt{H})^+)$
are associated with the same scalar for every $\wt{x}=(g_1,\dots,g_n)\sigma\in\wt G$. For this, arguing as in \cite[Lemma 4.2(b)]{Spa17} and noticing that $\bl(\wt{\psi})^{\wt{N}}=\bl(\wt{\vartheta})$, we may assume that $\wt{x}$ belongs to Dade's ramification group $\wt{G}[\bl(\wt{\vartheta})]$. In this case, \cite[Proposition 2.5(a)]{Spa17I} implies that $\wt{x}\in \wt{N}\cent{\wt{G}}{\wt{D}}\leq G^n$. Since $\wt{\P}_{G^n}=\P\otimes\dots\otimes\P$, arguing as in the proof of Lemma \ref{lem:Direct products} we can show that the matrices above are associated with the same scalar and the result follows.
\end{proof}

The next result, often referred to as the Butterfly theorem, shows that when considering block isomorphisms of (ordinary or) modular character triples we can replace the ambient group with any other group inducing the same automorphisms on the normal subgroups of the triples. This was formally stated for the first time in \cite[Theorem 5.3]{Spa17} building on earlier similar ideas used in several other reduction theorems.

\begin{lem}[Butterfly theorem]
\label{lem:Butterfly theorem}
Let $(G, N, \vartheta)$ and $(H, M, \psi)$ be modular character triples with $(G, N, \vartheta)\isob(H, M, \psi)$. Suppose that $N\normal K$ and that $\epsilon_G(G)=\epsilon_K(K)$ where $\epsilon_G:G\to \aut{N}$ and $\epsilon_K:K\to \aut{N}$ are the homomorphisms defined by conjugation. If $L:=\epsilon_K^{-1}(\epsilon_G(H))$, then $(K, N, \vartheta)\isob (L, M, \psi)$.
\end{lem}

\begin{proof}
This is \cite[Theorem 3.5]{Spa17I}
\end{proof}

Let $(G,N,\vartheta)$ be a modular character triple. For any choice of projective representation $\P$ associated with $(G,N,\vartheta)$ we can construct a central extension of $G$ containing an isomorphic copy of $N$ and where the character corresponding to $\vartheta$ extends. This can often be used to reduce questions about $(G,N,\vartheta)$ to the case where $\vartheta$ extends to $G$. The next lemma shows that this construction is compatible with block isomorphisms of modular character triples. The analogous result of ordinary character triple can be found in \cite[Theorem 4.1]{Nav-Spa14I}.

\begin{lem}
\label{lem:4.1}
Let $(G, K, \vartheta)$ be a modular character triple. Let $\P$ be a projective representation of $G$ associated with $\vartheta$. Then $\P$ defines a group $\wh{G}$ together with a surjective homomorphism $\eps:\wh{G}\to G$ with finite cyclic central kernel $Z$ of $p'$-order with the following properties.
\begin{enumerate}
\item $\wh{K}=K_0\times Z$ where $\wh{K}=\eps^{-1}(K)$, $K_0\cong K$ via the restriction $\eps_{K_0}$ and $K_0\normal \wh{G}$. Further, the action of $\wh{G}$ on $K_0$ coincides with the action of $G$ on $K$ via $\eps$.
\item The character $\vartheta_0\in\IBr(K_0)$ associated to $\vartheta$ via the isomorphism $\eps_{K_0}$ extends to $\wh{G}$.
\item If $K\leq J\leq G$ and $\wh{J}=\eps^{-1}(J)$ then $\eps(\cent{\wh{G}}{\wh{J}})=\cent{G}{J}$.
\item Let $(H, M, \vartheta')$ be a modular character triple with $(G,K,\vartheta)\isoc(H,M,\vartheta')$, and denote by $M_0$ the subgroup of $K_0$ corresponding to $M\leq K$ under $\eps_{K_0}$ and by $\vartheta_0'$ the character corresponding to $\vartheta'$. If $(\wh{G},K_0,\vartheta_0)\isob (\wh{H},M_0,\vartheta_0')$ then $(G,K,\vartheta)\isob(H, M,\vartheta')$. 
\end{enumerate}
\end{lem}

\begin{proof}
This follows arguing as in the proof of \cite[Theorem 4.1]{Nav-Spa14I}.
\end{proof}

Before proceeding further, we consider some additional compatibility properties of the construction given above.

\begin{rem}
\label{rem:sigma-maps-central-extension}
Consider the setting of Lemma \ref{lem:4.1} and let $K\leq J\leq G$. If $\sigma_J:\IBr(J\mid\vartheta)\to\IBr(J\cap H\mid\vartheta')$ and $\sigma_{\wh{J}}:\IBr(\wh{J}\mid \vartheta_0\times 1_Z)\to\IBr(\wh{J}\cap \wh{H}\mid \vartheta_0'\times 1_Z)$ are the character bijections induced by the isomorphisms of character triples considered in Lemma \ref{lem:4.1} (iv), then
\[\sigma_J(\chi)=\eps_{\wh{J\cap H}}(\sigma_{\wh{J}}(\wh{\chi}))\]
for every $\wh{\chi}\in\IBr(\wh{J}\mid\vartheta_0\times 1_Z)$ and where $\chi\in\IBr(J\mid\vartheta)$ corresponds to $\wh{\chi}$ via the isomorphism $\wh{J}/Z\simeq J$ induced by $\epsilon_{\wh{J}}$.
\end{rem}

Next, we consider the behaviour of block isomorphisms of modular character triples with respect to inflation of Brauer characters.

\begin{lem}
\label{lem:Lifting isomorphisms from quotients}
Let $G$ be a finite group and consider subgroups $H\leq G$ and $Z,N\unlhd G$ such that $Z\leq M:=N\cap H$. Set $\overline{G}:=G/Z$, $\overline{H}:=H/Z$, $\overline{N}:=N/Z$, and $\overline{M}:=M/Z$ and suppose that $(\overline{G},\overline{N},\overline{\vartheta})\isob(\overline{H},\overline{M},\overline{\varphi})$ for some irreducible Brauer characters $\overline{\vartheta}\in\IBr(\overline{N})$ and $\overline{\varphi}\in\IBr(\overline{M})$. If $\vartheta\in\IBr(N)$ and $\varphi\in\IBr(M)$ are the inflations of $\overline{\vartheta}$ and $\overline{\varphi}$ respectively, then $(G,N,\vartheta)\isob(H,M,\varphi)$.
\end{lem}

\begin{proof}
To verify the group theoretical conditions, observe that by hypothesis there is a defect group $Q$ of $\bl(\overline{\varphi})$ such that $\c_{\overline{G}}(Q)\leq \overline{H}$ and that according to \cite[Theorem 9.9]{Nav98} we can find a defect group $D$ of $\bl(\varphi)$ satisfying $Q\leq DZ/Z$ and hence $\c_G(D)Z/Z\leq \c_{\overline{G}}(Q)\leq \overline{H}$. It follows that $\c_G(D)\leq H$ as required. We can now conclude arguing as in the proof of \cite[Lemma 3.12]{Nav-Spa14I}. 
\end{proof}

The converse of Lemma \ref{lem:Lifting isomorphisms from quotients} does not hold in general. However, we can still prove an analogous statement under additional structural assumptions.

\begin{lem}
\label{lem:going to quotients}
Let $(G, N,\vartheta)$ and $(H, M,\varphi)$ be modular character triples and assume that $(G,N,\vartheta)\isob(H, M,\varphi)$. Consider $Z\leq \ker\vartheta\cap\ker\varphi$ and set $\overline{J}:=JZ/Z$ for every $J\leq G$. Denote by $\overline{\vartheta}\in\IBr(\overline{N})$ and $\overline{\varphi}\in\IBr(\overline{M})$ the characters corresponding to $\vartheta$ and $\varphi$ respectively via inflation. If $\cent G N/Z=\cent{G/Z}{N/Z}$ and $p$ does not divide the order of $Z$, then $(\overline{G},\overline{N},\overline{\vartheta})\isob(\overline{H},\overline{M},\overline{\varphi})$.
\end{lem}

\begin{proof}
By \cite[Lemma 2.17]{Spa18} we know that $(\overline{G},\overline{N},\overline{\vartheta})\isoc(\overline{H},\overline{M},\overline{\varphi})$. Next, observe that under the above assumptions we have $\cent G D/Z=\cent{G/Z}{D/Z}$ for some defect group $D$ of $\bl(\varphi)$. Then, we conclude $(\overline{G},\overline{N},\overline{\vartheta})\isob(\overline{H},\overline{M},\overline{\varphi})$ by \cite[Proposition 2.4(b)]{Nav-Spa14I}.
\end{proof}

We conclude this section by considering the compatibility of block isomorphisms of modular character triples with respect to multiplication of characters. This situation appears, for instance, when applying Gallagher's theorem and was described in \cite[Theorem 4.6]{Nav-Spa14I} in the ordinary case.

\begin{lem}
\label{lem:4.6}
Let $K\normal G$, $H\leq G$, $M=K\cap H$ and $Z\leq M$ such that $Z\normal G$. Consider $\chi\in\IBr(G)$ and suppose that $\chi_Z$ is irreducible and there exists $\beta\in\Irr(Z)$ such that $\beta^0=\chi_Z$, $\o_p(Z)\leq \ker{\beta}$ and where the inflation of $\beta$ to $Z/\o_p(Z)$ has defect zero. Set $\overline{J}:=JZ/Z$ for every $J\leq G$ and let $(\overline{G},\overline{K},\overline{\rho})$ and $(\overline{H},\overline{M},\overline{\rho}')$ be modular character triples with $(\overline{G},\overline{K},\overline{\rho})\isob(\overline{H},\overline{M},\overline{\rho}')$. Denote by $\rho$ and $\rho'$ the inflations of $\overline{\rho}$ and $\overline{\rho}'$ to $K$ and $M$ respectively and define $\tau:=\rho\chi_K$ and $\tau':=\rho'\chi_M$. If $\c_G(D)\leq H$ for some defect group $D$ of $\bl(\tau')$, then $(G, K,\tau)\isob(H, M,\tau')$.
\end{lem}

\begin{proof}
Suppose that $(\overline{G},\overline{N},\overline{\rho})\isob(\overline{H},\overline{M},\overline{\rho}')$ is given by a choice of projective representations $\overline{\P}$ and $\overline{P}'$ and consider the inflations $\P$ and $\P'$ of $\overline{\P}$ and $\overline{\P}'$ respectively. Let $\Q$ be a representation of $G$ affording $\chi$ and define $\R=\Q\otimes \P$ and $\R'=\Q_H\otimes \P'$. Then it follows that $(G, K,\tau)\isoc(H, M,\tau')$ via $(\R,\R')$. Consider now $N\leq J\leq G$ and let $\sigma_J:\IBr(J\mid\tau)\to\IBr(J\cap H\mid\tau')$ be the map given by Theorem \ref{thm:strong-iso} with respect to the projective representations $\R$ and $\R'$. Similarly, let $\overline{\sigma}_{\overline{J}}:\IBr(\overline{J}\mid\overline{\rho})\to\IBr(\overline{J\cap H}\mid\overline{\rho}')$ be the map given by Theorem \ref{thm:strong-iso} with respect to the projective representations $\overline{\P}$ and $\overline{\P}'$. If $\psi\in\IBr(J\mid\tau)$ we may write $\psi=\chi_J\vartheta$, where $\overline{\vartheta}\in\IBr(\overline{J}\mid\overline{\rho})$, and then
\[\sigma_J(\chi_J\vartheta)=\chi_{J\cap H}\vartheta'\]
where $\vartheta'\in\IBr(J\cap H)$ corresponds to $\overline{\vartheta}':=\overline{\sigma}_{\overline{J}}(\overline{\vartheta})$ via inflation. Consider $\overline{\gamma}\in\Irr(\bl(\overline{\vartheta}))$ and $\overline{\gamma}'\in\Irr(\bl(\overline{\vartheta}'))$ with inflations $\gamma$ and $\gamma'$ to $J$ and $J\cap H$ respectively. By Lemma \ref{lem:ordinary-modular extension}, we can find some $\xi\in\Irr(G)$ such that $\xi^0=\chi$ so that $\xi_Z=\beta$. Then, according to Lemma \ref{lem:relative-blocks}, we get $\bl(\gamma\xi_J)=\bl(\vartheta\chi_J)$ and similarly $\bl(\gamma'\xi_{J\cap H})=\bl(\vartheta'\chi_{J\cap H})$. Therefore to show that $\bl(\sigma_J(\psi))^J=\bl(\psi)$ it suffices to show that $\bl(\gamma\xi_J)=\bl(\gamma'\xi_{J\cap H})^J$, or equivalently, that
\[\lambda_{\gamma\xi_J}(\CL_J(x)^+)=\lambda_{\gamma'\xi_{J\cap H}}^J(\CL_J(x)^+)\]
for all $x\in J$. Write $L/Z:=\cent{\overline{J}}{xZ}$ and notice that 
\[\lambda_{\gamma\xi_J}(\CL_J(x)^+)=\lambda_{\xi_L}(\CL_L(x)^+)\lambda_{\overline{\gamma}}(\CL_{\overline{J}}(xZ)^+)\]
by \cite[Lemma 2.5]{Spa13I}. Since $\bl(\overline{\gamma}')^{\overline{J}}=\bl(\overline{\gamma})$, we deduce that
\[\lambda_{\overline{\gamma}}(\CL_{\overline{J}}(xZ)^+)=\lambda_{\overline{\gamma}'}^{\overline{J}}(\CL_{\overline{J}}(xZ)^+).\]
Finally, using \cite[Lemma 5.3.1(i)]{Nag-Tsu89} we conclude that
\begin{align*}
\lambda_{\gamma\xi_J}(\CL_J(x)^+)&=\lambda_{\xi_L}(\CL_L(x)^+)\lambda_{\overline{\gamma}'}^{\overline{J}}(\CL_{\overline{J}}(xZ)^+)
\\
&=\left(\frac{|\CL_L(x)|\xi(x)}{\xi(1)}\right)^*\left(\frac{|\CL_{\overline{J}}(xZ)|(\overline{\gamma}')^{\overline{J}}(xZ)}{(\overline{\gamma}')^{\overline{J}}(1)}\right)^*
\\
&=\left(\frac{|\CL_J(x)|\xi(x)(\overline{\gamma}')^{\overline{J}}(xZ)}{\xi(1)(\overline{\gamma}')^{\overline{J}}(1)}\right)^*
\\
&=\left(\frac{|\CL_J(x)|(\gamma'\xi_{J\cap H})^J(x)}{(\gamma'\xi_{J\cap H})^J(1)}\right)^*
\\
&=\lambda_{\gamma'\xi_{J\cap H}}^J(\CL_J(x)^+)
\end{align*}
as desired.
\end{proof}

\section{Inductive Blockwise Alperin Weight Condition and consequences}
\label{sec:iBAW}

The \textit{Inductive Alperin Weight Condition} (iAWC) first appeared in the reduction theorem of Navarro and Tiep for the Alperin Weight Conjecture \cite{Nav-Tie11}. More precisely, in \cite[Theorem A]{Nav-Tie11} it was shown  that the Alperin Weight Conjecture holds for a finite group $G$ provided that every finite non-abelian simple group involved in $G$ satisfies the iAWC. Later, in \cite{Spa13I} Sp\"ath introduced the \textit{Inductive Blockwise Alperin Weight Condition} (iBAWC) and proved a similar reduction theorem for the blockwise version of the Alperin Weight Conjecture. Both the iAWC and the iBAWC were later reformulated by Cabanes \cite{Cab13} and Sp\"ath \cite{Spa17I} in terms of isomorphisms of modular character triples. Observe that, while originally tailored to finite quasi-simple groups, these conditions can be formulated more generally for arbitrary finite groups. Before introducing a precise statement, we introduce some further notation.

Let $G$ be a finite group and denote by $\dzo(G)$ the set of irreducible Brauer characters $\psi$ of $G$ whose corresponding character $\overline{\psi}$ of the quotient $G/\o_p(G)$ belongs to a $p$-block of defect zero. Here, we are using the fact that $\o_p(G)$ is contained in the kernel of any irreducible Brauer character of $G$ thanks to \cite[Lemma 2.32]{Nav98}. Moreover, notice that in this case $\overline{\psi}=\overline{\varphi}^0$ for some uniquely defined $\overline{\varphi}\in\dz(G/\o_p(G))$ according to \cite[Theorem 3.18]{Nav98}. Now, we define a \textit{$p$-weight} of $G$ to be a pair $(Q,\psi)$ where $Q$ is a radical $p$-subgroup of $G$, that is $Q=\o_p(\n_G(Q))$, and $\psi\in\dzo(\n_G(Q))$. Let $\Alp(G)$ be the set of $p$-weights of $G$. We also denote by $\Rad(G)$ the set of radical $p$-subgroups of $G$ and write $\Rado(G)$ for the subset consisting of those radical $p$-subgroups $Q$ of $G$ such that $(Q,\psi)\in\Alp(G)$ for some $\psi\in\dzo(\n_G(Q))$. Notice that the group $G$ acts by conjugation on $\Alp(G)$ and denote by $\Alp(G)/G$ the corresponding set of $G$-orbits and by $\overline{(Q,\psi)}$ the $G$-orbit of a $p$-weight $(Q,\psi)$. We can now state the Inductive Blockwise Alperin Weight Condition for arbitrary finite groups.

\begin{conj}[Inductive Blockwise Alperin Weight Condition]
\label{conj:iBAWC 2}
Let $G\unlhd A$ be finite groups and consider a prime number $p$. Then there exists an $A$-equivariant bijection
\[\Omega:\IBr(G)\to\Alp(G)/G\]
such that
\[\left(A_\vartheta,G,\vartheta\right)\isob\left(\n_A(Q)_\psi,\n_G(Q),\psi\right)\]
for every $\vartheta\in\IBr(G)$ and $(Q,\psi)\in\Omega(\vartheta)$.
\end{conj}

Throughout the rest of this paper we say that Conjecture \ref{conj:iBAWC 2} holds for a finite group $G$ at the prime $p$ if it holds with respect to the prime $p$ and for every choice of $G\unlhd A$. We will often avoid mentioning the choice of the prime $p$ when this is clear from the context. We now collect some important properties and consequence of Conjecture \ref{conj:iBAWC 2}. First, using Lemma \ref{lem:Direct products} and Lemma \ref{lem:Wreath products} we show that Conjecture \ref{conj:iBAWC 2} extends to direct and wreath products.

\begin{pro}
\label{prop:iAWC and direct and wreath products}
Let $G\unlhd A$ be finite groups, $n$ a positive integer, and set $\wt{G}:=G^n$ and $\wt{A}:=A\wr S_n$. If Conjecture \ref{conj:iBAWC 2} holds for $G\unlhd A$, then it holds for $\wt{G}\unlhd \wt{A}$.
\end{pro}

\begin{proof}
By hypothesis we have an $A$-equivariant bijection $\Omega:\IBr(G)\to \Alp(G)/G$ such that for every $\chi\in\IBr(G)$ and $(Q,\psi)\in\Omega(\chi)$ we have 
\[(A_\chi,G,\chi)\isob(\norm A Q_\chi,\norm G Q,\psi).\]
By \cite[Theorem 8.21]{Nav98} we know that $\IBr(\wt{G})$ consists of Brauer character of the form $\chi_1\times \dots\times \chi_n$ with $\chi_i\in\IBr(G)$. Similarly, by \cite[Lemma 2.3(b)]{Nav-Tie11} the radical $p$-subgroups of $\wt{G}$ can be written as $Q_1\times \dots\times Q_n$ for some $Q_i\in\Rad(G)$ and hence each $\wt{\psi}\in\dzo(\norm{\wt{G}}{\wt{Q}})$ can be written as a product $\wt{\psi}=\psi_1\times\dots\times\psi_n$ with $\psi_i\in\dzo(\norm{G}{Q_i})$. We can then define a bijection $\wt{\Omega}:\IBr(\wt{G})\to \Alp(\wt{G})/\wt{G}$ by setting
\[\wt{\Omega}\left(\chi_1\times \dots\times \chi_n\right):=\overline{(Q_1\times\dots\times Q_n,\psi_1\times\dots\times \psi_n)}\]
for every $\chi_1,\dots,\chi_n\in\IBr(G)$ and $(Q_1,\psi_1),\dots, (Q_n,\psi_n)\in\Alp(G)$ such that $(Q_i,\psi_i)\in\Omega(\chi_i)$.
It follows from this definition, and using the fact that $\Omega$ is $A$-invariant, we also deduce that $\wt{\Omega}$ is an $\wt{A}$-equivariant bijection. To conclude, we fix $\wt{\chi}\in\IBr(\wt{G})$ and $(\wt{Q},\wt{\psi})\in\wt{\Omega}(\wt{\chi})$ and show that the corresponding modular character triples are block isomorphic. By Lemma \ref{lem:Basic properties} it is no loss of generality to assume that $\wt\chi=\chi_1\times\dots\times \chi_n$, $\wt{Q}=Q_1\times \dots\times Q_n$, and $\wt{\psi}=\psi_1\times\dots\times\psi_n$ with $(Q_i,\psi_i)\in\Omega(\chi_i)$ and where $\chi_i$ and $\chi_j$ are either equal or not $A$-conjugate and $(Q_i,\psi_i)=(Q_j,\psi_j)$ whenever $\chi_i=\chi_j$. In this case the stabiliser $\wt{A}_{\wt{\chi}}$ is a direct product of groups of the form $A_{\chi_i}\wr{S_{m_i}}$ where $m_i$ is the number of factors equal to $\chi_i$ appearing in $\wt{\chi}$. Similarly, $\n_{\wt{A}}(\wt{Q})_{\wt{\psi}}$ is the direct product of groups of the form $\n_A(Q_i)_{\psi_i}\wr{S_{m_i}}$ and we then obtain
\[\left(\wt{A}_{\wt{\chi}},\wt{G},\wt{\chi}\right)\isob\left(\norm{\wt{A}}{\wt{Q}}_{\wt\psi},\norm{\wt{G}}{\wt{Q}},\wt{\psi}\right)\]
by applying Lemma \ref{lem:Direct products} and Lemma \ref{lem:Wreath products}. This completes the proof.
\end{proof}

The following corollary allows us to control how simple groups embed in a larger group and is an important ingredient in the proof of Theorem \ref{thm:Main, Reduction}. The following argument is somewhat standard and already appeared in the reduction theorems for other local-global conjectures (see, for instance, \cite[Theorem 10.25]{Nav18}).

\begin{cor}
\label{cor:iBAW for embedded perfect group}
Let $K\unlhd A$ be finite groups with $K$ perfect and assume that $p\nmid |\zent K|$ and that $K/\z(K)$ is a direct product of isomorphic, say to $S$, non-abelian simple groups of order divisible by $p$. If Conjecture \ref{conj:iBAWC 2} holds for $X\unlhd X\rtimes \aut{X}$ where $X$ is the universal $p'$-covering group of $S$, then Conjecture \ref{conj:iBAWC 2} holds for $K\unlhd A$.
\end{cor}

\begin{proof}
Assume that $K/\z(K)$ is isomorphic to $r$ copies of $S$ and write $H=X^r$, so that $H$ is a perfect central extension of $K$, and $\wt{A}=\Aut(H)$. By Proposition \ref{prop:iAWC and direct and wreath products} we know that Conjecture \ref{conj:iBAWC 2} holds for $H\normal H\rtimes \wt{A}$ and so there exists an $\wt{A}$-equivariant bijection
\[\wt{\Omega}:\IBr(H)\to \Alp(H)/H\]
such that
\[\left(H\rtimes\wt{A}_{\chi},H,\chi\right)\isob\left(\norm H D\rtimes\norm{\wt{A}}{D}_{\varphi},\norm H D,\varphi\right)\]
for all $\chi\in\IBr(H)$ and $(D,\varphi)\in\wt{\Omega}(\chi)$. Now let $\pi:H\to K$ be the canonical epimorphism and set $Z:=\ker\pi\leq \zent H$ and $\overline{J}:=JZ/Z$ for every $J\leq H$. By the definition of central isomorphism, for all $\chi\in\IBr(H)$ and $(D,\varphi)\in\wt{\Omega}(\chi)$, we have $\IBr(\chi_Z)=\IBr(\varphi_Z)$. In particular, $Z\leq \ker\chi$ if and only if $Z\leq \ker \varphi$ and so, if $\wt{A}_Z$ denotes the stabiliser of $Z$ under the action of $\wt{A}$, then it follows that the $\wt{\Omega}$ induces an $\wt{A}_Z$-equivariant bijection
\[\wt{\Omega}_Z:\IBr\left(\overline{H}\right)\to\Alp\left(\overline{H}\right)/\overline{H}.\]
Moreover, by applying Lemma \ref{lem:going to quotients} together with \cite[Theorem 10.24(c)]{Nav18}, we deduce that
\[\left(\overline{H}\rtimes \wt{A}_Z,\overline{H},\overline{\chi}\right)\isob\left(\norm{\overline{H}}{\overline{D}}\rtimes\norm{\wt{A}_Z}{\overline{D}},\norm{\overline{H}}{\overline{D}},\overline{\varphi}\right)\]
for all $\overline{\chi}\in\IBr(\overline{H})$ and $(\overline{D},\overline\varphi)\in\wt{\Omega}_Z(\overline\chi)$. Since $\overline{H}\simeq K$ and $\overline{H}\rtimes \wt{A}_Z\simeq K\rtimes\Aut(K)$ (see, for instance, the proof of \cite[Theorem 10.25]{Nav18}) this proves that Conjecture \ref{conj:iBAWC 2} holds for $K\normal K\rtimes\Aut(K)$. Therefore there exists an $\Aut(K)$-equivariant bijection 
\[\Omega:\IBr(K)\to\Alp(K)/K\]
such that
\[\left(K\rtimes \Aut(K),K,\psi\right)\isob\left(\norm K Q\rtimes\norm{\Aut(K)}{Q}_\vartheta,\norm K Q,\vartheta\right)\]
for every $\psi\in\IBr(K)$, $(Q,\vartheta)\in\Omega(\psi)$. We can now conclude by applying Sp\"ath's Butterfly theorem. More precisely, let $\eps:A\to \Aut(K)$ and $\wh{\eps}:K\rtimes \Aut(K)\to \Aut(K)$ be the homomorphisms induced by conjugation on $K$ via elements of $A$ and $K\rtimes \aut{K}$ respectively. Set $Y:=\epsilon(A)\leq \aut{K}$ and notice that $\Omega$ is $Y$-equivariant since it is $\aut{K}$-equivariant. 
Now, let $\psi\in\IBr(K)$, $(Q,\vartheta)\in\Omega(\psi)$ and write $U=\wh{\eps}^{-1}(Y_{\psi})$. Observe that $Y_\psi=\eps(A_{\psi})$ and that $\wh{\eps}(U)=Y_\psi=\eps(A_{\psi})$. Since $U\leq K\rtimes\Aut(K)$ we get
\[\left(U,K,\psi\right)\isob\left(\norm U Q,\norm K Q,\vartheta\right)\]
by Lemma \ref{lem:Basic properties}. Then, if we prove that $\norm{A}{Q}_{\psi}=\eps^{-1}(\wh{\eps}(\norm U Q))$, applying Lemma \ref{lem:Butterfly theorem} we finally obtain
\[\left(A_\chi,K,\psi\right)\isob\left(\norm A Q_{\vartheta},\norm K Q,\vartheta\right).\]
To prove the claimed equality, let $x\in\norm{A}{Q}_\psi$ so that $\eps(x)\in Y_{\psi}$ and hence $\eps(x)=\eps(g)$ for some $g\in U$. Since $x$ and $g$ induce the same action on $K$, it follows that $g$ normalizes $Q$ and so $g\in\norm U Q$. Conversely, if $g\in\norm U Q$ then $\wh{\eps}(g)\in Y_{\psi}$ and there is some $x\in A_{\psi}$ with $\eps(x)=\wh{\eps}(g)$. Then $x$ must normalize $Q$ and the claim follows. This finally completes the proof.
\end{proof}

Following the proofs of \cite[Proposition 2.10]{Ros22} and \cite[Proposition 2.4]{Ros-iMcK}, we now show how to lift the bijections given by Corollary \ref{cor:iBAW for embedded perfect group} from the group $K$ to any intermediate subgroup $K\leq J\leq A$. We first introduced one more definition.

\begin{defi}
\label{def:Above Weights}
For any finite group $G$ with a normal subgroup $K\unlhd G$, we define the set $\Alp(G\mid K)$ consisting of pairs $(Q,\eta)$ where $Q\in \Rad(K)$ and $\eta\in\IBr(\n_G(Q))$ lies above some Brauer characters in the set $\dzo(\n_K(Q))$. Observe that the group $G$ acts by conjugation on the $\Alp(G\mid K)$ and denote by $\Alp(G\mid K)/G$ the corresponding set of $G$-orbits.
\end{defi}

We can now prove the main result of this section.

\begin{thm}
\label{thm:Lifting bijection for weights}
Suppose that Conjecture \ref{conj:iBAWC 2} holds with respect to the finite groups $K\unlhd A$. If $K\leq J\leq A$, then there exists an $\n_A(J)$-equivariant bijection
\[\Omega_K^J:\IBr(J)\to \Alp(J\mid K)/J\]
such that
\begin{equation}
\label{eq:Lifting bijections for weights}
\left(\n_A(J)_\chi,J,\chi\right)\isob\left(\n_A(J,Q)_\eta,\n_J(Q),\eta\right)
\end{equation}
for every $\chi\in\IBr(J)$ and $(Q,\eta)\in\Omega_K^J(\chi)$.
\end{thm}

\begin{proof}
To start, replacing $A$ with $\n_A(J)$, observe that it is no loss of generality to assume that $J$ is normal in $A$. Let $\Omega_K$ be the $A$-equivariant bijection given by Conjecture \ref{conj:iBAWC 2} and, for every $\vartheta\in\IBr(K)$ and $(Q,\psi)\in\Omega_K(\vartheta)$, fix a pair of projective representations $(\P^{(\vartheta)},\P^{(Q,\psi)})$ inducing the block isomorphism of modular character triples
\begin{equation}
\label{eq:Lifting bijections for weights 1}
\left(A_\vartheta,K,\vartheta\right)\isob\left(\n_A(Q)_\psi,\n_K(Q),\psi\right).
\end{equation}
Let $\mathcal{S}$ be an $A$-transversal in $\IBr(K)$ and denote by $\mathbb{S}$ the set consisting of the $K$-orbits $\Omega_K(\vartheta)$ for $\vartheta\in\mathcal{S}$. The equivariance properties of the bijection $\Omega_K$ imply that $\mathbb{S}$ is an $A$-transversal in $\Alp(K)/K$. Observe that each irreducible Brauer character $\chi$ of $J$ lies above an irreducible Brauer character $\vartheta'$ of $K$ that is $A$-conjugate to a unique $\vartheta\in\mathcal{S}$. In particular, there exists an $A$-transversal $\mathcal{T}$ in $\IBr(J)$ such that each element $\chi\in\mathcal{T}$ lies above some element $\vartheta\in\mathcal{S}$. Furthermore, if $\chi$ lies above another $\vartheta'\in\mathcal{S}$ then there exists an element $x\in J$ such that $\vartheta'=\vartheta^x$. Since $J\leq A$, the choice of $\mathcal{S}$ yields $\vartheta=\vartheta'$. Therefore, every element $\chi\in\mathcal{T}$ lies over a unique $\vartheta\in\mathcal{S}$. Consider now the Clifford correspondent $\varphi\in\IBr(J_\vartheta\mid \vartheta)$ of $\chi$ (see \cite[Theorem 8.9]{Nav98}) and remember that the choice of projective representations $(\P^{(\vartheta)},\P^{(Q,\psi)})$ associated with \eqref{eq:Lifting bijections for weights 1} induces an $\n_A(Q)_\vartheta$-equivariant bijection
\[\sigma_{J_\vartheta}:\IBr\left(J_\vartheta\enspace\middle|\enspace \vartheta\right)\to \IBr\left(\n_J(Q)_\psi\enspace\middle|\enspace \psi\right)\]
where we are using the fact that $\n_A(Q)_\vartheta=\n_A(Q)_\psi$. Again using \cite[Theorem 8.9]{Nav98} it follows that $\sigma_{J_\vartheta}(\varphi)^{\n_J(Q)}$ is an irreducible Brauer character of $\n_J(Q)$ for each $\varphi\in\IBr(J_\vartheta\mid \vartheta)$. Then, the set $\mathbb{T}$ consisting of $J$-orbits of pairs $(Q,\sigma_{J_\vartheta}(\varphi)^{\n_J(Q)})$ is an $A$-transversal in $\Alp(J\mid K)/J$ and there exists a bijection
\[\Phi:\mathcal{T}\to\mathbb{T}\]
given by sending $\chi$ to the $J$-orbit of $(Q,\sigma_{J_\vartheta}(\varphi)^{\n_J(Q)})$ where $\varphi$ is the Clifford correspondent of $\chi$ over $\vartheta$ as above. We can finally define an $A$-equivariant bijection by setting
\[\Omega_K^J\left(\chi^x\right):=\Phi(\chi)^x\]
for every $\chi\in\mathcal{T}$ and $x\in A$. It remains to show that the isomorphism \eqref{eq:Lifting bijections for weights 1} implies \eqref{eq:Lifting bijections for weights}. For this purpose, let $\chi\in\mathcal{T}$ and $(Q,\eta)\in\Phi(\chi)$ so that $\eta=\sigma_{J_\vartheta}(\varphi)^{\n_J(Q)}$ where $\vartheta$ is the unique character of $\mathcal{S}$ lying below $\chi$ and $\varphi\in\IBr(J_\vartheta)$ is the Clifford correspondent of $\chi$ over $\vartheta$. By Lemma \ref{lem:Basic properties} (ii) it is enough to show that the condition on modular character triples \eqref{eq:Lifting bijections for weights} is satisfied for this specific choice of $\chi$ and $(Q,\eta)$. Observe that because $\sigma_{J_\vartheta}$ is $\n_A(Q)_\vartheta$-equivariant and $\n_A(Q)_\vartheta=\n_A(Q)_\psi$, the stabiliser $\n_A(Q)_{\vartheta,\varphi}$ coincides with $\n_A(Q)_{\psi,\sigma_{J_\vartheta}(\varphi)}$. Then, by applying Lemma \ref{lem:Bijections induced by isomorphisms are compatible with isomorphisms} to the block isomorphism \eqref{eq:Lifting bijections for weights 1}, we get
\[\left(A_{\vartheta,\varphi},J_\vartheta,\varphi\right)\isob\left(\n_A(Q)_{\psi,\sigma_{J_\vartheta}(\varphi)},\n_J(Q)_\psi,\sigma_{J_\vartheta}(\varphi)\right)\]
from which we deduce
\[\left(A_{\chi},J,\chi\right)\isob\left(\n_A(Q)_{\eta},\n_J(Q),\eta\right)\]
according to Lemma \ref{lem:Irreducible induction}. Observe that the latter result can be applied because $A_{\vartheta,\varphi}=A_{\vartheta,\chi}$ by the Clifford correspondence while $A_\chi=JA_{\vartheta,\chi}$ and $A_\chi=J\n_A(Q)_\eta$ by the Frattini argument applied together with Clifford's theorem and the equivariance properties of $\Omega_K^J$ respectively. 
\end{proof}

\section{The Dade--Glauberman--Nagao correspondence and modular character triples}
\label{sec:DGN}

The aim of this section is to obtain a bijection for Brauer characters compatible with the Dade--Glauberman--Nagao correspondence and inducing block isomorphisms of modular character triples. Our Theorem \ref{thm:Above modular DGN} below extends \cite[Theorem 4.2]{Nav-Tie11}, \cite[Theorem 3.8]{Spa13I}, \cite[Proposition 3.11]{Fen-Li-Zha23II}, and provides a modular version of \cite[Theorem 5.13]{Nav-Spa14I}. 

\subsection{Relative defect zero Brauer characters}

Let $N\unlhd G$ be finite groups. For every irreducible characters $\chi\in\irr{G}$ and $\vartheta\in\irr{N}$ with $\chi$ lying above $\vartheta$, recall that $\chi(1)/\vartheta(1)$ divides the index $|G:N|$ according to \cite[Theorem 5.12]{Nav18}. Then, we define the \textit{$N$-relative defect} of $\chi$ to be the non-negative integer $d_N(\chi)$ such that
\[p^{d_N(\chi)}=\dfrac{|G:N|_p}{\chi(1)_p/\vartheta(1)_p}.\]
Observe that $d_N(\chi)$ does not depend on the choice of $\vartheta\in\irr{N}$ lying below $\chi$. For a given $\vartheta\in\irr{N}$, we denote by $\rdz(G\mid \vartheta)$ the set of irreducible characters of $G$ with $N$-relative defect zero and lying above $\vartheta$. 

Our first aim is to define a notion of relative defect zero Brauer character. Unfortunately, for a Brauer character $\chi\in\IBr(G)$ it is not true in general that $\chi(1)_p$ divides $|G|_p$ (see the example preceding \cite[Theorem 3.18]{Nav98}) and therefore the obvious definition in terms of character degrees will not work in this context. To circumvent this problem, we show that (under suitable assumptions) relative defect zero characters remain irreducible under reduction modulo $p$. More precisely, we prove the following result.

\begin{lem}
\label{lem:Definition of rdzo}
Let $K\leq M$ be normal subgroups of $G$ with $M/K$ a $p$-group and consider a $G$-invariant $\vartheta\in\dz(K)$. Let $\wh{\vartheta}\in\irr{M}$ be a $G$-invariant extension of $\vartheta$ (which exists according to \cite[Theorem 2.4]{Nav-Tie11}).
\begin{enumerate}
\item If $\chi\in\rdz(G\mid \wh{\vartheta})$, then $\chi^0\in\IBr(G)$.
\item The map $\rdz(G\mid \wh{\vartheta})\to\IBr(G)$ given by sending $\chi$ to $\chi^0$ is injective.
\item The image of $\rdz(G\mid \wh{\vartheta})$ in $\IBr(G)$ under the above map does not depend on the choice of the extension $\wh{\vartheta}$.
\end{enumerate}
\end{lem}

\begin{proof}
By \cite[Problem 8.13]{Nav98} we can find an ordinary-modular character triple $(H,Z,\lambda)$ with $Z$ a central subgroup of $H$ with order prime to $p$ and an isomorphism of ordinary-modular character triples $(\sigma,\tau):(G,K,\vartheta)\to (H,Z,\lambda)$. Let $Z\leq N\leq H$ such that $\tau(M/K)=N/Z$ and set $\wh{\lambda}:=\sigma_M(\wh{\vartheta})$. Observe that $\lambda\in\dz(Z)$, that $\wh{\lambda}$ is an $H$-invariant extension of $\lambda$, and that $N/Z$ is a $p$-group. Next, let $\varphi:=\sigma_G(\chi)$ and notice that $\varphi$ lies above $\wh{\lambda}$. Furthermore, $\chi(1)/\wh{\vartheta}(1)=\varphi(1)/\wh{\lambda}(1)$ and therefore $\varphi\in\rdz(H\mid \wh{\lambda})$. Furthermore, since $(\sigma,\tau)$ is an isomorphism of ordinary-modular character triples, it follows that $\chi^0\in\IBr(G)$ if and only if $\varphi^0\in\IBr(H)$. Hence, it is no loss of generality to assume that $K$ is a central subgroup of $G$ of order prime to $p$.

Now, we can write $M=K\times D$ for a Sylow $p$-subgroup $D$ of $M$ and $\wh{\vartheta}=\vartheta\times \mu$ for some $G$-invariant linear character $\mu\in\irr{D}$. In this situation \cite[Theorem 4.1]{Nav04I} yields a canonical bijection
\begin{align*}
\dz(G/D)&\to \rdz(G\mid \mu)
\\
\psi &\mapsto\psi_\mu
\end{align*}
and where $\psi_\mu^0=\wh{\mu}\cdot\psi^0$ for some linear Brauer character $\wh{\mu}$ of $G$. Then, recalling that $\psi^0$ is an irreducible Brauer character for every $\psi\in\dz(G/D)$, we deduce that $\psi_\mu^0\in\IBr(G)$ for every $\psi_\mu\in\rdz(G\mid \mu)$. To prove the first statement, it now suffices to show that each character $\chi\in\rdz(G\mid \wh{\vartheta})$ belongs to $\rdz(G\mid \mu)$. To see this recall that $K$ is a $p'$-group and $\mu$ is linear so that $\chi(1)_p=\wh{\vartheta}(1)_p|G:M|_p=\mu(1)_p|G:D|$ for every $\chi\in\rdz(G\mid \wh{\vartheta})$. This shows that $\chi\in\rdz(G\mid \mu)$ as claimed. Furthermore, if $\chi_i\in\rdz(G\mid \wh{\vartheta})$ for $i=1,2$, then we can find $\psi_i\in\dz(G/D)$ such that $\chi_i=\psi_{i,\mu}$. If $\chi_1^0=\chi_2^0$, then we get $\wh{\mu}\cdot\psi_1^0=\wh{\mu}\cdot\psi_1^0$ and therefore $\psi_1^0=\psi_2^0$. This implies that $\psi_1=\psi_2$ and therefore that $\chi_1=\chi_2$ which implies the second sentence of the statement.

We now show that the set of characters of the form $\chi^0$ for $\chi\in\rdz(G\mid \wh{\vartheta})$ does not depend on the choice of the extension $\wh{\vartheta}$. Suppose that $\wh{\vartheta}'$ is another $G$-invariant extension of $\vartheta$ to $M$. Arguing as in the previous paragraph, this determines a unique $G$-invariant character $\mu'\in\irr{D}$ and a linear Brauer character $\wh{\mu}'\in\IBr(G)$ such that $\psi_{\mu'}'^0=\wh{\mu}'\cdot\psi'^0$ for every $\psi'\in\dz(G/D)$. Now let $\chi'\in\rdz(G\mid \mu')$ and write $\chi'=\psi'_{\mu'}$ for some $\psi'\in\dz(G/D)$. In order to prove (iii), we need to find $\chi\in\rdz(G\mid \mu)$ such that $\chi'^0=\chi^0$. As explained before, we can write $\chi'^0=\wh{\mu}'\cdot\psi'^0$ and hence $\chi'^0=\wh{\mu}\cdot(\lambda\cdot\psi'^0)$ for $\lambda:=\wh{\mu}^{-1}\cdot\wh{\mu}'$. Moreover, since $\lambda$ is linear, Lemma \ref{lem:Multiplication by linear Brauer characters} implies that $\lambda\cdot\psi'^0$ belongs to a block of defect zero of $G/D$. In particular, there exists some $\psi\in\dz(G/D)$ such that $\psi^0=\lambda\cdot\psi'^0$. This implies that $\chi'^0=\wh{\mu}\cdot\psi^0=\psi_\mu^0$ and our claim follows by setting $\chi=\psi_\mu\in\rdz(G\mid \mu)$.
\end{proof}

We can now define the set of relative defect zero Brauer characters.

\begin{defi}
\label{def:rdzo}
Let $K\leq M$ be normal subgroups of $G$ with $M/K$ a $p$-group and consider a $G$-invariant $\varphi\in\dzo(K)$. Set $\overline{H}:=H\o_p(K)/\o_p(K)$ for every $H\leq G$. By definition $\overline{\varphi}$ belongs to a block of defect zero of $\overline{K}$ and we can find a unique $\overline{\vartheta}\in\irr{\overline{K}}$ such that $\overline{\vartheta}^0=\overline{\varphi}$ according to \cite[Theorem 3.18]{Nav98}. By applying Lemma \ref{lem:Definition of rdzo} to $\overline{G}$, for any $\overline{G}$-invariant extension $\wh{\vartheta}\in\irr{\overline{M}}$ of $\overline{\vartheta}$, the set of Brauer characters $\overline{\chi}^0$ for $\overline{\chi}\in\rdz(\overline{G}\mid \wh{\vartheta})$ is a well defined subset of $\IBr(\overline{G})$ which does not depend on the choice of the extension $\wh{\vartheta}$. We define the set $\rdzo(G\mid M,\varphi)$ to be the set of inflations to $G$ of such Brauer characters $\overline{\chi}^0$. More generally, if $\varphi\in\dzo(K)$ is $M$-invariant, but not necessarily $G$-invariant, then we denote by $\rdzo(G\mid M,\varphi)$ the set of Brauer characters of $G$ whose Clifford correspondent over $\varphi$ (see \cite[Theorem 8.9]{Nav98}) belongs to $\rdzo(G_\varphi\mid M,\varphi)$.
\end{defi}

Recall that the set $\dz(G)$ of defect zero characters of a finite group $G$ can be recovered as the set of $1$-relative defect zero characters, i.e. $\dz(G)=\rdz(G\mid 1)$ and where we denote by $1$ the trivial character of the identity group. Similarly, we observe that the set $\dzo(G)$ can be recovered as a particular case of Definition \ref{def:rdzo}.

\begin{rem}
Consider $K\leq M\leq G$ and $\vartheta$ as in Definition \ref{def:rdzo}. If $K=1$, $M=\o_p(G)$ and $\vartheta=1$, then \cite[Theorem 3.18]{Nav98} implies that $\dzo(G)$ coincides with $\rdzo(G\mid \o_p(G),1)$ and where we denote by $1$ the trivial Brauer character of the identity group.
\end{rem}

\subsection{A bijection above the Dade--Glauberman--Nagao correspondence}

We now come to the main result of this section. In order to introduce this statement, we quickly recall the definition of the Dade--Glauberman--Nagao correspondence as defined in \cite[Section 4]{Nav-Tie11}. Assume that $K\unlhd M$ with $M/K$ a $p$-group and let $\vartheta\in\dz(K)$ be $M$-invariant. By \cite[Corollary 9.6]{Nav98} there is a unique block, say $b$, of $M$ covering the block of $\vartheta$. Furthermore, if $D$ is a defect group of $b$, then \cite[Theorem 9.17]{Nav98} implies that $D$ is a complement of $K$ in $M$, that is, $M=KD$ and $1=K\cap D$. Now, notice that $\n_M(D)=D\times \c_K(D)$ and that, if $C$ is the Brauer correspondent of $b$ in $\n_M(D)$, then $C$ covers a unique block $c$ of $\c_K(D)$ with defect zero. We denote by $\Pi_D(\vartheta)\in\dz(\c_K(D))$ the unique ordinary character belonging to $c$, called the \textit{Dade--Glauberman--Nagao correspondent} (DGN correspondent for short) of $\vartheta$ with respect to $D$. We refer the reader to \cite[Section 4]{Nav-Tie11} and \cite{Nav-Spa14II} for further information.

Next, we define a version of the DGN correspondence for Brauer characters. To start, and for future reference, we consider the following hypothesis.

\begin{hyp}\label{hyp:DGN hypotheses}
Suppose that $K\unlhd M$ with $M/K$ a $p$-group and let $\varphi\in\dzo(K)$ be $M$-invariant. Set $L:=\o_p(K)$ and denote by $\overline{\varphi}$ the Brauer character of $\overline{K}:=K/L$ corresponding to $\varphi$. By \cite[Theorem 3.18]{Nav98} we can find a unique character $\overline{\vartheta}\in\dz(\overline{K})$ such that $\overline{\vartheta}^0=\overline{\varphi}$. Let $D$ be a $p$-subgroup of $M$ such that $\overline{D}:=D/L$ is a defect group of the unique block of $\overline{M}:=M/L$ covering the block of $\overline{\varphi}$ and notice that $M=KD$ and $K\cap D=L$ by \cite[Theorem 9.17]{Nav98}. Observe that the uniqueness of $\overline{\vartheta}$ implies that $\overline{\vartheta}$ is $\overline{M}$-invariant.
\end{hyp}

We can now define a version of the DGN correspondence for Brauer characters as follows.

\begin{defi}
\label{def:DGN for Brauer characters}
Assume Hypothesis \ref{hyp:DGN hypotheses}. By the above paragraph we can define the DGN correspondent $\overline{\vartheta}':=\Pi_{\overline{D}}(\overline{\vartheta})\in\dz(\c_{\overline{K}}(\overline{D}))$ of $\overline{\vartheta}$ with respect to $\overline{D}$. Now, let $\vartheta'$ be the ordinary character of $\n_K(D)$ corresponding to $\overline{\vartheta}'$ via inflation and define the \textit{Dade--Glauberman--Nagao correspondence} (DGN correspondent for short) of $\varphi$ with respect to $D$ by setting
\[\pi_D(\varphi):=\vartheta'^0\in\dzo(\norm K D).\]
\end{defi}

We are now ready to state the main result of this section. 

\begin{thm}
\label{thm:Above modular DGN}
Let $K\leq M\leq A$ be finite groups with $K$ and $M$ normal in $A$ and $M/K$ a $p$-group. Let $\varphi\in\dzo(K)$ be $A$-invariant and $D$ a $p$-subgroup of $M$ such that $D/\o_p(K)$ is a defect group of the unique block of $M/\o_p(K)$ covering the block of $\overline{\varphi}$ in $K/\o_p(K)$. Consider the DGN correspondent $\pi_D(\varphi)\in\dzo(\n_K(D))$ as in Definition \ref{def:DGN for Brauer characters}. If $M\leq G\unlhd A$ and $M/K$ is a radical $p$-subgroup of $G/K$, then there exists an $\n_A(D)$-equivariant bijection
\[\Delta_{D,\varphi}^G:\rdzo\left(G\enspace\middle|\enspace M,\varphi\right)\to\dzo\left(\n_G(D)\enspace\middle|\enspace\pi_D(\varphi)\right)\]
such that
\[\left(A_\chi,G,\chi\right)\isob\left(\n_A(D)_\chi,\n_G(D),\Delta_{D,\varphi}^G(\chi)\right)\]
for every $\chi\in\rdzo(G\mid M,\varphi)$.
\end{thm}

\begin{rem}
Observe that if $D$ is the $p$-subgroup of $M$ considered in the above theorem, then $D$ is a radical $p$-subgroup of $G$. In fact, notice first that $D$ is a radical $p$-subgroup of $M$ and recall furthermore that $M/K$ is a radical $p$-subgroup of $G/K$ by hypothesis. Since $G/K$ is isomorphic to $\n_G(D)/\n_K(D)$ we deduce that $\n_M(D)/\n_K(D)$ is a radical $p$-subgroup of $\n_G(D)/\n_K(D)$. Then, since $\n_M(D)/\n_K(D)$ is normal in $\n_G(D)/\n_K(D)$, it follows that $\o_p(\n_G(D)/\n_M(D))=1$. This implies that $\o_p(\n_G(D)/D)$ is contained in $\n_M(D)/D$. But then, recalling that $D$ is a radical $p$-subgroup of $M$, we get $\o_p(\n_M(D)/D)=1$ and therefore $\o_p(\n_G(D)/D)\leq\o_p(\n_M(D)/D)=1$ which implies that $D=\o_p(\n_G(D))$ as claimed.
\end{rem}

Our proof of Theorem \ref{thm:Above modular DGN} is inspired by the argument developed in \cite[Section 5]{Nav-Spa14I}. We start with the following lemma (see \cite[Proposition 5.12]{Nav-Spa14I}).

\begin{lem}
\label{lem:scalars-centralizer}
Assume Hypothesis \ref{hyp:DGN hypotheses} with $L=1$ and suppose that $\vartheta$ extends to $\wt{\vartheta}\in\Irr(A)$. Set $\vartheta':=\Pi_D(\vartheta)$, $H:=\n_A(D)$, and $N:=\c_K(D)$. Then there exists an extension $\wt{\vartheta}'\in\Irr(H)$ of $\vartheta'$ such that
\[\wt{\vartheta}(x)^*=e^*\wt{\vartheta}'(x)^*\]
whenever $x$ is a $p$-regular element of $G$ with $D\in\Syl_p(\cent M x)$ and where $e=[\vartheta_N,\vartheta']$. Furthermore, we have
\[\IBr\left((\wt{\vartheta}^0)_{\cent A M}\right)=\IBr\left((\wt{\vartheta}'^0)_{\cent A M}\right).\]
\end{lem}

\begin{proof}
Let $\wt{\vartheta}'$ be the extension of $\vartheta$ given by \cite[Proposition 5.12]{Nav-Spa14I} and observe that the first part of the statement is satisfied and that in addition
\[\irr{\wt{\vartheta}_{\c_A(M)}}=\irr{\wt{\vartheta}'_{\c_A(M)}}.\]
But then the latter equality implies that $\IBr(\wt{\vartheta}^0_{C})=\IBr((\wt{\vartheta}')^0_C)$ as required.
\end{proof}

We now use Lemma \ref{lem:scalars-centralizer} to construct the following block isomorphisms of modular character triples.

\begin{pro}
\label{pro:dgn-isomorphism-with-extension}
Assume Hypothesis \ref{hyp:DGN hypotheses} with $L=1$ and suppose that $M, K\normal A$ and that $\varphi$ extends to $\wt{\varphi}\in\IBr(A)$. Consider $\vartheta'=\Pi_{D}(\vartheta)$ and $\varphi'=\pi_D(\varphi)=(\vartheta')^0$, and let $\psi\in\IBr(M)$ and $\psi'\in\IBr(\norm M D)$ be the unique characters lying above $\varphi$ and $\varphi'$ respectively. Then
\[(A,M,\psi)\isob (\n_A(D),\n_M(D),\psi').\]
Furthermore, if $M\leq J\leq A$ and $D$ is a radical $p$-subgroup of $J$ then the bijection
\[\sigma_J:\IBr\left(J\enspace\middle|\enspace\psi\right)\to\IBr\left(\n_J(D)\enspace\middle|\enspace\psi'\right)\] given by Theorem \ref{thm:strong-iso} maps $\rdzo(J\mid M,\varphi)$ onto $\dzo(\norm J D\mid\varphi')$.
\end{pro}

\begin{proof}
To start, we construct the block isomorphism of modular character triples stated above. Straightforward calculations show that the group theoretical conditions from Definition \ref{def:block-iso} are satisfied. By Lemma \ref{lem:ordinary-modular extension}, there exists an extension $\wt{\vartheta}\in\Irr(A)$ of $\vartheta$ such that $\wt{\vartheta}^0=\wt{\varphi}$. Set $H:=\n_A(D)$ and let $\wt{\vartheta}'\in\Irr(H)$ be the extension of $\vartheta'$ given by Lemma \ref{lem:scalars-centralizer}. Now, if we define $\wt{\varphi}':=(\wt{\vartheta}')^0$, then we deduce from Lemma \ref{lem:ordinary-modular extension} that $\wt{\varphi}'\in\IBr(H)$ is an extension of $\varphi'$. By \cite[Theorem 8.11]{Nav98} it follows that $\wt{\varphi}$ and $\wt{\varphi}'$ are actually extensions of $\psi$ and $\psi'$ respectively. Let $\P$ and $\P'$ be modular representations of $A$ and $H$ affording $\wt{\varphi}$ and $\wt{\varphi}'$ respectively and consider the corresponding strong isomorphism of modular character triples $(\sigma,\tau):(A, M, \psi)\to(H, \norm M D, \psi')$ given by Theorem \ref{thm:strong-iso}. Notice that if $M\leq J\leq A$ and $\chi\in\IBr(J\mid\psi)$, then \cite[Corollary 8.20]{Nav98} implies that $\chi=\beta\wt{\varphi}_J$ for some $\beta\in\IBr(J/M)$ and then we also have $\sigma_J(\chi)=\beta_{J\cap H}\wt{\varphi}'_{J\cap H}$ with $\beta_{J\cap H}\in \IBr(J\cap H/\n_M(D))$.

Next, by the definition of $\wt{\varphi}'$ and according to Lemma \ref{lem:scalars-centralizer}, we deduce that $\IBr(\wt{\varphi}_{\c_A(M)})=\IBr(\wt{\varphi}'_{\c_A(M)})$ and therefore that $\IBr(\wt{\varphi}_{\c_J(M)})=\IBr(\wt{\varphi}'_{\cent J M})$ for every $M\leq J\leq A$. Furthermore, if $\beta\in\IBr(J/M)$, then $\IBr((\beta\wt{\varphi})_{\cent J M})$ coincides with the set of irreducible constituents of the characters of the form $\eta\nu_{\cent J M}$ with $\eta\in\IBr(\beta_{\cent J M})$ and $\nu\in\IBr(\wt{\varphi}_{\cent G M})$. This shows that
\[\IBr\left((\beta\wt{\varphi})_{\cent J M}\right)=\IBr\left((\beta_{H}\wt{\varphi}')_{\cent J M}\right)=\IBr\left(\sigma_J(\beta\wt{\varphi})_{\cent J M}\right)\]
and hence that $(A, M, \psi)\isoc(H, \norm M D, \psi')$.

We now show that the condition on block induction from Definition \ref{def:block-iso} is satisfied. For this purpose, we first show that $\bl(\wt{\varphi}_J)=\bl(\wt{\varphi}'_{J\cap H})^J$, or equivalently that $\bl(\wt{\vartheta}_J)=\bl(\wt{\vartheta}'_{J\cap H})^J$, for every $M\leq J\leq A$. Let $e:=[\vartheta_{\cent K D},\vartheta']\neq 0 \mod p$ and recall that, according to Lemma \ref{lem:scalars-centralizer}, the characters $\wt{\vartheta}$ and $\wt{\vartheta}'$ satisfy
\begin{equation}\label{eq:vartheta}
\wt{\vartheta}(x)^*=e^*\wt{\vartheta}'(x)^*
\end{equation}
for all $x\in H^0$ with $D\in\Syl_p(\cent M x)$. By \cite[Theorem 5.2]{Nav-Spa14I} we get 
\[\wt{\vartheta}(1)_{p'}\equiv e|K:\cent K D|_{p'}\wt{\vartheta}'(1)_{p'} \mod p\]
and, by using the fact that $|K:\cent K D|=|M:\norm M D|$ and applying \eqref{eq:vartheta}, we obtain
\[\left(\frac{|\norm M D|_{p'}\wt{\vartheta}'(x)}{\wt{\vartheta}'(1)_{p'}}\right)^*=\left(\frac{|M|_{p'}\wt{\vartheta}(x)}{\wt{\vartheta}(1)_{p'}}\right)^*,\]
for all $x\in H^0$ with $D\in{\rm Syl}_p(\cent M x)$.
Now, observe that $D$ is a common defect group of $\bl(\wt{\vartheta}_M)$ and $\bl(\wt{\vartheta}'_{\norm M D})$. In fact, by definition $D$ is a defect group of $\bl(\psi)$ and $\bl(\psi)=\bl(\wt{\vartheta}_M)$. Moreover, if $E$ is a defect group of $\bl(\wt{\vartheta}'_{\n_M(D)})$, then $D\leq \o_p(\n_M(D))\leq E$. On the other hand, by \cite[Theorem 9.17]{Nav98} we know that $E\cap \c_K(D)$ is a defect group of $\bl(\vartheta')$ and hence $E\cap K=1$. This shows that $E$ is a complement of $K$ in $M$. But so is $D$ (see \cite[Theorem 9.17]{Nav98}) and therefore $D=E$ is a common defect group of $\bl(\wt{\vartheta}_M)$ and $\bl(\wt{\vartheta}'_{\norm M D})$. It now follows from \cite[Lemma 4.2]{Nav-Spa14I} that $\bl(\wt{\vartheta}_J)=\bl(\wt{\vartheta}'_{J\cap H})^J$ as required. Finally, by \cite[Proposition 3.6]{Spa13I} we know that
\[\bl\left(\beta\wt{\varphi}_J\right)=\bl\left(\beta_{J\cap H}\wt{\varphi}'_{J\cap H}\right)^J=\bl\left(\sigma_J(\beta\wt{\varphi}_J)\right)^J\]
which proves that $(A, M, \psi)\isob(H, \norm M D, \psi')$. This concludes the first part of the proof.

Consider now $M\leq J\leq A$ such that $D$ is a radical $p$-subgroup of $J$ and let $\gamma\in\IBr(J\mid\psi)$. We wish to prove that $\gamma\in\rdzo(J\mid M,\varphi)$ if and only if $\gamma':=\sigma_J(\gamma)\in\dzo(\norm J D\mid \varphi')$. As in the previous paragraph, we may write $\gamma=\beta\wt{\varphi}_J$ and $\gamma'=\beta_{J\cap H}\wt{\varphi}'_{J\cap H}$ for some $\beta\in\IBr(J/M)=\IBr(J/K)$. Assume first that $\gamma'\in\dzo(\norm J D\mid \varphi')$ and recall that $D$ is a radical $p$-subgroup of $J$. Then, there exists $\chi\in\Irr(\norm J D)$ with $D\leq\ker\chi$ and such that $\chi^0=\gamma'$ and $\chi(1)_p=|\norm J D:D|_p$. Now, by Lemma \ref{lem:gallagher-lift} we can find some $\eta'\in\Irr(J\cap H/\n_M(D))$ such that $\eta'^0=\beta_{J\cap H}$. Furthermore via the isomorphism $J/M\simeq \norm J D/\norm M D$ we can write $\eta'=\eta_{J\cap H}$ for a unique $\eta\in\Irr(J/M)$. Observe then that $\eta^0=\beta$. Furthermore,
\[\beta(1)_p=\frac{\sigma_J(1)_p}{\wt{\varphi}'(1)_p}=\frac{|\norm J D:D|_p}{|\cent K D|_p}=|J:M|_p.\]
By applying Lemma \ref{lem:gallagher-lift} once again, we deduce now that there is some $\rho\in\Irr(J)$ with $\rho^0=\gamma$. Moreover, $\rho$ lies over some extension $\wt{\vartheta}\in\Irr(M)$ of $\vartheta$ and satisfies
\[\rho(1)_p=\gamma(1)_p=\beta(1)_p\wt{\varphi}(1)_p=|J:M|_p\wt{\vartheta}(1)_p\]
which implies that $\rho\in\rdz(J\mid\wt{\vartheta})$, so $\gamma\in\rdzo(J\mid M,\varphi)$. A similar argument shows that if $\gamma\in\rdzo(J\mid M,\varphi)$ then $\gamma'\in\dzo(\n_J(D)\mid \varphi')$ and the proof is now complete.
\end{proof}

We now extend Proposition \ref{pro:dgn-isomorphism-with-extension} to the case where $\varphi$ is $A$-invariant but does not necessarily extend to $A$. Before proving this result, we show that the central extension constructed in Lemma \ref{lem:4.1} preserves (relative) defect zero Brauer characters.

\begin{lem}
\label{lem:4.1 with dzo}
Consider the setting of Lemma \ref{lem:4.1} .
\begin{enumerate}
\item For every $K\leq J\leq G$ the restriction $\eps_{\wh{J}}:\wh{J}\to J$ maps $\dzo(\wh{J}\mid \vartheta_0\times 1_Z)$ onto $\dzo(J\mid\vartheta)$.
\item Let $K\leq U\leq J\leq G$ such that $ {U}\normal  {G}$ and ${U}/{K}$ is a $p$-group. Then the restriction $\eps_{\wh{J}}:\wh{J}\to J$ maps $\rdzo(\wh{J}\mid \wh{U},\vartheta_0\times 1_Z)$ onto $\rdzo(J\mid U, \vartheta)$.
\end{enumerate}
\end{lem}

\begin{proof}
Let $\varphi\in\IBr(\wh{J}\mid\vartheta_0\times 1_Z)$ and assume there is some $\chi\in\Irr(\wh{J})$ with $\chi^0=\varphi$. Then $\eps_{\wh{J}}(\chi)\in\Irr(J)$ satisfies $\eps_{\wh{J}}(\chi)^0=\eps_{\wh{J}}(\varphi)$. Furthermore, using that $Z$ is a central $p'$-group  we obtain that $\chi(1)_p=|\wh{J}:\oh{p}{\wh{J}}|_p$ if and only if $\eps_{\wh{J}}(\chi)(1)_p=|J:\oh{p}J|_p$, and then the first part of the statement follows. Similarly, noticing that $\chi(1)_p=|\wh{J}:\wh{U}|_p$ if and only if $\eps_{\wh{J}}(\chi)(1)_p=|J:U|_p$, we obtain the second part of the statement.
\end{proof}

Using the central extension constructed in Lemma \ref{lem:4.1} we obtain the following corollary as a direct consequence of Proposition \ref{pro:dgn-isomorphism-with-extension} and Lemma \ref{lem:4.1 with dzo}.

\begin{cor}
\label{cor:dgn-isomorphism-without-extension}
Assume Hypothesis \ref{hyp:DGN hypotheses} with $L=1$ and suppose that $M, K\normal A$ and that $\varphi$ is $A$-invariant. Consider $\vartheta'=\Pi_{D}(\vartheta)$ and $\varphi'=\pi_D(\varphi)=(\vartheta')^0$, and let $\psi\in\IBr(M)$ and $\psi'\in\IBr(\norm M D)$ be the unique characters lying above $\varphi$ and $\varphi'$ respectively. Then
\[(A,M,\psi)\isob (\n_A(D),\n_M(D),\psi').\]
Furthermore, if $M\leq J\leq A$ and $D$ is a radical $p$-subgroup of $J$ then the bijection
\[\sigma_J:\IBr\left(J\enspace\middle|\enspace\psi\right)\to\IBr\left(\n_J(D)\enspace\middle|\enspace\psi'\right)\] given by Theorem \ref{thm:strong-iso} maps $\rdzo(J\mid M,\varphi)$ onto $\dzo(\norm J D\mid\varphi')$.
\end{cor}

\begin{proof}
Let $\P$ be a projective representation of $A$ associated with $(A, M,\psi)$ and consider the central extension $\eps:\wh{A}\to A$ from Lemma \ref{lem:4.1}. Let $Z=\ker\eps$ and set $\wh{M}:=\eps^{-1}(M)$, so that $\wh{M}=M_0\times Z$, $M_0\cong M$ via the restriction $\eps_{M_0}$, and $\psi_0=\psi\circ\eps_{M_0}\in\IBr(M_0)$ extends to $\wh{A}$. By first applying Proposition \ref{pro:dgn-isomorphism-with-extension} in the group $\wh{A}$ and then using Lemma \ref{lem:4.1} (iv), we deduce that $(A,M,\psi)\isob (\n_A(D),\n_M(D),\psi')$. Moreover, using the second half of Proposition \ref{pro:dgn-isomorphism-with-extension} and applying Lemma \ref{lem:4.1 with dzo} (see also Remark \ref{rem:sigma-maps-central-extension}), we deduce that the bijection $\sigma_J$ maps $\rdzo(J\mid M,\varphi)$ onto $\dzo(\norm J D\mid\varphi')$ as required.
\end{proof}

We can finally prove Theorem \ref{thm:Above modular DGN}.

\begin{proof}[Proof of Theorem \ref{thm:Above modular DGN}]
Let $L:=\o_p(K)$ and define $\overline{H}=H/L$ for every $L\leq H\leq A$. Notice that $L$ is contained in $D$ and therefore $L\leq \n_K(D)$ so that $\overline{\n_K(D)}=\n_{\overline{K}}(\overline{D})=\c_{\overline{K}}(\overline{D})$. Now we can consider the Brauer characters $\overline{\varphi}\in\dzo(\overline{K})$ and $\overline{\varphi}'\in\dzo(\c_{\overline{K}}(\overline{D}))$ corresponding to $\varphi\in\dzo(K)$ and $\varphi':=\pi_D(\varphi)$ respectively via inflation of characters. Noticing that $\overline{\varphi}':=\pi_{\overline{D}}(\overline{\varphi})$ we can apply Corollary \ref{cor:dgn-isomorphism-without-extension}, with $J=\overline{G}$, to obtain an $\n_{\overline{A}}(\overline{D})$-equivariant bijection
\[\overline{\Delta}_{\overline{D},\overline{\varphi}}^{\overline{G}}:\rdzo\left(\overline{G}\enspace\middle|\enspace \overline{M},\overline{\varphi}\right)\to\dzo\left(\n_{\overline{G}}(\overline{D})\enspace\middle|\enspace \overline{\varphi}'\right)\]
that satisfies
\[\left(\overline{A}_{\overline{\chi}},\overline{G},\overline{\chi}\right)\isob\left(\n_{\overline{A}}(\overline{D})_{\overline{\chi}},\n_{\overline{G}}(\overline{D}),\Delta_{\overline{D},\overline{\varphi}}^{\overline{G}}(\overline{\chi})\right)\]
for every $\overline{\chi}\in\rdzo(\overline{G}\mid \overline{M},\overline{\varphi})$ thanks to the block isomorphism of modular character triples given by Corollary \ref{cor:dgn-isomorphism-without-extension} and applying Lemma \ref{lem:Bijections induced by isomorphisms are compatible with isomorphisms}. Then, since $L$ is contained in every character belonging either to $\rdzo(G\mid M,\varphi)$ or to $\dzo(\n_G(D)\mid \varphi')$, we deduce that by inflation of characters the bijection $\overline{\Delta}_{\overline{D},\overline{\varphi}}^{\overline{G}}$ induces a bijection $\Delta_{D,\varphi}^G$ with the required properties. To obtain the block isomorphism of modular character triples from the statement, we apply Lemma \ref{lem:Lifting isomorphisms from quotients}. This completes the proof.
\end{proof}

We conclude this section with a result of independent interest. This can be seen as a modular version of \cite[Corollary 11.3]{Lad10}.

\begin{cor}
\label{cor:DGN modular iso}
Assume Hypothesis \ref{hyp:DGN hypotheses}. Then the modular character triples $(A,K,\varphi)$ and $(\norm A D,\norm K D,\pi_D(\varphi))$ are strongly isomorphic in the sense of Theorem \ref{thm:strong-iso}.
\end{cor}

\begin{proof}
First, notice that since $M\norm A D =A$ and $M=KD\leq K\norm A D$ then we have $K\norm A D=M\norm A D=A$ so the group theoretical conditions of Theorem \ref{thm:strong-iso} are satisfied. Let $\P$ and $\P'$ be projective representations giving the isomorphism
\[(A,M,\psi)\isob(\norm A D,\norm M D,\psi')\]
from Theorem \ref{thm:Above modular DGN} and where $\psi$ and $\psi'$ are the extensions of $\varphi$ and $\pi_D(\varphi)$ to $M$ and $\norm M D$ respectively. Notice that $\P$ and $\P'$ are actually projective representations associated with $(A,K,\varphi)$ and $(\norm A D,\norm K D,\pi_D(\varphi))$. Now, using \cite[Theorem 8.14]{Nav98} it follows that the factor sets of $\P$ and $\P'$ coincide via the natural isomorphism $A/K\to\norm A D/\norm K D$. We can then conclude by applying Theorem \ref{thm:strong-iso}.
\end{proof}

\section{The reduction}
\label{sec:Reduction}

In this section, we finally show that Conjecture \ref{conj:Main, iBAWC} reduces to quasi-simple groups and hence prove Theorem \ref{thm:Main, Reduction}. For the reader's convenience, we restate Theorem \ref{thm:Main, Reduction} below. Recall that a simple group $S$ is said to be \textit{involved} in a finite group $G$ if there exist subgroups $N\unlhd H\leq G$ such that $H/N$ is isomorphic to $S$.

\begin{thm}
\label{thm:Reduction for iBAW}
Let $G$ be a finite group and suppose that Conjecture \ref{conj:iBAWC 2} holds for every covering group of any non-abelian finite simple group of order divisible by $p$ involved in $G$. Then Conjecture \ref{conj:iBAWC 2} holds for $G$.
\end{thm}

We prove Theorem \ref{thm:Reduction for iBAW} arguing inductively by considering a minimal counterexample. Suppose that $G$ is a finite group satisfying the assumptions of Theorem \ref{thm:Reduction for iBAW} and for which the conclusion fails. More precisely suppose that, Conjecture \ref{conj:iBAWC 2} fails with respect to $G\unlhd A$ and that $G$ and $A$ have been minimized with respect to $|G:\z(G)|$ first and then $|A|$. In this case, the pair $(G,A)$ satisfies the following hypothesis.

\begin{hyp}
\label{hyp:Inductive hypothesis}
Let $G\unlhd A$ be finite groups, with $G$ satisfying the requirements of Theorem \ref{thm:Reduction for iBAW}, and suppose that Conjecture \ref{conj:iBAWC 2} holds for every $X\unlhd Y$ such that every non-abelian finite simple group of order divisible by $p$ involved in $X$ is also involved in $G$, and at least one of the following conditions is satisfied:
\begin{enumerate}
\item $|X:\z(X)|<|G:\z(G)|$;
\item $|X:\z(X)|=|G:\z(G)|$ and $|Y|<|A|$.
\end{enumerate}
\end{hyp}

We now describe the structure of the minimal counterexample $G$. To start, we show that $G$ does not have non-trivial normal $p$-subgroups.

\begin{lem}
\label{lem:Removing normal p-subgroups}
If Conjecture \ref{conj:iBAWC 2} holds for $G/\o_p(G)\unlhd A/\o_p(G)$, then it holds for $G\unlhd A$.
\end{lem}

\begin{proof}
Set $\overline{A}:=A/\o_p(G)$ and $\overline{G}:=G/\o_p(G)$ and suppose that $\overline{\Omega}:\IBr(\overline{G})\to \Alp(\overline{G})/\overline{G}$ is the map given by Conjecture \ref{conj:iBAWC 2} applied to $\overline{G}\unlhd \overline{A}$. By \cite[Lemma 2.32]{Nav98} we know that $\o_p(G)$ is contained in the kernel of every irreducible Brauer character and therefore we can identify $\IBr(\overline{G})$ with $\IBr(G)$ via inflation of characters. On the other hand, if $(Q,\psi)$ is a $p$-weight of $G$ then $Q$ is a radical $p$-subgroup and hence $\o_p(G)$ is contained in $Q$ (see, for instance, \cite[Lemma 1.3]{Dad92}). Then, applying \cite[Lemma 2.32]{Nav98} we deduce that $\psi\in\dzo(\n_G(Q))$ can be identified with the corresponding character $\overline{\psi}\in\dzo(\n_{\overline{G}}(\overline{Q}))$ where $\overline{Q}:=Q/\o_p(G)$. It follows that the map $\overline{\Omega}$ induces an $A$-equivariant bijection $\Omega:\IBr(G)\to \Alp(G)/G$. To conclude, notice that the isomorphisms of modular character triples induced by $\overline{\Omega}$ can be lifted to analogous block isomorphisms for the map $\Omega$ thanks to Lemma \ref{lem:Lifting isomorphisms from quotients}.
\end{proof}

\begin{cor}
\label{cor:Trival p-core}
Suppose that Hypothesis \ref{hyp:Inductive hypothesis} holds for the pair $(G,A)$ while Conjecture \ref{conj:iBAWC 2} fails with respect to $G\unlhd A$. Then $\o_p(G)=1$. 
\end{cor}

\begin{proof}
Assume that $\o_p(G)\neq 1$, set $\overline{A}:=A/\o_p(G)$ and $\overline{G}:=G/\o_p(G)$. Then $|\overline{G}:\z(\overline{G})|\leq |G:\z(G)|$ and $|\overline{A}|<|A|$ so that Conjecture \ref{conj:iBAWC 2} holds for $\overline{G}\unlhd \overline{A}$ thanks to Hypothesis \ref{hyp:Inductive hypothesis}. Lemma \ref{lem:Removing normal p-subgroups} now implies that Conjecture \ref{conj:iBAWC 2} holds for $G\unlhd A$, a contradiction.
\end{proof}

By applying Corollary \ref{cor:Trival p-core} together with Corollary \ref{cor:iBAW for embedded perfect group}, we can now give a description of the structure of $G$.

\begin{pro}
\label{prop:Structure of a minimal counterexample}
Suppose that Hypothesis \ref{hyp:Inductive hypothesis} holds for the pair $(G,A)$ while Conjecture \ref{conj:iBAWC 2} fails with respect to $G\unlhd A$. Then, there exists a subgroup $K$ of $G$ with $K\unlhd A$ and $K\nleq \z(G)$ such that Conjecture \ref{conj:iBAWC 2} holds with respect to $K\unlhd A$.
\end{pro}

\begin{proof}
By Corollary \ref{cor:Trival p-core} we have $\o_p(G)=1$ and thus $\z(G)\leq \o_{p'}(G)$. If $\z(G)<\o_{p'}(G)$, then we define $K:=\o_{p'}(G)$ and observe that Conjecture \ref{conj:iBAWC 2} trivially holds for $K\unlhd A$. We may therefore assume that $\o_{p'}(G)=\z(G)$. Now, if $G$ is $p$-solvable it must be abelian and Conjecture \ref{conj:iBAWC 2} holds for $G$, against our assumptions. We conclude that $G$ has a non-abelian composition factor of order divisible by $p$. More precisely, $\z(G)<\F^*(G)$ and we can find a perfect subgroup $K\leq \F^*(G)$ with $K$ characteristic in $G$, hence normal in $A$, and such that $K/\z(K)$ is isomorphic to a direct product of copies of $S$ for some non-abelian simple group $S$ of order divisible by $p$. Observe also that $p$ does not divide the order of $\z(K)$ since $\o_p(G)=1$. Then, $K\nleq \z(G)$ and Conjecture \ref{conj:iBAWC 2} holds with respect to $K\unlhd A$ by our assumption and thanks to Corollary \ref{cor:iBAW for embedded perfect group}.
\end{proof}

If $K$ is the group given by Proposition \ref{prop:Structure of a minimal counterexample}, then by applying Theorem \ref{thm:Lifting bijection for weights} with $J=G$ we obtain an $A$-equivariant bijection between the sets $\IBr(G)$ and $\Alp(G\mid K)/G$ (see Definition \ref{def:Relative weights}) inducing block isomorphisms of modular character triples. Therefore, to conclude the proof of Theorem \ref{thm:Reduction for iBAW} we now need to construct a bijection with similar properties between the sets $\Alp(G\mid K)/G$ and $\Alp(G)/G$. We start by introducing some further notation.

\begin{defi}
\label{def:Relative weights}
We denote by $\Alpr(G\mid \vartheta)$ the set of pairs $(R,\chi)$ such that $R/K\in\Rad(G_\vartheta/K)$ and $\chi\in\rdzo(\n_G(R)\mid R, \vartheta)$. The group $G_\vartheta$ acts by conjugation on $\Alpr(G\mid \vartheta)$ and we let $\Alpr(G\mid \vartheta)/G_\vartheta$ denote the corresponding set of $G_\vartheta$-orbits.
\end{defi}

The argument used to prove the following lemma is inspired by \cite[Lemma 7.3]{Nav-Spa14I}.

\begin{lem}
\label{lem:Reduction, above a character in dzo}
Suppose that Hypothesis \ref{hyp:Inductive hypothesis} holds for the pair $(G,A)$ and let $K\leq G$ with $K\unlhd A$ and with an $A$-invariant $\vartheta\in\dzo(K)$. If $|G:K\z(G)|<|G:\z(G)|$, then there exists an $A$-equivariant bijection
\[\Upsilon_\vartheta^G:\IBr\left(G\enspace\middle|\enspace\vartheta\right)\to\Alpr(G\mid \vartheta)/G\]
such that
\[\left(A_\eta,G,\eta\right)\isob\left(\n_A(R)_\chi,\n_G(R),\chi\right)\]
for every $\eta\in\IBr(G\mid \vartheta)$ and $(R,\chi)\in\Upsilon_\vartheta^G(\eta)$.
\end{lem}

\begin{proof}
Let $\P$ be a projective representation associated with the modular character triple $(A,K,\vartheta)$ and consider the central extension $\wh{A}$ of $A$ by the $p'$-subgroup $Z$ constructed in Lemma \ref{lem:4.1}. Let $\eps:\wh{A}\to A$ be the epimorphism given by $\eps(x,z)=x$ for each $x\in A$ and $z\in Z$ and set $\wh{L}:=\eps^{-1}(L)$ for every $L\leq A$. Recall then that $\wh{K}=K_0\times Z$ for a subgroup $K_0\leq\wh{A}$ isomorphic to $K$ via the restriction $\eps_{K_0}$. Let $\vartheta_0:=\vartheta\circ \eps_{K_0}$ and observe that $\vartheta_0\in\dzo(K_0)$ and that, by Lemma \ref{lem:4.1}, $\vartheta_0$ has an extension $\wt{\vartheta}\in\IBr(\wh{A})$. Notice that $\wh{G}/\wh{K}$ is isomorphic to $G/K$ and, recalling that $\wh{K}/K_0\simeq Z$ is a $p'$-group, it follows that every non-abelian simple group of order divisible by $p$ involved in $\wh{G}/K_0$ is also involved in $G$. Furthermore, if we set $\overline{X}:=XK_0/K_0$ for every $X\leq \wh{A}$, then $|\overline{G}:\zent{\overline{G}}|\leq|G:K\zent G|<|G:\zent G|$ and therefore, by Hypothesis \ref{hyp:Inductive hypothesis}, Conjecture \ref{conj:iBAWC 2} holds for the pair $(\overline{G},\overline{A})$. Thus, there exists an $\overline{A}$-equivariant bijection
\[\Omega_{\overline{G}}:\IBr\left(\overline{G}\right)\to \Alp\left(\overline{G}\right)/\overline{G}\]
such that
\begin{equation}
\label{eq:Reduction, above a character in dzo, 1}
\left(\overline{A}_{\overline{\rho}}, \overline{G},\overline{\rho}\right)\isob\left(\norm{\overline{A}}{\overline{R}}_{\overline{\psi}}, \norm{\overline{G}}{\overline{R}},\overline{\psi}\right)
\end{equation}
for all $\overline{\rho}\in\IBr(\overline{G})$ and $(\overline{R}, \overline{\psi})\in\Omega_{\overline{G}}(\overline{\rho})$. Consider now $\tau\in\IBr(\wh{G}\mid\vartheta_0)$. Since $\vartheta_0$ extends to $\wt{\vartheta}\in\IBr(\wh{A})$, applying \cite[Corollary 8.20]{Nav98}, we can find a unique $\overline{\rho}\in\IBr(\overline{G})$ such that $\tau=\rho\wt{\vartheta}_{\wh{G}}$ and where $\rho$ is the inflation of $\overline\rho$ to $\wh{G}$. Similarly, given any pair $(\wh{R},{\varphi})\in\Alp_r(\wh{G}\mid\vartheta_0)$, we can find a unique $\overline{\psi}\in\IBr(\norm{\wh{G}}{\wh{R}}/K_0)$, with inflation $\psi\in\IBr(\norm{\wh{G}}{\wh{R}})$, such that $\varphi=\psi\wt{\vartheta}_{\norm{\wh{G}}{\wh{R}}}$. Since $\vartheta_0$ has defect zero (recall here that it is no loss of generality to assume that $\o_p(G)=1$ thanks to Lemma \ref{lem:Removing normal p-subgroups}), it lifts to an ordinary character of $K_0$. Then Lemma \ref{lem:gallagher-lift} implies that $\psi$ lifts to some ordinary character since so does $\varphi$. Moreover,
\[\psi(1)_p=\varphi(1)_p/\vartheta_0(1)_p=|\norm{\wh{G}}{\wh{R}}:\wh R|_p=|\norm{\overline{G}}{\overline{R}}:\overline{R}|_p\]
and hence we conclude that $(\overline{R},\overline{\psi})$ belongs to $\Alp(\overline{G})$. This shows that the map $\Omega_{\overline{G}}$ induces a bijection
\[\Upsilon_{\vartheta_0}^{\wh G}:\IBr\left(\wh{G}\enspace\middle|\enspace\vartheta_0\right)\to \Alp_r\left(\wh{G}\enspace\middle|\enspace\vartheta_0\right)/\wh{G}\]
given by sending the Brauer character $\rho\wt{\vartheta}_{\wh{G}}$ to the $\wh{G}$-orbit of the pair $(\wh{R},\psi\wt{\vartheta}_{\n_{\wh{G}}(\wh{R})})$ whenever $(\overline{R},\overline{\psi})$ belongs to the $\overline{G}$-orbit $\Omega_{\overline{G}}(\overline{\psi})$ as described above. Now, according to \cite[Corollary 1.5 (i.b)]{Mur96}, we can find a defect group $\overline{U}$ of $\bl(\overline{\psi})$ and a defect group $D$ of $\bl(\varphi)$ such that $\overline{U}\leq DK_0/K_0$. On the other hand, since $\wh{R}/K_0$ is a radical $p$-subgroup of $\wh{G}/K_0$, we deduce that $\wh{R}/K_0\leq \overline{U}$ thanks to \cite[Theorem 4.8]{Nav98}. It follows that $\c_{\wh{A}}(D)K_0/K_0\leq \n_{\wh{A}}(\wh{R})/K_0$ and therefore $\c_{\wh{A}_\rho}(D)\leq \n_{\wh{A}}(\wh{R})_\rho=\n_{\wh{A}}(\wh{R})_\psi$. We can now apply Lemma \ref{lem:4.6} to the block isomorphisms \eqref{eq:Reduction, above a character in dzo, 1} to get
\begin{equation}
\label{eq:Reduction, above a character in dzo, 2}
\left(\wh{A}_{\tau}, \wh{G},\tau\right)\isob\left(\norm{\wh{A}_\varphi}{\wh{R}}, \norm{\wh G}{\wh{R}},\varphi\right),
\end{equation}
where recall that $\tau=\rho\wt{\vartheta}_{\wh{G}}$ and $\varphi=\psi\wt{\vartheta}_{\n_{\wh{G}}(\wh{R})}$. Observe also that the bijection $\Upsilon_{\vartheta_0}^{\wh{G}}$ is $\wh{A}$-equivariant because the Brauer character $\wt{\vartheta}_{\wh G}$ is $\wh{A}$-invariant and using the equivariance properties of $\Omega_{\overline{G}}$. Next, for every $\tau\in\IBr(\wh{G}\mid\vartheta_0)$ and $(\wh{R},\varphi)\in\Upsilon_{\vartheta_0}^{\wh{G}}(\tau)$, we know by Definition \ref{def:central-iso} that $Z\leq \ker{\tau}$ if and only if $Z\leq \ker{\varphi}$. As a consequence, $\Upsilon_{\vartheta_0}^{\wh G}$ restricts to a bijection
\[\IBr\left(\wh G\enspace\middle|\enspace\vartheta_0\times 1_Z\right)\to \Alp_r\left(\wh G\enspace\middle|\enspace\vartheta_0\times 1_Z\right)/\wh{G}.\]
Finally, observe that $\norm{\wh{G}}{\wh{R}}=\widehat{\norm G R}$ and hence $\norm G R$ is isomorphic to $\norm{\wh{G}}{\wh{R}}/Z$. Therefore the epimorphism $\eps$ maps the character sets $\IBr(\wh{G}\mid\vartheta_0\times 1_Z)$ and $\Alp_r(\wh{G}\mid\vartheta_0\times 1_Z)/\wh{G}$ onto $\IBr(G\mid\vartheta)$ and $\Alp_r(G\mid\vartheta)/G$ respectively. In this way we can then construct an $A$-equivariant bijection 
\[\Upsilon^{G}_{\vartheta}:\IBr(G\mid\vartheta)\to\Alp_r(G\mid\vartheta)/G\]
as required in the statement. To prove the condition on block isomorphisms let $\eta\in\IBr(G\mid\vartheta)$ and $(R,\chi)\in\Upsilon_{\vartheta}^G(\eta)$. Since $\cent{\wh{A}}{\wh{G}}/Z=\cent A G$ by Lemma \ref{lem:4.1}, we can apply Lemma \ref{lem:going to quotients} to the block isomorphism \eqref{eq:Reduction, above a character in dzo, 2} in order to get
\[(A_\eta,G,\eta)\isob(\norm A R_\chi, \norm G R, \chi)\]
as desired.
\end{proof}

We now combine the bijections obtain in the lemma above for the various $\vartheta\in\dzo(K)$. Denote by $\Alpr(G\mid \dzo(K))$ and $\IBr(G\mid \dzo(K))$ the union of the sets $\Alpr(G\mid \vartheta)$ and $\IBr(G\mid \vartheta)$ respectively, for $\vartheta$ running in the set $\dzo(K)$. Since the set $\dzo(K)$ is stable under $G$-conjugation, we deduce that $G$ acts on the set $\Alpr(G\mid \dzo(K))$ and denote by $\Alpr(G\mid \dzo(K))/G$ the corresponding set of $G$-orbits.

\begin{pro}
\label{prop:Reduction, above any character in dzo}
Suppose that Hypothesis \ref{hyp:Inductive hypothesis} holds for the pair $(G,A)$ and let $K\leq G$ with $K\unlhd A$. If $|G:K\z(G)|<|G:\z(G)|$, then there exists an $A$-equivariant bijection
\[\Upsilon_K^G:\IBr\left(G\enspace\middle|\enspace\dzo(K)\right)\to\Alpr(G\mid \dzo(K))/G\]
such that
\[\left(A_\eta,G,\eta\right)\isob\left(\n_A(R)_\chi,\n_G(R),\chi\right)\]
for every $\eta\in\IBr(G\mid \vartheta)$ and $(R,\chi)\in\Upsilon_K^G(\eta)$.
\end{pro}

\begin{proof}
Let $\mathcal{U}$ be an $A$-transversal in the set $\dzo(K)$ and observe that for each $\vartheta\in\mathcal{U}$, by applying Lemma \ref{lem:Reduction, above a character in dzo} to $G_\vartheta\unlhd A_\vartheta$, there exists an $A_\vartheta$-equivariant bijection
\[\Upsilon_\vartheta^{G_\vartheta}:\IBr(G_\vartheta\mid \vartheta)\to \Alpr(G_\vartheta\mid \vartheta)/G_\vartheta\]
such that
\begin{equation}
\label{eq:Reductiom, above any character in dzo, 1}
\left(A_{\vartheta,\nu},G_\vartheta,\nu\right)\isob\left(\n_{A_\vartheta}(R)_\psi,\n_{G_\vartheta}(R),\psi\right)
\end{equation}
for every $\nu\in\IBr(G_\vartheta\mid \vartheta)$ and $(R,\psi)\in\Upsilon_\vartheta^{G_\vartheta}(\nu)$. Next, choose an $A_\vartheta$-transversal $\mathbb{S}_\vartheta$ in $\IBr(G_\vartheta\mid \vartheta)$ and observe that the equivariance properties of $\Upsilon_\vartheta^{G_\vartheta}$ imply that the image $\mathbb{T}_\vartheta:=\Upsilon_\vartheta^{G_\vartheta}(\mathbb{S}_\vartheta)$ is an $A_\vartheta$-transversal in $\Alpr(G_\vartheta\mid \vartheta)/G_\vartheta$. By using \cite[Theorem 8.9]{Nav98} we deduce that the set $\mathcal{S}_\vartheta$ of characters of the form $\eta=\nu^G$ with $\nu\in\mathbb{S}_\vartheta$ is an $A_\vartheta$-transversal in the set $\IBr(G\mid \vartheta)$. Similarly, if $\mathcal{T}_\vartheta$ denotes the set of $G$-orbits of pairs of the form $(R,\chi)$ with $\chi=\psi^{\n_G(R)}$ and where the $G_\vartheta$-orbit of $(R,\psi)$ belongs to $\mathbb{T}_\vartheta$, then $\mathcal{T}_\vartheta$ is an $A_\vartheta$-transversal in $\Alpr(G\mid \vartheta)/G$. Finally, \cite[Corollary 8.7]{Nav98} shows that the set $\mathcal{S}$ consisting of characters $\eta$ belonging to $\mathcal{S}_\vartheta$ for some $\vartheta\in \mathcal{U}$ is an $A$-transversal in $\IBr(G\mid \dzo(K))$. Likewise, the set $\mathcal{T}$ consisting of $G$-orbits $\overline{(R,\chi)}\in\mathcal{T}_\vartheta$ for some $\vartheta\in\mathcal{U}$ is an $A$-transversal in $\Alpr(G\mid \dzo(K))/G$. We then obtain an $A$-equivariant bijection $\Upsilon_K^G$ by setting
\[\Upsilon_K^G\left(\eta^x\right):=\overline{(R,\chi)}^x\]
for all $x\in A$ and every $\eta\in\mathcal{S}$ and $\overline{(R,\chi)}\in\mathcal{T}$ such that there exists some $\vartheta\in\mathcal{U}$ for which $\eta\in\mathcal{S}_\vartheta$, $\overline{(R,\chi)}\in\mathcal{T}_\vartheta$ and $(R,\psi)\in\Upsilon_\vartheta^{G_\vartheta}(\nu)$, where $\nu\in\IBr(G_\vartheta\mid \vartheta)$ and $\psi\in\IBr(\n_G(R)_\vartheta\mid \vartheta)$ are the Clifford correspondents of $\eta$ and $\chi$ respectively.

It remains to prove the condition on block isomorphisms of modular character triples from the statement. Let $\eta$, $\nu$, $\chi$, and $\psi$ be as in the previous paragraph and observe that by Lemma \ref{lem:Basic properties} (ii) it is enough to show that
\[\left(A_\eta,G,\eta\right)\isob\left(\n_A(R)_\chi,\n_G(R),\chi\right)\]
for our choice of $\eta$ and $(R,\chi)$. This condition follows by applying Lemma \ref{lem:Irreducible induction} since the block isomorphism of modular character triples from \eqref{eq:Reductiom, above any character in dzo, 1} is satisfied with respect to our choice of $\nu$ and $\psi$. This completes the proof.
\end{proof}

Using Proposition \ref{prop:Reduction, above any character in dzo} we can now construct a bijection from $\Alp(G\mid K)/G$ to an intermediate set $\Wr(G\mid K)/G$ that we now define. Recall that for any finite group $H$ we denote by $\Rado(H)$ the set of radical $p$-subgroups $Q$ of $H$ such that $(Q,\psi)\in\Alp(H)$ for some $\psi\in\dzo(\n_H(Q))$.

\begin{defi}
\label{def:Relative weights triples}
Let $K\unlhd G$ be finite groups. We denote by $\Wr(G\mid K)$ the set of triples $(Q,R,\chi)$ where $Q$ is a radical $p$-subgroup in $\Rado(K)$ and the pair $(R,\chi)$ belongs to the set $\Alpr(\n_G(Q)\mid \dzo(\n_K(Q)))$ introduced in Definition \ref{def:Relative weights} (see also the comment before Proposition \ref{prop:Reduction, above any character in dzo}). Once again the group $G$ acts by conjugation on $\Wr(G\mid K)$ and we denote by $\Wr(G\mid K)/G$ the set of $G$-orbits.
\end{defi}

We now construct a bijection between the set $\Alp(G\mid K)/G$ (from Definition \ref{def:Above Weights}) and the set $\Wr(G\mid K)/G$.

\begin{thm}
\label{thm:Reduction, inductive step}
Suppose that Hypothesis \ref{hyp:Inductive hypothesis} holds for the pair $(G,A)$ and let $K$ be a subgroup of $G$ with $K\unlhd A$ and $K\nleq \z(G)$. Then there exists an $A$-equivariant bijection
\[\Psi_K^G:\Alp(G\mid K)/G\to\Wr(G\mid K)/G\]
such that
\begin{equation}
\label{eq:Reduction, inductive step}
\left(\n_A(Q)_\eta,\n_G(Q),\eta\right)\isob\left(\n_A(Q,R)_\chi,\n_G(Q,R),\chi\right)
\end{equation}
for every $(Q,\eta)\in\Alp(G\mid K)$ and every $(Q,R,\chi)\in\Wr(G\mid K)$ whose $G$-orbits correspond via the bijection $\Psi_K^G$.
\end{thm}

\begin{proof}
To start, let $\mathbb{T}$ be an $A$-transversal in $\Rado(K)$ and observe that, for every $Q\in\mathbb{T}$, we have $|\n_G(Q):\n_K(Q)\z(\n_G(Q))|\leq |\n_G(Q):\n_K(Q)\z(G)|\leq |G:K\z(G)|<|G:\z(G)|$. We can then apply (the argument of) Proposition \ref{prop:Reduction, above any character in dzo} to $\n_G(Q)\unlhd \n_A(Q)$ to construct an $\n_A(Q)$-equivariant bijection
\[\Upsilon_{\n_K(Q)}^{\n_G(Q)}:\IBr\left(\n_G(Q)\enspace\middle|\enspace\dzo(\n_K(Q))\right)\to\Alpr(\n_G(Q)\mid \dzo(\n_K(Q)))/\n_G(Q)\]
such that
\[\left(\n_A(Q)_\eta,\n_G(Q),\eta\right)\isob\left(\n_A(Q,R)_\chi,\n_G(Q,R)_\chi,\chi\right)\]
for every $\eta\in\IBr(\n_G(Q)\mid \dzo(\n_K(Q)))$ and $(R,\chi)\in\Upsilon_{\n_K(Q)}^{\n_G(Q)}(\eta)$. Next, let $\mathbb{S}_Q$ be an $\n_A(Q)$-transversal in $\IBr(\n_G(Q)\mid \dzo(\n_K(Q)))$ and consider its image $\mathcal{S}_Q$ under the bijection $\Upsilon_{\n_K(Q)}^{\n_G(Q)}$. Since the latter is $\n_A(Q)$-equivariant, it follows that $\mathcal{S}_Q$ is an $\n_A(Q)$-transversal in $\Alpr(\n_G(Q)\mid \dzo(\n_K(Q)))/\n_G(Q)$. Observe that the set $\mathbb{S}$ consisting of $G$-orbits of pairs $(Q,\eta)$ with $Q\in\mathbb{T}$ and $\eta\in\mathbb{S}_Q$ is an $A$-transversal in $\Alp(G\mid K)/G$. Similarly, the set $\mathcal{S}$ consisting of $G$-orbits of triples $(Q,R,\chi)$ with $Q\in\mathbb{T}$ and where the $\n_G(Q)$-orbit of $(R,\chi)$ belongs to $\mathcal{S}_Q$, more precisely $(R,\chi)\in\Upsilon_{\n_K(Q)}^{\n_G(Q)}(\eta)$, is an $A$-transversal in $\Wr(G\mid K)/G$. Our construction shows that there is a bijection $\wh{\Psi}_K^G:\mathbb{S}\to \mathcal{S}$ and we can therefore define an $A$-equivariant bijection $\Psi_K^G$ as required in the statement above by setting
\[\Psi_K^G\left(\overline{(Q,\eta)}^x\right):=\wh{\Psi}_K^G\left(\overline{(Q,\eta)}\right)^x\]
for every $G$-orbit $\overline{(Q,\eta)}\in\mathbb{S}$ and every $x\in A$. Observe that the desired block isomorphisms of modular character triples are given directly by the properties of the bijections $\Upsilon_{\n_K(Q)}^{\n_G(Q)}$.
\end{proof}

In our final step we use the results on the Dade--Glauberman--Nagao correspondence obtained in Section \ref{sec:DGN}, and in particular Theorem \ref{thm:Above modular DGN}, to construct a bijection between $\Wr(G\mid K)/G$ and $\Alp(G)/G$. Observe that this step of our proof is independent on the inductive hypothesis and holds in full generality.

\begin{thm}
\label{thm:Reduction, relative weights to weights}
Let $K\leq G\leq A$ be finite groups with $K,G\unlhd A$. Then, there exists an $A$-equivariant bijection
\[\Lambda_K^G:\Wr(G\mid K)/G\to\Alp(G)/G\] 
such that
\[\left(\n_A(Q,R)_\chi,\n_G(Q,R),\chi\right)\isob\left(\n_A(D)_\nu,\n_G(D),\nu\right)\]
for every $(Q,R,\chi)\in\Wr(G\mid K)$ and every $(D,\nu)\in\Lambda_K^G(\overline{(Q,R,\chi)})$.
\end{thm}

\begin{proof}
To start fix $(Q,R,\chi)\in\Wr(G\mid K)$. Recall that this means that $Q$ is a radical $p$-subgroup of $K$ and there is some $\psi\in\dzo(\n_K(Q))$ such that $R/\n_K(Q)$ is a radical $p$-subgroup of $\n_G(Q)_\psi/\n_K(Q)$ and $\chi\in\rdzo(\n_G(Q,R)\mid R,\psi)$. Observe that $\psi$ is uniquely determined by $(Q,R,\chi)$ up to $\n_G(Q,R)$-conjugation. Now let $D/Q$ be a defect group of the unique block of $R/Q$ that covers $\bl(\overline{\psi})$ and where $\overline{\psi}$ is the Brauer character of $\n_K(Q)/Q$ corresponding to $\psi$ via inflation. By \cite[Theorem 9.17]{Nav98} we know that $R=D\n_K(Q)$ and $Q=D\cap \n_K(Q)=D\cap K$. In particular $\n_K(D)=\n_{\n_K(Q)}(D)$ and, because $\psi$ is $D$-invariant, we can define the Brauer character $\xi:=\pi_D(\psi)\in\dzo(\n_K(D))$ as in Definition \ref{def:DGN for Brauer characters}. Next, notice that the Brauer character $\chi$ determines a unique $\chi_\psi\in\IBr(\n_G(Q,R)_\psi)$ lying above $\psi$ according to \cite[Theorem 8.9]{Nav98}. Moreover, using the fact that $\chi\in\rdzo(\n_G(Q,R)\mid R,\psi)$, we deduce that $\chi_\psi$ belongs to the set $\rdzo(\n_G(Q,R)_\psi\mid R,\psi)$. Notice furthermore that $\n_{\n_G(Q,R)_\psi}(D)=\n_G(D)_\xi$ and that a Frattini argument yields $\n_G(Q,R)_\psi=R\n_G(D)_\xi=\n_K(Q)\n_G(D)_\xi$. Now, if
\[\Delta_{D,\psi}^{\n_G(Q,R)_\psi}:\rdzo(\n_G(Q,R)_\psi\mid R,\psi)\to \dzo(\n_G(D)_\xi\mid \xi)\]
denotes the bijection given by Theorem \ref{thm:Above modular DGN}, then we define $\nu_\xi:=\Delta_{D,\psi}^{\n_G(Q,R)_\psi}(\chi_\psi)$ and set $\nu:=(\nu_\xi)^{\n_G(D)}$ which is an irreducible Brauer character belonging to $\dzo(\n_G(D))$ (see \cite[Theorem 8.9]{Nav98}). We now use the above argument to construct an equivariant map $\Lambda_K^G$ with the properties required in the statement above. Observe that because $\Delta_{D,\psi}^{\n_G(Q,R)_\psi}$ is not canonical, our construction will have to depend on some choices. We now make these choices explicit.

Pick an $A$-transversal $\mathbb{T}$ in the set of radical $p$-subgroups of $K$ and, for each $Q\in\mathbb{T}$, choose an $\n_A(Q)$-transversal $\mathbb{T}^{(Q)}$ in the set of Brauer characters $\dzo(\n_K(Q))$. Now, for each $\psi\in\mathbb{T}^{(Q)}$ we fix an $\n_A(Q)_\psi$-transversal $\mathbb{T}^{(Q,\psi)}$ in the set of radical $p$-subgroups $\Rad(\N_G(Q)_\psi/\n_K(Q))$. Each element of $\mathbb{T}^{(Q,\psi)}$ can be written as $R/\n_K(Q)$ for some $\n_K(Q)\leq R\leq \n_G(Q)_\psi$. Moreover, if $D/Q$ is a defect group of the unique block of $R/Q$ covering $\bl(\overline{\psi})$ with $\overline{\psi}$ the Brauer character of $\n_K(Q)/Q$ corresponding to $\psi$ via inflation, then \cite[Theorem 9.17]{Nav98} tells us that $R=D\n_K(Q)$ and $Q=D\cap K$. Observe that $D$ is uniquely determined up to $\n_K(Q)$-conjugation. In order to fix a choice of $D$, for every $R/\n_K(Q)\in\mathbb{T}^{(Q,\psi)}$ we pick an $\n_A(Q,R)_\psi$-transversal $\mathbb{T}^{(Q,\psi,R)}$ in the set of defect groups of the unique block of $R/Q$ covering the block of $\overline{\psi}$. Finally, for each $D\in \mathbb{T}^{(Q,\psi,R)}$, observe that $\n_A(D)_\psi=\n_A(Q,R,D)_\psi$ and choose an $\n_A(D)_\psi$-transversal $\mathbb{T}^{(Q,\psi,R,D)}$ in $\rdzo(\n_G(Q,R)\mid R,\psi)$. Then, the set $\mathcal{T}$ of $G$-orbits $\overline{(Q,R,\chi)}$ with $Q\in\mathbb{T}$ and where $R/\n_K(Q)\in\mathbb{T}^{(Q,\psi)}$ and $\chi\in\mathbb{T}^{(Q,\psi,R,D)}$, for some $\psi\in\mathbb{T}^{(Q)}$ and $D\in\mathbb{T}^{(Q,\psi,R)}$, is an $A$-transversal in $\Wr(G\mid K)/G$.

We claim that the set $\mathbb{S}$ of groups $D$ belonging to $\mathbb{T}^{(Q,\psi,R)}$ for some $Q\in\mathbb{T}$, $\psi\in\mathbb{T}^{(Q)}$, and $R\in\mathbb{T}^{(Q,\psi)}$ is an $A$-transversal in $\Rado(G)$. To see this, observe first that two distinct elements of $\mathbb{S}$ cannot be $A$-conjugate. On the other hand, let $D'\in\Rado(G)$ and notice that $Q':=D'\cap K$ is a radical $p$-subgroup of $K$ by \cite[Lemma 2.3(a)]{Nav-Tie11}. Then there exist $Q\in\mathbb{T}$ and $x\in A$ such that $Q'^a=Q$. Recall that by the definition of $\Rado(G)$ there exists $\nu'\in\dzo(\n_G(D'))$ and that we can see $\bl(\nu')$ as a block of defect zero in the quotient $\n_G(D')/D'$. Let $c$ be a block of $\n_{KD'}(D')/D'$ covered by $\bl(\nu')$. By \cite[Lemma 9.27]{Nav98} it follows that $c$ is a block of defect zero of $\n_{KD'}(D')/D'$. Moreover, let $\wh{\xi}\in\IBr(\n_{KD'}(D'))$ correspond to the unique Brauer character of $\n_{KD'}(D')/D'$ belonging to $c$ (see \cite[Theorem 3.18]{Nav98}). Since $\n_{KD'}(D')/D'\simeq \n_K(D')/Q'$ we deduce that the character $\wh{\xi}$ restricts irreducibly to $\xi\in\IBr(\n_K(D'))$ and its block $\bl(\xi)$ has defect zero when regarded as a block of $\n_K(D')/Q'$. Thus $\xi$ belongs to $\dzo(\n_K(D'))$ and is $D'$-invariant. Since $\n_K(D')=\n_K(Q',D')$, we can then write $\xi=\pi_{D'}(\psi')$ for a unique $\psi'\in\dzo(\n_K(Q'))$. It follows that $\psi'^a$ belongs to $\dzo(\n_K(Q))$ and there exist $\psi\in\mathbb{T}^{(Q)}$ and $x\in\n_A(Q)$ such that $\psi'^{ax}=\psi$. Define $R':=\n_{KD'}(Q')=\n_K(Q')D'$ and observe that $R'/\n_K(Q')$ is a radical $p$-subgroup of $\n_G(Q')_{\psi'}/\n_K(Q')$ and that $D'/Q'$ is a defect group of the unique block of $R'/Q'$ covering $\bl(\psi')$, where the latter is regarded as a block of $\n_K(Q')/Q'$ . As a consequence $R'^{ax}/\n_K(Q)$ is a radical $p$-subgroup of $\n_G(Q)_\psi/\n_K(Q)$ and we can find $R/\n_K(Q)\in\mathbb{T}^{(Q,\psi)}$ and $y\in\n_A(Q)_\psi$ such that $R'^{axy}=R$. Moreover, $D'^{axy}/Q$ is a defect group of the unique block of $R/Q$ covering $\bl(\psi)$, viewed as a block of $\n_K(Q)/Q$, and hence there is $D\in\mathbb{T}^{(Q,\psi,R)}$ and $z\in\n_A(Q,R)_\psi$ such that $D'^{axyz}=D$. This shows that $D'$ is $A$-conjugate to $D\in\mathbb{S}$ and so $\mathbb{S}$ is an $A$-transversal in $\Rado(G)$ as claimed.

We now want to construct an $A$-transversal in $\Alp(G)/G$ in bijection with $\mathcal{T}$. By \cite[Theorem 8.9]{Nav98} each $\chi\in\mathbb{T}^{(Q,\psi,R,D)}$ is induced by a unique Brauer character $\chi_\psi\in\IBr(\n_G(Q,R)_\psi\mid \psi)$. As explained in the previous paragraph, we have that $\chi_\psi\in\rdzo(\n_G(Q,R)_\psi\mid R, \psi)$ and it follows that the set $\wh{\mathbb{T}}^{(Q,\psi,R,D)}$ of all Brauer characters $\chi_\psi$ for $\chi\in\mathbb{T}^{(Q,\psi,R,D)}$ is an $\n_A(D)_{\psi}$-transversal in $\rdzo(\n_G(Q,R)_\psi\mid R,\psi)$. Since the map $\Delta_{D,\psi}^{\n_G(Q,R)_\psi}$ from Theorem \ref{thm:Above modular DGN} is $\n_A(D)_\psi$-equivariant and $\n_A(D)_\psi=\n_A(D)_{\pi_D(\psi)}$, it follows that the image $\wh{\mathbb{S}}^{(Q,\psi,R,D)}$ of $\wh{\mathbb{T}}^{(Q,\psi,R,D)}$ under $\Delta_{D,\psi}^{\n_G(Q,R)_\psi}$ is an $\n_A(D)_\psi$-transversal in $\dzo(\n_G(D)_{\pi_D(\psi)}\mid \pi_D(\psi))$. As before, by applying \cite[Theorem 9.14]{Nav98} we deduce that the set $\mathbb{S}^{(Q,\psi,R,D)}$ of Brauer characters $\nu:=\vartheta^{\n_G(D)}$ for $\vartheta\in\wh{\mathbb{S}}^{(Q,\psi,R,D)}$ is an $\n_A(D)_\psi$-transversal in the set $\dzo(\n_G(D)\mid \pi_D(\psi))$. We can now conclude that the set $\mathcal{S}$ of $G$-orbits $\overline{(D,\nu)}$ with $D\in\mathbb{T}^{(Q,\psi,R)}$ and $\nu\in\mathbb{S}^{(Q,\psi,R,D)}$, for some $Q\in\mathbb{T}$, $\psi\in\mathbb{T}^{(Q)}$, and $R\in\mathbb{T}^{(Q,\psi)}$, is an $A$-transversal in $\Alp(G)/G$. In addition, there is a bijection between $\mathcal{T}$ and $\mathcal{S}$ given by mapping the $G$-orbit of $(Q,R,\chi)$ to that of $(D,\nu)$ whenever $Q\in\mathbb{T}$ and there is some $\psi\in\mathbb{T}^{(Q)}$ such that $R/\n_K(Q)\in\mathbb{T}^{(Q,\psi)}$, $D\in\mathbb{T}^{(Q,\psi,R)}$, $\chi\in\mathbb{T}^{(Q,\psi,R,D)}$ and $\nu\in\mathbb{S}^{(Q,\psi,R,D)}$ with $\chi$ corresponding to $\nu$ as described above. We define the map $\Lambda_K^G$ by setting
\[\Lambda_K^G\left(\overline{(Q,R,\chi)}^x\right):=\overline{(D,\nu)}^x\]
for every $\overline{(Q,R,\chi)}\in\mathcal{T}$ corresponding to $\overline{(D,\nu)}\in\mathcal{S}$ and every $x\in A$. This defines an $A$-equivariant bijection between $\Wr(G\mid K)/G$ and $\Alp(G)/G$. To conclude, we need to show that
\begin{equation}
\label{eq:Reduction, relative weights to weights}
\left(\n_A(Q,R)_\chi,\n_G(Q,R),\chi\right)\isob\left(\n_A(D)_\nu,\n_G(D),\nu\right).
\end{equation}
First, let $\chi_\psi\in\wh{\mathbb{T}}^{(Q,\psi,R,D)}$ and $\nu_\psi\in\wh{\mathbb{S}}^{(Q,\psi,R,D)}$ such that $\chi=(\chi_\psi)^{\n_G(Q,R)}$ and $\nu=(\nu_\psi)^{\n_G(D)}$. By construction, we know that $\nu_\psi$ is the image of $\chi_\psi$ under the bijection $\Delta_{D,\psi}^{\n_G(Q,R)_\psi}$ and hence Theorem \ref{thm:Above modular DGN} yields
\[\left(\n_A(Q,R)_{\psi,\chi_\psi},\n_G(Q,R)_\psi,\chi_\psi\right)\isob\left(\n_A(D)_{\psi,\nu_\psi},\n_G(D)_\psi,\nu_\psi\right)\]
from which \eqref{eq:Reduction, relative weights to weights} follows thanks to Lemma \ref{lem:Irreducible induction}. Observe that the latter can be applied since $\n_A(Q,R)_\chi=\n_G(Q,R)\n_A(Q,R)_{\psi,\chi}$ by a Frattini argument and using Clifford's theorem \cite[Corollary 8.7]{Nav98}. This concludes the proof.
\end{proof}

Finally, we can prove Theorem \ref{thm:Reduction for iBAW} as a consequence of Theorem \ref{thm:Lifting bijection for weights}, Proposition \ref{prop:Structure of a minimal counterexample}, Theorem \ref{thm:Reduction, inductive step}, and Theorem \ref{thm:Reduction, relative weights to weights}.

\begin{proof}[Proof of Theorem \ref{thm:Reduction for iBAW}]
We consider a counterexample $G$ to Theorem \ref{thm:Reduction for iBAW} and assume that Conjecture \ref{conj:iBAWC 2} fails to hold for a choice $G\unlhd A$. We further assume that $G$ and $A$ have been minimised with respect to $|G:\z(G)|$ first and then $|A|$. As explained at the beginning of this section, it follows that Hypothesis \ref{hyp:Inductive hypothesis} holds for the pair $(G,A)$. By Proposition \ref{prop:Structure of a minimal counterexample} there exists a subgroup $K$ of $G$ with $K\unlhd A$ and $K\nleq \z(G)$ such that Conjecture \ref{conj:iBAWC 2} holds for $K\unlhd A$. Now, we can apply Theorem \ref{thm:Lifting bijection for weights} with $J=G$ to obtain an $A$-equivariant bijection
\[\Omega_K^G:\IBr(G)\to \Alp(G\mid K)/G\]
such that
\begin{equation}
\label{eq:Reduction for iBAW 1}
\left(A_\varphi,G,\varphi\right)\isob\left(\n_A(Q)_\eta,\n_G(Q),\eta\right)
\end{equation}
for every $\varphi\in\IBr(G)$ and $(Q,\eta)\in\Omega_K^G(\varphi)$. On the other hand, combining the bijections $\Psi_K^G$ and $\Lambda_K^G$ given by Theorem \ref{thm:Reduction, inductive step} and Theorem \ref{thm:Reduction, relative weights to weights} respectively, we obtain an $A$-equivariant bijection
\[\Lambda_K^G\circ\Psi_K^G:\Alp(G\mid K)/G\to\Wr(G\mid K)/G\to \Alp(G)/G\]
such that
\begin{equation}
\label{eq:Reduction for iBAW 2}
\left(\n_A(Q)_\eta,\n_G(Q),\eta\right)\isob\left(\n_A(Q,R)_\chi,\n_G(Q,R)_\chi,\chi\right)\isob\left(\n_A(D)_\nu,\n_G(D)_\nu,\nu\right)
\end{equation}
whenever $(Q,R,\chi)\in\Psi_K^G(\overline{(Q,\eta)})$ and $(D,\nu)\in\Lambda_K^G(\overline{(Q,R,\chi)})$. We conclude that the map $\Omega:=\Lambda_K^G\circ\Psi_K^G\circ\Omega_K^G$ satisfies the requirements of Conjecture \ref{conj:iBAWC 2} by \eqref{eq:Reduction for iBAW 1} and \eqref{eq:Reduction for iBAW 2} thanks to the transitivity of the relation $\isob$. This contradicts the choice of $G$ and $A$ and the proof is now complete.
\end{proof}

\subsection{A reduction in the block-free case}
\label{sec:Reduction for block-free version}

When proving that two modular character triple isomorphisms are block isomorphic it is necessary, in particular, to show that they are also central isomorphic. With this in mind, an inspection of the proofs of the lemmas in Section \ref{sec:Isomorphisms of modular character triples} shows that all those statements admit a version where the block isomorphisms are replaced by central isomorphisms. Similarly, we can state a block-free version of Conjecture \ref{conj:iBAWC 2} as follows.

\begin{conj}
\label{conj:iAWC}
Let $G\unlhd A$ be finite groups. Then there exists an $A$-equivariant bijection
\[\Omega:\IBr(G)\to\Alp(G)/G\]
such that
\[\left(A_\vartheta,G,\vartheta\right)\isoc\left(\n_A(Q)_\psi,\n_G(Q),\psi\right)\]
for every $\vartheta\in\IBr(G)$ and $(Q,\psi)\in\Omega(\vartheta)$.
\end{conj}

Proceeding as in Section \ref{sec:Reduction} we can then obtain the following block-free version of Theorem \ref{thm:Main, Reduction}.

\begin{thm}
\label{thm:Reduction, block-free version}
Let $G$ be a finite group and $p$ a prime number. Suppose that Conjecture \ref{conj:iAWC} holds at the prime $p$ for every covering group of any non-abelian finite simple group of order divisible by $p$ involved in $G$. Then Conjecture \ref{conj:iAWC} holds for $G$ at the prime $p$.
\end{thm}

\begin{proof}
The proofs of Lemma \ref{lem:Removing normal p-subgroups}, Corollary \ref{cor:Trival p-core}, and Proposition \ref{prop:Structure of a minimal counterexample} show that in a minimal counterexample $G\unlhd A$ we can find a normal subgroup $K$ of $A$ contained in $G$ with $K\nleq \z(G)$ and such that Conjecture \ref{conj:iAWC} holds for $K\unlhd A$. The argument used to prove Theorem \ref{thm:Lifting bijection for weights} now yields an $A$-equivariant bijection
\[\Omega_K^G:\IBr(G)\to \Alp(G\mid K)/G\]
such that
\[\left(A_\chi,G,\chi\right)\isoc\left(\n_A(Q)_\eta,\n_G(Q),\eta\right)\]
for every $\chi\in\IBr(G)$ and $(Q,\eta)\in\Omega_K^G(\chi)$. Next, proceeding as in Lemma \ref{lem:Reduction, above a character in dzo}, Proposition \ref{prop:Reduction, above any character in dzo}, and Theorem \ref{thm:Reduction, inductive step} we construct an $A$-equivariant bijection
\[\Psi_K^G;\Alp(G\mid K)/G\to \Wr(G\mid K)/G\]
such that
\[\left(\n_A(Q)_\eta,\n_G(Q),\eta\right)\isoc\left(\n_A(Q,R)_\chi,\n_G(Q,R),\chi\right)\]
for every $(Q,\eta)\in\Alp(G\mid K)$ and every $(Q,R,\chi)\in\Wr(G\mid K)$ whose $G$-orbits correspond via the bijection $\Psi_K^G$. Finally, recalling that block isomorphisms of modular character triples are, in particular, central isomorphisms, we conclude by combining the above bijections $\Omega_K^G$ and $\Psi_K^G$ and applying Theorem \ref{thm:Reduction, relative weights to weights}.
\end{proof}

\section{Application to Navarro's Conjecture}
\label{sec:Navarro's conjecture}

In \cite{Nav17}, Navarro introduced a new conjecture (see \cite[Conjecture E]{Nav17}) that unifies the Alperin Weight Conjecture and the Glauberman correspondence into a single statement. In our paper we are mainly interested in the blockwise version of this statement that was introduced in Conjecture \ref{conj:Main, Gabriel Conjecture} and which we recall below. Recall that if $G\unlhd \Gamma$ and $B$ is a block of $G$, then we denote by $\IBr_\Gamma(B)$ the set of $\Gamma$-invariant irreducible Brauer characters of $G$ that belongs to $B$.

\begin{conj}[Navarro]
\label{conj:Blockwise Conjecture E}
Let $G\unlhd \Gamma$ be finite groups with $\Gamma/G$ a $p$-group. For every $\Gamma$-invariant $p$-block $B$ of $G$, we have
\begin{equation}
\label{eq:Blockwise Conjecture E}
\left|\IBr_\Gamma(B)\right|=\sum\limits_{Q\in\Theta_B/\Gamma}\left|\dz\left(\n_\Gamma(Q)/Q\middle| B\right)\right|
\end{equation}
where $\Theta_B$ is the set of $p$-subgroups $Q$ of $\Gamma$ such that $\Gamma=GQ$ and $Q\cap G$ is contained in some defect group of the block $B$, and $\dz(\n_\Gamma(Q)/Q\mid B)$ is the set of irreducible characters $\overline{\vartheta}\in\dz(\n_\Gamma(Q)/Q)$ such that $\bl(\vartheta)^\Gamma$ covers $B$ and where $\vartheta\in\irr{\n_G(Q)}$ corresponds to $\overline{\vartheta}$ via inflation of characters.
\end{conj}

As for \cite[Conjecture E]{Nav17}, the above statement unifies into a single statement both the blockwise Alperin Weight Conjecture and the Dade--Glauberman--Nagao correspondence. In fact, Conjecture Conjecture \ref{conj:Blockwise Conjecture E} becomes the blockwise Alperin Weight Conjecture when $G=\Gamma$, and implies the count of the Dade--Glauberman--Nagao correspondence (see \cite[Theorem 4.1]{Nav-Tie11}) when considering blocks of defect zero. We prove the latter implication in the following lemma.

\begin{lem}
\label{lem:Navarro implies DGN}
Let $G\unlhd \Gamma$ be finite groups with $\Gamma/G$ a $p$-group and consider a (possibly empty) set $\mathcal{S}$ of representatives for the $\Gamma$-conjugacy classes of complements of $G$ in $\Gamma$. If Conjecture \ref{conj:Blockwise Conjecture E} holds for every $\Gamma$-invariant block of defect zero of $G$, then
\[\left|\dz_\Gamma(G)\right|=\sum\limits_{Q\in\mathcal{S}}\left|\dz(\c_G(Q))\right|.\]
\end{lem}

\begin{proof}
First, notice that the number of $\Gamma$-invariant defect zero characters of $G$ coincides with the number of $\Gamma$-invariant irreducible Brauer characters belonging to some block $B$ of defect zero of $G$, that is
\begin{equation}
\label{eq:Conjecture blockwise implies DGN, 1}
\left|\dz_\Gamma(G)\right|=\sum\limits_{B}\left|\IBr_\Gamma(B)\right|
\end{equation}
where the sum runs over all $\Gamma$-invariant blocks $B$ of defect zero of $G$. On the other hand, for each such block $B$, observe that $\mathcal{S}$ is a representative set for the $\Gamma$-orbits on $\Theta_B$ as defined in Conjecture \ref{conj:Blockwise Conjecture E}. In particular, if $Q\in\Theta_B$, then we have $\n_\Gamma(Q)=\c_G(Q)\times Q$ and it follows that $\dz(\n_\Gamma(Q)/Q)$ is in bijection with $\dz(\c_G(Q))$. Next, for $Q\in\mathcal{S}$, let $\overline{\vartheta}\in\dz(\n_\Gamma(Q)/Q)$, consider its inflation $\vartheta$ to $\n_\Gamma(Q)$, and set $C:=\bl(\vartheta)^\Gamma$. Since $\n_\Gamma(Q)=\c_G(Q)\times Q$ we can write $\vartheta=\varphi\times 1_Q$ for some $\varphi\in\dz(\c_G(Q))$ and hence $Q$ is a defect group of $\bl(\vartheta)$. We deduce from \cite[Lemma 2.1]{Nav-Spa14II} that $C$ covers a $\Gamma$-invariant block $B$ of $G$ with defect $G\cap Q=1$. This shows that for every $Q\in\mathcal{S}$ each character of $\dz(\n_\Gamma(Q)/Q)$ belongs to some set $\dz(\n_\Gamma(Q)/Q\mid B)$ for some $\Gamma$-invariant block $B$ of defect zero of $G$. Now, Conjecture \ref{conj:Blockwise Conjecture E} implies that the right hand side of \eqref{eq:Conjecture blockwise implies DGN, 1} coincides with
\begin{equation}
\label{eq:Conjecture blockwise implies DGN, 2}
\sum\limits_{B}\sum\limits_{Q\in\mathcal{S}}\left|\dz\left(\n_\Gamma(Q)/Q\middle|B\right)\right|=\sum\limits_{Q\in\mathcal{S}}\left|\dz\left(\n_\Gamma(Q)/Q\right)\right|=\sum\limits_{Q\in\mathcal{S}}\left|\dz\left(\c_G(Q)\right)\right|
\end{equation}
where $B$ runs over all $\Gamma$-invariant blocks of defect zero of $G$. Combining \eqref{eq:Conjecture blockwise implies DGN, 1} and \eqref{eq:Conjecture blockwise implies DGN, 2} we obtain the desired equality.
\end{proof}

We now want to prove that the above conjecture follows from the Inductive Blockwise Alperin Weight Condition and hence obtain Theorem \ref{thm:Main, iBAWC implies Gabriel Conjecture}. Together with our Theorem \ref{thm:Main, Reduction}, this will also yield Corollary \ref{cor:Main, Reduction for Gabriel Conjecture}. In this section, we obtain all these results as consequences of a stronger theorem. In fact, we can show that the Inductive Blockwise Alperin Weight Condition implies a more general version of the above Conjecture \ref{conj:Blockwise Conjecture E} which does not require the quotient $\Gamma/G$ to be a $p$-group. To introduce this new statement we first collect some further notation.

Let $G\unlhd \Gamma$ be finite groups and consider a union of $p$-blocks $\mathcal{B}$ of $G$. We denote by $\EBr(\Gamma\mid \mathcal{B})$ the set of those $\chi\in\IBr(\Gamma)$ that are extensions of some Brauer character belonging to some block contained in $\mathcal{B}$, that is, such that $\chi_G\in\IBr(B)$ for some $B\in\mathcal{B}$. Similarly, we denote by $\EBr(\Gamma\mid \dzo(\mathcal{B}))$ the set of those $\chi\in\IBr(\Gamma)$ such that $\chi_G\in\IBr(B)\cap \dzo(G)$ for some $B\in\mathcal{B}$. Finally, given a $p$-subgroup $Q$ of $G$ and a $p$-block $B$ of $G$, we denote by $B_Q$ the union of all $p$-blocks $b$ of $\n_G(Q)$ such that $b^G=B$. We can now generalise Conjecture \ref{conj:Blockwise Conjecture E} to arbitrary quotients $\Gamma/G$ as follows.

\begin{conj}
\label{conj:Blockwise Conjecture E, extended}
Let $G\unlhd \Gamma$ be finite groups and consider a block $B$ of $G$. Then
\[\left|\EBr\left(\Gamma\middle| B\right)\right|=\sum\limits_{Q}\left|\EBr\left(\n_\Gamma(Q)\middle|\dzo\left(B_Q\right)\right)\right|\]
where $Q$ runs over a set of representatives for the $\Gamma$-orbits of radical $p$-subgroups of $G$ such that $\Gamma=G\n_\Gamma(Q)$.
\end{conj}

As mentioned above, Conjecture \ref{conj:Blockwise Conjecture E} can be recovered from Conjecture \ref{conj:Blockwise Conjecture E, extended} in the case where the quotient $\Gamma/G$ is a $p$-group. We prove this fact in the following lemma.

\begin{lem}
\label{lem:Navarro extended implies Navarro}
Let $G\unlhd \Gamma$ be finite groups and consider a $p$-block $B$ of $G$. If $\Gamma/G$ is a $p$-group, then:
\begin{enumerate}
\item $|\EBr(\Gamma\mid B)|=|\IBr_\Gamma(B)|$;
\item if $Q$ is a radical $p$-subgroup of $G$ such that $\Gamma=G\n_\Gamma(Q)$ and $\EBr(\n_\Gamma(Q)\mid\dzo(B_Q))$ is non-empty, then there exists some $D\in\Theta_B$, unique up to $\n_\Gamma(Q)$-conjugation, such that $Q=D\cap G$;
\item if $Q$ and $D$ are the $p$-subgroups considered above, then
\[\left|\EBr\left(\n_\Gamma(Q)\middle|\dzo\left(B_Q\right)\right)\right|=\left|\dz\left(\n_\Gamma(D)/D\middle|B\right)\right|.\]
\end{enumerate}
In particular, if Conjecture \ref{conj:Blockwise Conjecture E, extended} holds for the block $B$, then so does Conjecture \ref{conj:Blockwise Conjecture E}.
\end{lem}

\begin{proof}
By \cite[Theorem 8.11]{Nav98} every $\Gamma$-invariant irreducible Brauer character of $G$ admits a unique extension to $\Gamma$ and therefore (i) follows. Let now $Q$ be a radical $p$-subgroup of $G$ such that $\Gamma=G\n_\Gamma(Q)$ and consider $\psi\in\EBr(\n_\Gamma(Q)\mid \dzo(B_Q))$. Set $\vartheta:=\psi_{\n_G(Q)}$ and observe that $\vartheta\in\dzo(\n_G(Q))$ is $\n_\Gamma(Q)$-invariant and satisfies $\bl(\vartheta)^G=B$. Let $\overline{\psi}$ and $\overline{\vartheta}$ be the Brauer characters of $\N_\Gamma(Q)/Q$ and of $\n_G(Q)/Q$ corresponding to $\psi$ and $\vartheta$ respectively via inflation. Now, if $D/Q$ is a defect group of $\bl(\overline{\psi})$, then \cite[Theorem 9.17]{Nav98} implies that $(D\cap \n_G(Q))/Q$ is a defect group of $\bl(\overline{\vartheta})$ and that $D\n_G(Q)/Q=\n_\Gamma(Q)/Q$. In particular, $\Gamma=G\n_\Gamma(Q)=GD$. Furthermore, since $\bl(\overline{\vartheta})$ has defect zero, we deduce that $Q=D\cap \n_G(Q)=D\cap G$ and, recalling that $Q$ is contained in every defect groups of $\bl(\vartheta)$ and that $\bl(\vartheta)^G=B$, we conclude that $D\in\Theta_B$ thanks to \cite[Lemma 4.13]{Nav98}. This proves (ii).

We keep $D$, $Q$, $\vartheta$, and $\psi$ as in previous paragraph and prove (iii). By \cite[Theorem 3.18]{Nav98} there exists a unique $\varphi\in\irr{\n_G(Q)/Q}$ such that $\varphi^0=\vartheta$. Furthermore, according to \cite[Theorem 2.4]{Nav-Tie11} there exists an extension $\chi\in\irr{\n_\Gamma(Q)/Q}$ of $\varphi$. Since $\chi$ belongs to $\rdz(\n_\Gamma(Q)\mid \chi)$, we can apply Lemma \ref{lem:Definition of rdzo} to show that $\chi^0\in\IBr(\n_\Gamma(Q))$. Then $\chi^0$ is an extension of $\vartheta=\varphi^0$ and \cite[Theorem 8.11]{Nav98} yields $\chi^0=\psi$. This shows that $\psi\in\rdzo(\n_\Gamma(Q)\mid \n_\Gamma(Q),\vartheta)$. Conversely, each character of $\rdzo(\n_\Gamma(Q)\mid \n_\Gamma(Q),\vartheta)$ is an extension of $\vartheta$ and we conclude that
\begin{equation}
\label{eq:Navarro extended implies Navarro, 1}
\EBr\left(\n_\Gamma(Q)\middle|\vartheta\right)=\rdzo\left(\n_\Gamma(Q)\middle|\n_\Gamma(Q),\vartheta\right).
\end{equation}

Next, applying Theorem \ref{thm:Above modular DGN} with $A=M=G=\n_\Gamma(Q)$ and $K=\n_G(Q)$, we get a bijection
\[\Delta^{\n_\Gamma(Q)}_{D,\vartheta}:\rdzo\left(\n_\Gamma(Q)\middle| \n_\Gamma(Q),\vartheta\right)\to\dzo\left(\n_\Gamma(D)\middle| \pi_D(\vartheta)\right)\]
such that if $\zeta:=\Delta_{D,\vartheta}^{\n_\Gamma(Q)}(\psi)$ then $\bl(\zeta)^{\n_\Gamma(Q)}=\bl(\psi)$. On the other hand, since $\vartheta\in\dzo(B_Q)$, applying \cite[Theorem B]{Kos-Spa15} we deduce that $\bl(\psi)^\Gamma$ covers $B$ and by the transitivity of block induction the same holds for $\bl(\zeta)^\Gamma$. In other words, if $\xi\in\dz(\n_\Gamma(D)/D)$ is the ordinary character such that $\xi^0=\zeta$, then $\xi$ belongs to $\dz(\n_\Gamma(D)/D\mid B)$ and the assignment $\psi\mapsto \xi$ is one-to-one. We then obtain (iii) by arguing as above and consider any $\vartheta\in\dzo(B_Q)$.
\end{proof}

Next, we show that our Conjecture \ref{conj:Blockwise Conjecture E, extended} follows from the Inductive Blockwise Alperin Weight Condition.

\begin{pro}
\label{prop:iBAWC implies Gabriel extended}
Let $G\unlhd \Gamma$ be finite groups and consider a block $B$ of $G$. If Conjecture \ref{conj:iBAWC 2} holds for $G\unlhd \Gamma$, then Conjecture \ref{conj:Blockwise Conjecture E, extended} holds for $B$.
\end{pro}

\begin{proof}
By assumption there exists a $\Gamma$-equivariant bijection $\Omega_G$ from $\IBr(G)$ to $\Alp(G)/G$ inducing block isomorphisms of character triples. Suppose that $(Q,\vartheta)\in\Omega(\varphi)$ and let
\[\sigma_\Gamma:\IBr(\Gamma\mid \varphi)\to\IBr(\n_\Gamma(Q)\mid \vartheta)\]
be the bijection given by Theorem \ref{thm:strong-iso}. Then, for every $\chi\in\IBr(\Gamma\mid \varphi)$, we know that $\bl(\sigma_\Gamma(\chi))^\Gamma=\bl(\chi)$ and that $\chi_G=\varphi$ if and only if $\sigma_\Gamma(\chi)_{\n_G(Q)}=\vartheta$. With this in mind, proceeding as in the proof of Theorem \ref{thm:Lifting bijection for weights} with $A=J=\Gamma$ and $K=G$ we can construct a $\Gamma$-equivariant bijection $\Omega_G^\Gamma$ between the set of extensions $\EBr(\Gamma\mid B)$ and the set of $\Gamma$-orbits of pairs $(Q,\psi)$ with $Q$ a radical $p$-subgroup of $G$ and $\psi\in\EBr(\n_\Gamma(Q)\mid \dzo(B_Q))$. Next, we claim that if $(Q,\psi)$ is any such pair, then $\Gamma=G\n_\Gamma(Q)$. Observe that then the constructed bijection would imply the equality of Conjecture \ref{conj:Blockwise Conjecture E, extended}. To prove the claim, let $\chi\in\EBr(\Gamma\mid B)$ and $(Q,\psi)\in\Omega_G^\Gamma(\chi)$ so that $\varphi:=\chi_G\in\IBr(G)$, $\vartheta:=\psi_{\n_G(Q)}\in\IBr(\n_G(Q))$ and $(Q,\vartheta)\in\Omega_G(\varphi)$. Then, since $\varphi$ is $\Gamma$-invariant and $\Omega_G$ is $\Gamma$-equivariant, we deduce that $\Gamma$ fixes the $G$-orbit of $(Q,\vartheta)$, that is, $\Gamma=G\n_\Gamma(Q)_\vartheta$. On the other hand $\vartheta$ is $\n_\Gamma(Q)$-invariant and we conclude that $\Gamma=G\n_\Gamma(Q)$ as claimed.
\end{proof}

We can finally prove Theorem \ref{thm:Main, iBAWC implies Gabriel Conjecture} and Corollary \ref{cor:Main, Reduction for Gabriel Conjecture}.

\begin{proof}[Proof of Theorem \ref{thm:Main, iBAWC implies Gabriel Conjecture}]
Let $G\unlhd \Gamma$ be finite groups with $\Gamma/G$ a $p$-group and consider a block $B$ of $G$. Since by assumption Conjecture \ref{conj:iBAWC 2} holds with respect to $G\unlhd \Gamma$, we can apply Proposition \ref{prop:iBAWC implies Gabriel extended} to show that Conjecture \ref{conj:Blockwise Conjecture E, extended} holds for the block $B$. Then Lemma \ref{lem:Navarro extended implies Navarro} implies that Conjecture \ref{conj:Main, Gabriel Conjecture} holds for the block $B$.
\end{proof}

\begin{proof}[Proof of Corollary \ref{cor:Main, Reduction for Gabriel Conjecture}]
Let $G\unlhd \Gamma$ be finite groups with $\Gamma/G$ a $p$-group and consider a block $B$ of $G$. By assumption Conjecture \ref{conj:iBAWC 2} holds at the prime $p$ for every covering group of any non-abelian finite simple group of order divisible by $p$ involved in $G$. Then Conjecture \ref{conj:iBAWC 2} holds with respect to $G\unlhd \Gamma$ thanks to Theorem \ref{thm:Reduction for iBAW}. We can then apply Theorem \ref{thm:Main, iBAWC implies Gabriel Conjecture} to show that Conjecture \ref{conj:Main, Gabriel Conjecture} holds for the block $B$.
\end{proof}

\subsection{Navarro's Conjecture and isomorphisms of character triples}
\label{sec:Navarro with isomorphisms}

In the previous section we have introduced a generalisation of \cite[Conjecture E]{Nav17} to arbitrary quotients $\Gamma/G$. On the other hand, in this section we show how \cite[Conjecture E]{Nav17} can be strengthened in a different direction, namely by showing that it is compatible with isomorphisms of character triples.

Let $G\unlhd \Gamma$ be finite groups with $\Gamma/G$ a $p$-group and consider a $\Gamma$-invariant block $B$ of $G$. We denote by $\nav(\Gamma\mid B)$ the subset of $\Alp(\Gamma)$ consisting of those pairs $(Q,\psi)$ where $Q$ is a radical $p$-subgroup of $\Gamma$ such that $\Gamma=GQ$ and $\bl(\psi)^\Gamma$ covers $B$. Since $B$ is $\Gamma$-invariant, we deduce that $\nav(\Gamma\mid B)$ is a $\Gamma$-stable subset of $\Alp(\Gamma)$ and we denote by $\nav(\Gamma\mid B)/\Gamma$ the corresponding set of $\Gamma$-orbits. Furthermore, observe that if $(Q,\psi)\in\nav(\Gamma\mid B)$ then the restriction $\psi_{\n_G(Q)}$ is irreducible. This follows, for instance, by considering $\psi$ as a Brauer character of the quotient $\n_\Gamma(Q)/Q$ and noticing that, since $\Gamma=GQ$, the quotient $\n_\Gamma(Q)/Q$ is isomorphic to $\n_G(Q)/Q\cap \n_G(Q)$.

\begin{conj}
\label{conj:Gabriel with character triples}
Let $G\unlhd \Gamma$ be finite groups with $\Gamma/G$ a $p$-group and consider a $\Gamma$-invariant block $B$ of $G$. If $G,\Gamma\unlhd A$, then there exists an $A_B$-equivariant bijection
\[\Omega_B^\Gamma:\IBr_\Gamma(B)\to\nav(\Gamma\mid B)/\Gamma\]
such that the character triples $(A_\chi,G,\chi)$ and $(\n_A(Q)_\vartheta,\n_G(Q),\vartheta)$ are strongly isomorphic for every $\chi\in\IBr_\Gamma(B)$, $(Q,\psi)\in\Omega_B^\Gamma(\chi)$ and where $\vartheta=\psi_{\n_G(Q)}$.
\end{conj}

Since the number of $\Gamma$-orbits on the set $\nav(\Gamma\mid B)$ coincides with the right hand side of \eqref{eq:Blockwise Conjecture E} it follows that our Conjecture \ref{conj:Gabriel with character triples} implies Conjecture \ref{conj:Blockwise Conjecture E}. Next, we show that even this strengthened form of the conjecture is a consequence of the Inductive Blockwise Alperin Weight Condition.

\begin{thm}
Let $G\unlhd \Gamma$ be finite groups with $\Gamma/G$ a $p$-group and consider a $\Gamma$-invariant block $B$ of $G$. Suppose in addition that $G,\Gamma\unlhd A$ and that Conjecture \ref{conj:iBAWC 2} holds with respect to $G\unlhd A$. Then Conjecture \ref{conj:Gabriel with character triples} holds for the block $B$ with respect to $G,\Gamma\unlhd A$.
\end{thm}

\begin{proof}
To start, we define the set $\mathcal{A}_\Gamma(B)$ consisting of those pairs $(Q,\varphi)\in\Alp(G)$ such that $\Gamma=G\n_\Gamma(Q)_\varphi$ and $\bl(\varphi)^G=B$. Observe that the first condition is equivalent to requiring that the $G$-orbit of $(Q,\varphi)$ is invariant under the action of $\Gamma$ on $\Alp(G)/G$. Then, if we denote by $\mathcal{A}_\Gamma(B)/G$ the set of $G$-orbits on $\mathcal{A}_\Gamma(B)$, applying Conjecture \ref{conj:iBAWC 2} we obtain a bijection
\[\Omega:\IBr_\Gamma(B)\to\mathcal{A}_\Gamma(B)/G\]
inducing block isomorphisms of character triples. Fix now a pair $(Q,\varphi)\in\mathcal{A}_\Gamma(B)$ and observe that $\n_\Gamma(Q)=\n_\Gamma(Q)_\varphi$ and hence $\varphi$ is $\norm{\Gamma}{Q}$-invariant. Then, by \cite[Theorem 8.11]{Nav98}, there exists a unique extension $\tau\in\IBr(\n_\Gamma(Q))$ of $\varphi$. Let $D/Q$ be a defect group of the block $\bl(\overline{\tau})$ and where $\overline{\tau}$ is the Brauer character of $\n_\Gamma(Q)/Q$ corresponding to $\tau$ via inflation. By \cite[Theorem 9.17]{Nav98} we have $\Gamma=GD$ and $G\cap D=Q$, while applying \cite[Theorem B]{Kos-Spa15} we deduce that $\bl(\tau)^\Gamma$ covers $B$. Consider now the unique ordinary character $\varphi'\in\dz(\n_G(Q)/Q)$ such that $\varphi'^0=\varphi$ and let $\tau'\in\Irr(\n_\Gamma(Q)/Q)$ be an extension of $\varphi'$ (this exists according to \cite[Theorem 2.4]{Nav-Tie11}). Then, Lemma \ref{lem:Definition of rdzo} implies that $\tau'^0\in\IBr(\n_\Gamma(Q))$ and the uniqueness part of \cite[Theorem 8.11]{Nav98} yields $\tau'^0=\tau$. This shows that $\tau\in\rdzo(\n_\Gamma(Q)\mid \n_\Gamma(Q),\varphi)$. We can now apply Theorem \ref{thm:Above modular DGN} to shows that $\tau$ corresponds to a unique $\psi\in\dzo(\n_\Gamma(D)\mid\pi_{D}(\varphi))$. Furthermore, using the block isomorphisms given by Theorem \ref{thm:Above modular DGN}, we also get $\bl(\psi)^{\n_\Gamma(Q)}=\bl(\tau)$ and that $\psi_{\n_G(D)}$ is irreducible. We then deduce that the pair $(D,\psi)$ belongs to $\nav(\Gamma\mid B)$ and can define an $A_B$-equivariant bijection
\[\Delta:\mathcal{A}_\Gamma(B)/G\to \nav(\Gamma\mid B)/\Gamma\]
by sending the $G$-orbit of $(Q,\varphi)$ to the $\Gamma$-orbit of $(D,\psi)$ as described above. Finally, we define $\Omega_B^\Gamma$ to be the composition of $\Omega$ and $\Delta$. To complete the proof it remains to show that $\Omega_B^\Gamma$ induces strong isomorphisms of character triples. More precisely, let $\chi\in\IBr_\Gamma(B)$, $(Q,\varphi)\in\Omega(\chi)$ and $(D,\psi)$ as constructed above. First, by Conjecture \ref{conj:iBAWC 2} we know that $(A_\chi,G,\chi)$ and $(\n_A(Q)_\varphi,\n_G(Q),\varphi)$ are block isomorphic and therefore strongly isomorphic. Next, we notice set $\vartheta:=\psi_{\n_G(D)}$ and notice that $\vartheta=\pi_D(\varphi)$. Then, applying Corollary \ref{cor:DGN modular iso} we get a strong isomorphism between $(\n_A(Q)_\varphi,\n_G(Q),\varphi)$ and $(\n_A(D)_\vartheta,\n_G(D),\vartheta)$ as required. 
\end{proof}

We conclude this section with a remark on the isomorphisms of character triples considered above.

\begin{rem}
It is natural to ask whether the strong isomorphisms considered in Conjecture \ref{conj:Gabriel with character triples} are actually block isomorphisms. Unfortunately, this is not necessarily the case. The reason for this is that the condition on defect groups required by Definition \ref{def:block-iso} might fail in this case.
\end{rem}

\subsection{Block-free version of Navarro's Conjecture}
\label{sec:Block-free Navarro Conjecture}

In this section, using the results of Section \ref{sec:Reduction for block-free version}, we obtain block-free analogues of Theorem \ref{thm:Main, iBAWC implies Gabriel Conjecture} and Corollary \ref{cor:Main, Reduction for Gabriel Conjecture}. For the reader's convenience, we first state the block-free version of Conjecture \ref{conj:Blockwise Conjecture E}. The following is basically the statement of \cite[Conjecture E]{Nav17}, although notice that we do not require the group $\Gamma$ to split over $G$. Recall that if $G\unlhd \Gamma$ then $\IBr_\Gamma(G)$ denotes the set of irreducible Brauer characters of $G$ that are $\Gamma$-invariant.

\begin{conj}[Navarro]
\label{conj:Block-free Conjecture E}
Let $G\unlhd \Gamma$ be finite groups with $\Gamma/G$ a $p$-group. Then
\[\left|\IBr_\Gamma(G)\right|=\sum\limits_{Q\in\Theta/\Gamma}\left|\dz\left(\n_\Gamma(Q)/Q\right)\right|\]
where $\Theta$ is the set of $p$-subgroups $Q$ of $\Gamma$ such that $\Gamma=GQ$.
\end{conj}

Next, as done in the blockwise setting, we introduce a generalisation of the above conjecture to arbitrary quotients $\Gamma/G$. For $G\unlhd \Gamma$, let $\EBr(\Gamma\mid G)$ be the set of irreducible Brauer characters of $\Gamma$ that restrict irreducibly to $G$. Similarly, we denote by $\EBr(\Gamma\mid \dzo(G))$ the subset of those characters $\chi\in\EBr(\Gamma\mid G)$ whose restriction $\chi_G$ belongs to $\dzo(G)$. We can then state a block-free version of Conjecture \ref{conj:Blockwise Conjecture E, extended} as follows.

\begin{conj}
\label{conj:Block-free Conjecture E, extended}
Let $G\unlhd \Gamma$ be finite groups. Then
\[\left|\EBr\left(\Gamma\middle| G\right)\right|=\sum\limits_{Q}\left|\EBr\left(\n_\Gamma(Q)\middle|\dzo\left(\n_G(Q)\right)\right)\right|\]
where $Q$ runs over a set of representatives for the $\Gamma$-orbits of radical $p$-subgroups of $G$ such that $\Gamma=G\n_\Gamma(Q)$.
\end{conj}

Arguing as in the proofs of Lemma \ref{lem:Navarro extended implies Navarro} and Proposition \ref{prop:iBAWC implies Gabriel extended} we can then prove the following proposition.

\begin{pro}
\label{prop:iAWC implies block-free Gabriel}
Let $G\unlhd \Gamma$ be finite groups.
\begin{enumerate}
\item If Conjecture \ref{conj:iAWC} holds with respect to $G\unlhd \Gamma$, then Conjecture \ref{conj:Block-free Conjecture E, extended} holds with respect to $G\unlhd \Gamma$.
\item Suppose that $\Gamma/G$ is a $p$-group. If Conjecture \ref{conj:Block-free Conjecture E, extended} holds with respect to $G\unlhd \Gamma$, then Conjecture \ref{conj:Block-free Conjecture E} holds with respect to $G\unlhd \Gamma$.
\end{enumerate}
\end{pro}

As an immediate consequence we obtain a block-free version of Theorem \ref{thm:Main, iBAWC implies Gabriel Conjecture}.

\begin{thm}
\label{thm:iAWC implies block-free Gabriel}
Let $G\unlhd \Gamma$ be finite groups such that $\Gamma/G$ is a $p$-group. If Conjecture \ref{conj:iAWC} holds with respect to $G\unlhd \Gamma$, the Conjecture \ref{conj:Block-free Conjecture E} holds with respect to $G\unlhd \Gamma$.
\end{thm}

\begin{proof}
This follows immediately by combining the two parts of Proposition \ref{prop:iAWC implies block-free Gabriel}.
\end{proof}

Then, using the above result together with the reduction obtained in Theorem \ref{thm:Reduction, block-free version}, we can prove the following block-free analogue of Corollary \ref{cor:Main, Reduction for Gabriel Conjecture}.

\begin{cor}
\label{cor:Reduction for block-free Gabriel}
Let $G$ be a finite group and $p$ a prime number. If Conjecture \ref{conj:iAWC} holds at the prime $p$ for every covering group of any non-abelian finite simple group of order divisible by $p$ involved in $G$, then Conjecture \ref{conj:Main, Gabriel Conjecture} holds at the prime $p$ with respect to any $G\unlhd \Gamma$ such that $\Gamma/G$ is a $p$-group.
\end{cor}

\section{Verification of Conjecture \ref{conj:Main, Gabriel Conjecture} and Conjecture \ref{conj:Main, iBAWC}}
\label{sec:Results}

Since their introduction in \cite{Nav-Tie11} and \cite{Spa13I} the Inductive Alperin Weight Condition (see Conjecture \ref{conj:iAWC}) and the Inductive Blockwise Alperin Weight Condition (see Conjecture \ref{conj:iBAWC 2}) have been verified for many classes of blocks of quasi-simple groups. We refer the reader to the survey \cite{Fen-Zha22} and the references therein (see also the most recently published \cite{Fen-Li-Zha23I} and \cite{Fen-Li-Zha23II}). Using these results, and introducing some new arguments, we can then verify Conjecture \ref{conj:Main, Gabriel Conjecture} and Conjecture \ref{conj:Main, iBAWC} for several classes of groups and blocks.

\subsection{Groups with abelian Sylow $p$-subgroups and odd Sylow automiser}

The Inductive Blockwise Alperin Weight Condition has been verified for all quasi-simple groups whose simple quotient has an abelian Sylow $2$-subgroup \cite[Corollary 6.6]{Spa13I}, or an abelian Sylow $3$-subgroup \cite[Section 5]{Fen-Li-Zha23I}, or is involved in a group with odd Sylow normaliser \cite{Gur-Nav-Tie16} (see also \cite{Xu-Zho19} for the case of odd Sylow automisers). Thanks to Theorem \ref{thm:Main, iBAWC implies Gabriel Conjecture} and Theorem \ref{thm:Main, Reduction}, we then obtain new evidence for Conjecture \ref{conj:Main, Gabriel Conjecture} and Conjecture \ref{conj:Main, iBAWC}.

\begin{pro}
\label{prop:iBAW for groups with abelian Sylow 2-subgroups}
Let $G$ be a finite group and consider any prime number $p$. Then Conjecture \ref{conj:Main, Gabriel Conjecture} and Conjecture \ref{conj:Main, iBAWC} both hold for $G$ at the prime $p$ provided that at least one of the following conditions is satisfied:
\begin{enumerate}
\item every simple group involved in $G$ has either abelian Sylow $2$-subgroups or abelian Sylow $3$-subgroups or both;
\item $p$ is an odd prime and the automiser $\n_G(P)/P\c_G(P)$ has odd order for a Sylow $p$-subgroup $P$ of $G$.
\end{enumerate}
\end{pro}

\begin{proof}
Recall that Conjecture \ref{conj:Main, Gabriel Conjecture} follows from Conjecture \ref{conj:Main, iBAWC} thanks to Theorem \ref{thm:Main, iBAWC implies Gabriel Conjecture}. Moreover, in order to prove Conjecture \ref{conj:Main, iBAWC} for the group $G$ it suffices to show that it holds for every non-abelian simple group of order divisible by $p$ involved in $G$ according to Theorem \ref{thm:Main, Reduction}. We then obtain the first part of the statment by applying the results obtained in \cite[Theorem 6.6]{Spa13I}, \cite[Section 5]{Fen-Li-Zha23I} and the second part by applying the results in \cite{Xu-Zho19} (see also \cite{Gur-Nav-Tie16}).
\end{proof}

\subsection{Blocks with cyclic defect groups}

It was shown in \cite{Kos-Spa16I} and \cite{Kos-Spa16II} that the Inductive Blockwise Aperin Weight Condition holds for every block with cyclic defect groups of any quasi-simple group. In order to apply these results to verify Conjecture \ref{conj:Main, Gabriel Conjecture} and Conjecture \ref{conj:Main, iBAWC}, we first need to prove a version of Theorem \ref{thm:Main, Reduction} compatible with certain families of defect groups as done in \cite[Theorem C]{Spa13I}. To start, we need to restate Conjecture \ref{conj:Main, iBAWC}.  

Let $B$ be a $p$-block of a finite group $G$ and consider a $p$-weight $(Q,\psi)$ of $G$ as defined in Section \ref{sec:iBAW}. Observe that the block $\bl(\psi)^G$ of $G$ obtained via Brauer induction of blocks is well defined according to \cite[Theorem 4.14]{Nav98}. Then we say that $(Q,\psi)$ is a \textit{$p$-weight of $B$} provided that $\bl(\psi)^G=B$. We denote by $\Alp(B)$ the set of all $p$-weights of $B$ and by $\Alp(B)/G$ the set of $G$-orbits of such $p$-weights. Now, using the properties of block isomorphisms of character triples (see Definition \ref{def:block-iso}) we deduce that Conjecture \ref{conj:iBAWC 2} can be reformulated as follows.

\begin{conj}[Inductive Blockwise Alperin Weight Condition]
\label{conj:iBAWC}
Let $G\unlhd A$ be finite groups and consider a $p$-block $B$ of $G$. Then there exists an $A_B$-equivariant bijection
\[\Omega_B:\IBr(B)\to\Alp(B)/G\]
such that
\[\left(A_\vartheta,G,\vartheta\right)\isob\left(\n_A(Q)_\psi,\n_G(Q),\psi\right)\]
for every $\vartheta\in\IBr(B)$ and $(Q,\psi)\in\Omega_B(\vartheta)$.
\end{conj}

\begin{rem}
\label{rmk:iBAWC implies Navarro}
Let $G\unlhd A$ be finite groups and consider a block $B$ of $G$. Arguing as in the proof of Proposition \ref{prop:iBAWC implies Gabriel extended} we can show that if Conjecture \ref{conj:iBAWC} holds for the block $B$ with respect to $G\unlhd A$, then so does Conjecture \ref{conj:Blockwise Conjecture E, extended}. In particular, if $\Gamma/G$ is a $p$-group, we deduce that Conjecture \ref{conj:Main, Gabriel Conjecture} holds for the block $B$ as a consequence of Lemma \ref{lem:Navarro extended implies Navarro}. This observation will be used in the subsequent sections to obtain Conjecture \ref{conj:Main, Gabriel Conjecture} for certain classes of blocks.
\end{rem}

By considering the above version of the conjecture we can obtain a refined version of our Theorem \ref{thm:Reduction for iBAW} and prove a reduction theorem compatible with any fixed class of defect groups closed under taking quotients and subgroups as done in \cite[Theorem C]{Spa13I}. Given a family of $p$-subgroups $\mathcal{R}$, we say that Conjecture \ref{conj:iBAWC} holds for a finite group $G$ with respect to $\mathcal{R}$ if it holds for every choice of $G\unlhd A$ and every $p$-block $B$ of $G$ with defect groups contained in $\mathcal{R}$.

\begin{thm}
\label{thm:Reduction for iBAW with class of defect groups}
Let $\mathcal{R}$ be a family of finite $p$-groups closed under taking quotients and subgroups. Let $G$ be a finite group and suppose that Conjecture \ref{conj:iBAWC} holds with respect to $\mathcal{R}$ for every covering group of any non-abelian finite simple group of order divisible by $p$ involved in $G$. Then Conjecture \ref{conj:iBAWC} holds for $G$ with respect to $\mathcal{R}$.
\end{thm}

\begin{proof}
As for \cite[Theorem 5.20]{Spa13I}, this follows by an inspection of the proof of Theorem \ref{thm:Reduction for iBAW}. For this, fix a block $B$ with defect groups contained in $\mathcal{R}$. Then, for every normal subgroup $K\unlhd A$ with $K\leq G$ and every $A$-invariant $\vartheta\in\dzo(K)$ with $B$ covering the block of $\vartheta$, the construction used in the proof of Lemma \ref{lem:Reduction, above a character in dzo} restricted to the set of Brauer characters $\IBr(B\mid \vartheta)$ will produce a block $\overline{B}$ of $\overline{G}=\wh{G}/K_0$ with defect group contained in $\mathcal{R}$ by recalling that the central subgroup $Z$ of $\wh{G}$ has order prime to $p$. We can then apply the inductive hypothesis to the block $\overline{B}$ and finish the proof proceeding as in the remaining part of the proof of Theorem \ref{thm:Reduction for iBAW}. 
\end{proof}

Using the above theorem, we can then apply the results of \cite{Kos-Spa16I} and \cite{Kos-Spa16II} to obtain Conjecture \ref{conj:Main, Gabriel Conjecture}, Conjecture \ref{conj:Blockwise Conjecture E, extended}, and Conjecture \ref{conj:iBAWC} for every block with cyclic defect groups of any finite group, and where we further assume $\Gamma/G$ to be a $p$-group in the case of Conjecture \ref{conj:Main, Gabriel Conjecture}.

\begin{pro}
\label{prop:Conjectures for cyclic defect groups}
Let $G\unlhd A$ be finite groups and consider a $p$-block $B$ of $G$ with cyclic defect groups. Then Conjecture \ref{conj:Blockwise Conjecture E, extended} and Conjecture \ref{conj:iBAWC} hold for the block $B$. Furthermore, if $\Gamma/G$ is a $p$-group, then Conjecture \ref{conj:Main, Gabriel Conjecture} holds for $B$.
\end{pro}

\begin{proof}
By Remark \ref{rmk:iBAWC implies Navarro} it suffices to show that Conjecture \ref{conj:iBAWC} holds for the block $B$. The latter follows by applying Theorem \ref{thm:Reduction for iBAW with class of defect groups} since Conjecture \ref{conj:iBAWC} has been verified for all $p$-blocks with cyclic defect of quasi-simple groups (see \cite{Kos-Spa16I} and \cite{Kos-Spa16II}).
\end{proof}

\subsection{Nilpotent blocks}

The results obtained in \cite{Kos-Spa16I} also imply that the Inductive Blockwise Alperin Weight Condition holds for every nilpotent block of a quasi-simple group. Their argument heavily relies on the fact that such blocks are known to have abelian defect groups thanks to \cite{An-Eat11} and \cite{An-Eat13}. In order to obtain Conjecture \ref{conj:Main, Gabriel Conjecture}, Conjecture \ref{conj:Blockwise Conjecture E, extended}, and Conjecture \ref{conj:iBAWC} for nilpotent blocks of all finite groups, we therefore need to develop a new argument. Our proof does not depend on the Classification of Finite Simple Groups and makes use of some deep results on graded Morita equivalences obtained in \cite{Pui-Zho12}.

As a first step, we give a reformulation of Conjecture \ref{conj:iBAWC} in the spirit of \cite{Kno-Rob89}. For this, denote by $\mathcal{P}(G)$ the set of $p$-chains of the form $\sigma=\{1=Q_0<Q_1<\dots<Q_n\}$ where each $Q_i$ is a $p$-subgroup of $G$. We denote by $|\sigma|$ the integer $n$ called the \textit{length} of $\sigma$. Observe that $G$ acts by conjugation on the set of $p$-chains and denote by $G_\sigma=\cap_i\n_G(Q_i)$ the stabiliser of $\sigma$ in $G$. Next, for any block $B$ of $G$, let $B_\sigma$ denote the union of all blocks of $G_\sigma$ that correspond to $B$ via Brauer induction of blocks and let $\IBr(B_\sigma)$ be the union of all Brauer characters belonging to some block in $B_\sigma$. Similarly, for a $p$-subgroup $Q$ of $G$ we denote by $B_Q$ the union of blocks $B_{\{1<Q\}}$. We then consider the set $\Co(B)$ of pairs $(\sigma,\psi)$ where $\sigma\in\mathcal{P}(G)$ and $\psi\in\IBr(B_\sigma)$. The notion of length yields a partition of $\Co(B)$ into the sets $\Co(B)_\pm$ consisting of those pairs $(\sigma,\psi)$ such that $(-1)^{|\sigma|}=\pm 1$. The group $G$ now acts by conjugation on $\Co(B)_\pm$ and we denote by $\Co(B)_\pm/G$ the corresponding set of $G$-orbits and by $\overline{(\sigma,\psi)}$ the $G$-orbit of an element $(\sigma,\psi)\in\Co(B)_\pm$. We can now restate the Inductive Blockwise Alperin Weight Condition as follows.

\begin{lem}
\label{lem:Reformulation for iBAW in terms of chains}
The following statements are equivalent:
\begin{enumerate}
\item Conjecture \ref{conj:iBAWC} holds for every $p$-block of any finite group $G$ and any choice of $G\unlhd A$;
\item For all finite groups $G\unlhd A$ and every $p$-block $B$ of $G$ with non-trivial defect, there exists an $A_B$-equivariant bijection
\[\Xi:\Co(B)_-/G\to\Co(B)_+/G\]
such that
\[\left(A_{\sigma,\psi},G_\sigma,\psi\right)\isob\left(A_{\rho,\varphi},G_\rho,\varphi\right)\]
for every $(\sigma,\psi)\in\Co(B)_-$ and any $(\rho,\varphi)\in\Xi(\overline{(\sigma,\psi)})$.
\end{enumerate}
\end{lem}

\begin{proof}
Suppose first that (i) holds. Then, we may assume that $G$ is a minimal counterexample to (ii) (and similarly, if we suppose that (ii) holds, then we may assume that $G$ is a minimal counterexample to (i)). In particular, (ii) holds for $\n_G(Q)/Q$ for every non-trivial $p$-subgroup $Q$ of $G$. We now fix a $G$-transversal $\mathcal{T}$ in the set of non-trivial $p$-subgroups of $G$ and, for each $Q\in\mathcal{T}$, denote by $\Co_Q(B)$ the subset of $\Co(B)$ consisting of pairs $(\sigma,\psi)$ with $\sigma=\{1=Q_0<Q_1<\dots<Q_n\}$ where $n\geq 1$ and $Q_1$ is $G$-conjugate to $Q$. Observe that there is a one-to-one correspondence between the set $\Co_Q(B)$ and the set $\Co(\overline{B_Q})$ where $\overline{B_Q}$ is the set of blocks of $\n_G(Q)/Q$ dominated by some block in $B_Q$. Denote by $\overline{b_Q}$ the union of those blocks in $\overline{B_Q}$ with positive defect and by $\overline{c_Q}$ the union of the blocks in $\overline{B_Q}$ of defect zero. Now, applying \cite[Lemma 9.15 (a)]{Nav18} we deduce that a pair $(\overline{\sigma},\overline{\psi})\in\Co(\overline{c_Q})$ must have $\overline{\sigma}=\{1\}$ and $\bl(\overline{\psi})$ a block of defect zero in $\n_G(Q)/Q$. Hence, the set $\Co(\overline{c_Q})$ corresponds to the set of pairs $(\sigma,\psi)\in\Co_Q(B)$ with $\sigma=\{1=Q_0<Q_1=Q\}$ and $\psi\in\dzo(\n_G(Q))$ such that $\bl(\psi)^G=B$. We denote by $\mathcal{A}_Q(B)$ the set of such pairs and set $\mathcal{B}_Q(B):=\Co_Q(B)\setminus \mathcal{A}_Q(B)$. If we now consider the remaining blocks $\overline{b_Q}$ of positive defect, then arguing as in \cite[Lemma 2.2, Corollary 2.3, and Proposition 2.4]{Ros-CTC-max} and using the fact that (ii) holds for all blocks with non-trivial defect of $\n_G(Q)/Q$, we obtain a bijection
\[\Xi_Q:\mathcal{B}_Q(B)_+/G\to \mathcal{B}_Q(B)_-/G\]
inducing block isomorphisms of modular character triples. Finally, combining the bijections $\Xi_Q$ for $Q\in\mathcal{T}$, we are left to consider pairs $(\sigma,\psi)\in\Co(B)$ where either $\sigma=\{1\}$ or $(\sigma,\psi)$ belongs to $\mathcal{A}_Q(B)$ for some $Q\in\mathcal{T}$. In other words a bijection $\Xi$ with the properties required in (ii) exists if and only if we can find a bijection
\[\Omega:\IBr(B)\to\Alp(B)/G\]
satisfying the statement of Conjecture \ref{conj:iBAWC}.
\end{proof}

The reformulation of Conjecture \ref{conj:iBAWC 2} given in Lemma \ref{lem:Reformulation for iBAW in terms of chains} is inspired by \cite[Theorem 3.8]{Kno-Rob89} and can be used to prove Conjecture \ref{conj:iBAWC 2} in certain situations. While it is not true that the two statements in Lemma \ref{lem:Reformulation for iBAW in terms of chains} are equivalent block-by-block, they can still be shown to be equivalent for certain classes of blocks. This is the case for nilpotent blocks. For this, let $B$ be a nilpotent block of a finite group $G$ and consider a $p$-subgroup $Q$ of $G$. If $b$ is a block of $\n_G(Q)$ and $b^G=B$, then it follows that $b$ is nilpotent (see the first paragraph of the proof of \cite[Theorem 3.2.2]{Rob02I}) Moreover, if $Q$ is normal in $G$ and $\overline{B}$ is a block of $G/Q$ dominated by $B$, then it follows that $\overline{B}$ is nilpotent (see, for instance, \cite[Theorem 1.1 (i)]{Coc-Tod20} for a precise reference).

\begin{cor}
\label{cor:Reformulation for iBAW in terms of chains, nilpotent}
The statements (i) and (ii) in Lemma \ref{lem:Reformulation for iBAW in terms of chains} are equivalent for nilpotent blocks. More precisely, Conjecture \ref{conj:iBAWC} holds for every nilpotent block of every finite group $G$, and any choice $G\unlhd A$, if and only if a bijection $\Xi$ with the properties described in (ii) exists for every nilpotent block of every finite group $G$ and with respect to any $G\unlhd A$.
\end{cor}

\begin{proof}
This follows from the fact that the argument used in the proof of Lemma \ref{lem:Reformulation for iBAW in terms of chains} is compatible with nilpotent blocks. In fact, suppose that $B$ is a nilpotent block of a finite group $G$ and let $Q$ be a $p$-subgroup of $G$. If $b$ is a block in the union $B_Q$, then $b$ is a block of $\n_G(Q)$ such that $b^G=B$ and it follows that $b$ is nilpotent. Furthermore, if $\overline{b}$ is dominated by $b$, then it also follows that $\overline{b}$ is nilpotent. Hence every block in $\overline{B_Q}$ is nilpotent so that we can apply the inductive hypothesis in the $\n_G(Q)/Q$ and proceed as in the proof of Lemma \ref{lem:Reformulation for iBAW in terms of chains}.
\end{proof}

Next, we need a modular version of \cite[Corollary 3.2]{Ros-CTC}. For this, we will use the main result of \cite{Pui-Zho12} together with a modular version of \cite[Proposition 5.6]{Mar-Min21} that was kindly provided to us by A. Marcus. We state this latter result below for the reader's convenience. Let $N\unlhd G$ be finite groups and consider a $p$-modular system $(K,\mathcal{O},k)$ such that $k$ is algebraically closed and $K$ contains a primitive $|G|$-th root of unity. Following \cite[Section 4.1]{Mar-Min21}, let $C$ be a $G$-invariant block of $N$ and set $\mathcal{A}:=\mathcal{O}GC$ and $\mathcal{A}':=\mathcal{O}\n_G(D)c$ where $c$ is the Brauer correspondent of $C$ in $\n_N(D)$ with respect to a defect group $D$ of $C$. Denote by $\mathcal{B}:=\mathcal{O}NC$ and by $\mathcal{B}':=\mathcal{O}\n_N(D)c$ the $1$-component of $\mathcal{A}$ and $\mathcal{A}'$ respectively. Set $k\mathcal{B}:=k\otimes_{\mathcal{O}}\mathcal{B}$ and $k\mathcal{B}':=k\otimes_{\mathcal{O}}\mathcal{B}'$. Proceeding as in \cite[Definition 4.2 and Definition 5.1]{Mar-Min21} (but considering modules over $k$), for every $G$-invariant $k\mathcal{B}$-module $V$ and every $\n_G(D)$-invariant $k\mathcal{B}'$-module $V'$ we define the order relation $(\mathcal{A},\mathcal{B},V)\isob(\mathcal{A}',\mathcal{B}',V')$ which in turns implies that $(G,N,\vartheta)\isob(\n_G(D),\n_N(D),\vartheta')$ for the Brauer characters $\vartheta$ and $\varphi$ corresponding to $V$ and $V'$ respectively. We then have the following criterion for establishing $(\mathcal{A},\mathcal{B},V)\isob(\mathcal{A}',\mathcal{B}',V')$.

\begin{lem}
\label{lem:Modular Marcus-Minuta}
Let $N\unlhd G$ be finite groups and consider a $G$-invariant block $C$ of $N$ with defect group $D$ and Brauer correspondent $c$ in $\n_N(D)$. Set $\mathcal{A}:=\mathcal{O}GC$, $\mathcal{A}':=\mathcal{O}\n_G(D)c$, $\mathcal{B}:=\mathcal{O}NC$, and $\mathcal{B}':=\mathcal{O}\n_N(D)c$ as above. Suppose that $\mathcal{A}$ and $\mathcal{A}'$ are $G/N$-graded basic Morita equivalent and that this equivalence sends the $k\mathcal{B}$-module $V$ to the $k\mathcal{B}'$-module $V'$. Then $(\mathcal{A},\mathcal{B},V)\isob(\mathcal{A}',\mathcal{B}',V')$.
\end{lem}

\begin{proof}
This follows by arguing as in the proof of \cite[Proposition 5.6]{Mar-Min21} and noticing that, according to \cite[Remark 5.7]{Mar-Min21}, the Morita equivalence in the statement is compatible with the Brauer map as defined in \cite[Definition 5.5]{Mar-Min21}.
\end{proof}

We can now obtain the above mentioned modular version of \cite[Corollary 3.2]{Ros-CTC}.

\begin{lem}
\label{lem:Basic Morita induces modular isomorphisms}
Let $C$ be a $p$-block of a finite group $H$ with defect group $D$ and consider its Brauer correspondent $c$ in $\n_H(D)$. Assume that $C$ is nilpotent and let $\psi$ denote its unique Brauer character. Observe that $c$ must be nilpotent and hence contains a unique Brauer character $\varphi$. If $H\unlhd A$, then
\[\left(A_\psi,H,\psi\right)\isob\left(\n_A(D)_\varphi,\n_H(D),\varphi\right).\]  
\end{lem}

\begin{proof}
We proceed as in the proof of \cite[Corollary 3.2]{Ros-CTC}. First notice that it is no loss of generality to assume that $\psi$ is $A$-invariant in which case $\varphi$ is $\n_A(D)$-invariant. Since $\psi$ and $\varphi$ are the unique Brauer characters of $C$ and $c$ respectively it also follows that $C$ is $A$-invariant and $c$ is $\n_A(D)$-invariant. We now set $\overline{A}:=A/H$, $\mathcal{A}:=\mathcal{O}AC$, and $\mathcal{A}':=\mathcal{O}\n_A(D)c$. A Frattini argument yields $A=\n_A(D)H$ and hence $\overline{A}=\n_A(D)/\n_H(D)$. Since $C$ is nilpotent, \cite[Theorem 3.14 and Corollary 3.15]{Pui-Zho12} implies that there exists an $\overline{A}$-graded $(\mathcal{A},\mathcal{A}')$-bimodule $M$ inducing an $\overline{A}$-graded basic Morita equivalence between $\mathcal{A}$ and $\mathcal{A}'$. Consider the identity components $\mathcal{B}:=\mathcal{A}_1=\mathcal{O}HC$ and $\mathcal{B}':=\mathcal{A}'_1=\mathcal{O}\n_H(D)c$ and notice that $M_1$ induces a basic Morita equivalence between $\mathcal{B}$ and $\mathcal{B}'$ that sends the Brauer character $\psi$ of $C$ to the Brauer character $\varphi$ of $c$. Then, by applying Lemma \ref{lem:Modular Marcus-Minuta} and considering modules over $k=\mathcal{O}/J(\mathcal{O})$, we obtain the required block isomorphism of modular character triples.
\end{proof}

Using Lemma \ref{lem:Basic Morita induces modular isomorphisms}, we can finally show that Conjecture \ref{conj:Main, Gabriel Conjecture}, Conjecture \ref{conj:Blockwise Conjecture E, extended}, and Conjecture \ref{conj:iBAWC} hold for every nilpotent block of a finite group. This will follow from our next result thanks to Corollary \ref{cor:Reformulation for iBAW in terms of chains, nilpotent} (see also Remark \ref{rmk:iBAWC implies Navarro}).

\begin{pro}
\label{prop:Proving the reformulation for nilpotent blocks}
Let $G\unlhd A$ be finite groups and consider a nilpotent block $B$ of $G$. If $B$ has positive defect, then there exists an $A_B$-equivariant bijection
\[\Xi:\Co(B)_-/G\to\Co(B)_+/G\]
such that
\[\left(A_{\sigma,\psi},G_\sigma,\psi\right)\isob\left(A_{\rho,\varphi},G_\rho,\varphi\right)\]
for every $(\sigma,\psi)\in\Co(B)_-$ and any $(\rho,\varphi)\in\Xi(\overline{(\sigma,\psi)})$.
\end{pro}

\begin{proof}
The following argument is inspired by the cancellation theorem introduced in \cite[Proposition 5.5]{Kno-Rob89}. First observe that, as remarked in the proof of Corollary \ref{cor:Reformulation for iBAW in terms of chains, nilpotent}, for every $p$-chain $\rho\in\mathcal{P}(G)$ and every block $c$ contained in the union $B_\rho$ the nilpotency of $B$ forces $c$ to be nilpotent as well. In particular, for each pair $(\rho,\varphi)\in\Co(B)$ we know that $\varphi$ is the unique Brauer character in its block. Write $\rho=\{1=Q_0<Q_1<\dots<Q_n\}$ and consider a defect group $D$ of $\bl(\varphi)$. Observe that $Q_n$ is a normal $p$-subgroup of $G_\rho$ and so is contained in $D$ by \cite[Theorem 4.8]{Nav98}. Assume first that $Q_n=D$. If $Q_n=1$, then $\bl(\varphi)=B$ has defect $D=1$, a contradiction. Thus we must have $Q_n\neq 1$ and we can defined the $p$-chain $\sigma\in\mathcal{P}(G)$ given by deleting the last term $Q_n$ from $\rho$. Noticing that $G_\rho=\n_{G_\sigma}(D)$, by Brauer's First Main Theorem there exists a unique block $C=\bl(\varphi)^{G_\sigma}$ of $G_\sigma$ with defect group $D$. Once again $C$ must be nilpotent and it contains a unique Brauer character, say $\psi$. Then, by applying Lemma \ref{lem:Basic Morita induces modular isomorphisms} with $A:=A_\sigma$ and $H:=G_\sigma$, we get $\left(A_{\sigma,\psi},G_\sigma,\psi\right)\isob\left(A_{\rho,\varphi},G_\rho,\varphi\right)$. We then map the $G$-orbit of $(\rho,\varphi)$ to that of $(\sigma,\psi)$. On the other hand, if the last term $Q_n$ of $\rho$ is strictly contained in $D$, then we define $\sigma$ to be the $p$-chain obtained by adding $D$ at the end of $\rho$ and let $\psi$ be the unique Brauer character belonging to the nilpotent block of $G_\sigma$ that induces to $\bl(\varphi)$. In this case, we argue as before by switching the role of $(\sigma,\psi)$ and $(\rho,\varphi)$. This yields a bijection $\Xi$ with the properties required in the statement.
\end{proof}

\subsection{$2$-blocks with abelian defect groups}

The Blockwise Alperin Weight Conjecture has recently been proved for $2$-blocks with abelian defect groups by Ruhstorfer in \cite{Ruh22abelian}. This result was obtained as a consequence of the validity of the Alperin--McKay Conjecture for this class of blocks and applying \cite[Proposition 5.6]{Kno-Rob89}. Unfortunately, this strategy does not imply the full Inductive Blockwise Alperin Weight Condition. The latter has instead been verified by Zhang and Zhou \cite{Zha-Zho21} and Hu and Zhou \cite{Hu-Zho} and relies on the classification of $2$-blocks with abelian defect groups obtained in \cite{EKKS14}. Thanks to these results we can then prove Conjecture \ref{conj:Main, Gabriel Conjecture}, Conjecture \ref{conj:Blockwise Conjecture E, extended}, and Conjecture \ref{conj:iBAWC} for $2$-blocks with abelian defect groups in any finite group. 

\begin{pro}
\label{prop}
Conjecture \ref{conj:Main, Gabriel Conjecture}, Conjecture \ref{conj:Blockwise Conjecture E, extended}, and and Conjecture \ref{conj:iBAWC} hold for every $2$-block with abelian defect groups of any finite group.
\end{pro}

\begin{proof}
Let $B$ be a $2$-block with abelian defect groups of a finite group $G$ and consider $G\unlhd A$. By Remark \ref{rmk:iBAWC implies Navarro} it suffices to show that Conjecture \ref{conj:iBAWC} holds for $B$ with respect to $G\unlhd A$. Furthermore, by Theorem \ref{thm:Reduction for iBAW with class of defect groups} it is no loss of generality to assume that $G$ is quasi-simple. In this case the structure of $B$ has been determined in \cite[Theorem 6.1]{EKKS14} and, using this description, the results of \cite{Zha-Zho21} and \cite{Hu-Zho} show that Conjecture \ref{conj:iBAWC} holds for $B$ as desired.
\end{proof}

\bibliographystyle{alpha}

\vspace{1cm}

(J. M. Mart\'inez) {\sc{Departament de Matem\`atiques, Universitat de Val\`encia, 46100 Burjassot, Val\`encia, Spain and Dipartamento di Matematica e Informatica 'Ulisse Dini', 50134 Firenze, Italy.}}

\textit{Email address:} \href{mailto:josep.m.martinez@uv.es}{josep.m.martinez@uv.es}

(N. Rizo) {\sc{Departament de Matem\`atiques, Universitat de Val\`encia, 46100 Burjassot, Val\`encia, Spain.}}

\textit{Email address:} \href{mailto:noelia.rizo@uv.es}{noelia.rizo@uv.es}

(D. Rossi) {\sc{Department of Mathematical Science, Loughborough University, LE11 3TU, United Kingdom}}

\textit{Email address:} \href{mailto:damiano.rossi.math@gmail.com}{damiano.rossi.math@gmail.com}

\end{document}